\documentclass{article}

\usepackage{rotating}
\usepackage{tikz-cd}
\usepackage[all]{xy}
\usepackage{amssymb}
\usepackage{amsmath}
\usepackage{amsthm}
\newtheorem{thm}{Theorem}
\newtheorem{prop}{Proposition}
\newtheorem{lem}{Lemma}
\newtheorem{rem}{Remark}
\newtheorem{cor}{Corollary}
\newtheorem{defi}{Definition}
\newtheorem{defiprop}{Definition-Proposition}

\def\Top{\mathop{\rm Top}\nolimits}
\def\RTop{\mathop{\rm RTop}\nolimits}
\def\Sch{\mathop{\rm Sch}\nolimits}

\def\dim{\mathop{\rm dim}\nolimits}

\def\SmVar{\mathop{\rm SmVar}\nolimits}
\def\Gr{\mathop{\rm Gr}\nolimits}

\def\PSmVar{\mathop{\rm PSmVar}\nolimits}
\def\Var{\mathop{\rm Var}\nolimits}
\def\Hom{\mathop{\rm Hom}\nolimits}
\def\Spec{\mathop{\rm Spec}\nolimits}
\def\Coh{\mathop{\rm Coh}\nolimits}

\def\QPVar{\mathop{\rm QPVar}\nolimits}
\def\AnSp{\mathop{\rm AnSp}\nolimits}
\def\CW{\mathop{\rm CW}\nolimits}

\def\PVar{\mathop{\rm PVar}\nolimits}

\def\sing{\mathop{\rm sing}\nolimits}

\def\Cone{\mathop{\rm Cone}\nolimits}
\def\ad{\mathop{\rm ad}\nolimits}

\def\log{\mathop{\rm log}\nolimits}
\def\Diff{\mathop{\rm Diff}\nolimits}

\def\PSh{\mathop{\rm PSh}\nolimits}
\def\AnSm{\mathop{\rm AnSm}\nolimits}

\def\pt{\mathop{\rm pt}\nolimits}
\def\DA{\mathop{\rm DA}\nolimits}

\def\Fun{\mathop{\rm Fun}\nolimits}

\def\Ab{\mathop{\rm Ab}\nolimits}

\def\coker{\mathop{\rm coker}\nolimits}
\def\Ext{\mathop{\rm Ext}\nolimits}

\def\Cat{\mathop{\rm Cat}\nolimits}
\def\RCat{\mathop{\rm RCat}\nolimits}

\def\an{\mathop{\rm an}\nolimits}

\def\Mod{\mathop{\rm Mod}\nolimits}

\def\Shv{\mathop{\rm Shv}\nolimits}

\def\Vect{\mathop{\rm Vect}\nolimits}
\def\HubSp{\mathop{\rm HubSp}\nolimits}
\def\GrAb{\mathop{\rm GrAb}\nolimits}
\def\logSch{\mathop{\rm logSch}\nolimits}
\def\logVar{\mathop{\rm logVar}\nolimits}
\def\Grd{\mathop{\rm Gr}\nolimits}
\def\AbCat{\mathop{\rm AbCat}\nolimits}
\def\PSch{\mathop{\rm PSch}\nolimits}
\def\FSch{\mathop{\rm FSch}\nolimits}

\def\dar[#1]{\ar@<2pt>[#1]\ar@<-2pt>[#1]}

\title{Complex vs etale Abel Jacobi map for higher Chow groups and 
algebraicity of the zero locus of etale normal functions}

\author{Johann Bouali}

\begin{document}

\maketitle

\begin{abstract}
We prove, using $p$-adic Hodge theory for open algebraic varieties,
that for a smooth projective variety over a field $k\subset\mathbb C$ of finite type over $\mathbb Q$,
the complex abel jacobi map vanishes if the etale abel jacobi map vanishes.
This implies that for a smooth projective morphism $f:X\to S$ 
of smooth complex algebraic varieties over a field $k\subset\mathbb C$ of finite type over $\mathbb Q$  
and $Z\in\mathcal Z^d(X,n)^{f,\partial=0}$ an algebraic cycle flat over $S$ whose cohomology class vanishes on fibers, 
the zero locus of the etale normal function associated to $Z$ is contained in the zero locus
of the complex normal function associated to $Z$. From the work of Saito or Charles, 
we deduce that the zero locus of the complex normal function associated to $Z$ is defined over 
$\bar k$ if the zero locus of the etale normal function associated to $Z$ is not empty. 
We also prove an algebraicity result for the zero locus of an etale normal function associated to an algebraic cycle
over a field $k$ of finite type over $\mathbb Q$.
By the way, we get for a smooth morphism $f:X\to S$ of smooth complex algebraic varieties 
over a subfield $k\subset\mathbb C$ of finite type over $\mathbb Q$, 
the locus of Hodge Tate classes inside the locus of Hodge classes of $f$.
\end{abstract}

\section{Introduction}

\hskip 1cm Let $X$ be a connected smooth projective variety over ${\mathbb C}$. 
The Abel-Jacobi map associates to a cycle $Z\in\mathcal Z^d(X)$ homologically equivalent to 0 of codimension $d$, 
a point $AJ(X)(Z)\in J^d(X)$ in the intermediate Jacobian of $X$. 
In family, if $f:X\to S$ is a smooth projective morphism of smooth connected complex varieties and 
$Z\in\mathcal Z^d(X)$ is a relative cycle on $S$ of codimension $d$ homologicaly trivial on fibers, 
the Abel-Jacobi map provides a holomorphic and horizontal section of the relative intermediate Jacobian : 
$\nu_Z:S\to J^d(X/S)$ called the associated normal function. 
By a Brosnan-Pearlstein theorem, the zero-locus $V(\nu_Z)\subset S$ and more generally the torsion locus
$V(\nu_Z)\subset V_{tors}(\nu_Z)\subset S$ of $\nu_Z$ is an algebraic subvariety (\cite{BPS}). 
If $X/S$ and $Z$ are defined over a subfield $k\subset\mathbb C$, 
we conjecture in the spirit of the Bloch-Beilinson conjectures, 
that the zero-locus of the normal function is also defined over $k$ (see \cite{Charles}). \\

\hskip 1cm On the other hand, if $k$ is of finite type over ${\mathbb Q}$, $X$ is a connected smooth projective variety over $k$, 
and $p$ is a prime number, we can define via the continuous étale cohomology with $\mathbb Z_p$ coefficients
an etale Abel-Jacobi map which associate to a cycle $Z\in\mathcal Z^d(X,n)$ homologically equivalent to 0 of codimension $d$,
an element $AJ^{et,p}(X)(Z)\in H^1(G,H^{2d-1-n}(X_{\bar k},\mathbb Z_p))$ of the first degree Galois cohomology of 
the absolute galois group $G:=Gal(\bar k/k)$ with value in the etale cohomology of $X_{\bar k}$ with $\mathbb Z_p$ coefficients. 
In family, we get, for $f:X\to S$ a smooth projective morphism of smooth connected varieties over $k$ 
and $Z\in\mathcal Z^d(X,n)$ a relative cycle on $S$ of codimension $d$ homologicaly trivial on fibers,
a normal function $\nu_Z^{et,p}$ associated to $Z\in\mathcal Z^d(X,n)$ 
and thus a zero-locus $V(\nu_Z^{et,p})\subset S$ and more generally
a torsion locus $V(\nu_Z^{et,p})\subset V_{tors}(\nu_Z^{et,p})\subset S$ of $\nu_Z^{et}$ 
which are subsets of closed points of $S$. 
For $\sigma:k\hookrightarrow\mathbb C$ an embedding, and $V\subset S$ a subset, we denote 
$V_{\mathbb C}:=\pi_{k/\mathbb C}(S)^{-1}(V)\subset S_{\mathbb C}$ 
where $\pi_{k/\mathbb C}(S):S_{\mathbb C}\to S$ is the projection. \\

\hskip 1cm Let $k$ of finite type over ${\mathbb Q}$ and 
$f:X\to S$ a smooth projective morphism of smooth connected varieties over $k$. Let $Z\in\mathcal Z^d(X)$ a relative cycle.
F. Charles then proves that for any embedding $\sigma:k\hookrightarrow\mathbb C$ 
\begin{itemize}
\item assuming that $R^{2d-1}f^{an}_*(\mathbb C)$ has no global sections 
then $V(\nu_Z^{et,p})_{\mathbb C}= V(\nu_Z)$ (\cite{Charles}) 
\item if $V(\nu_Z)(\bar k)$ is non-empty then it is defined over $k$ (\cite{Charles2}).\\
\end{itemize}

\hskip 1cm In this work, we show that, for a field $k$ of finite type over $\mathbb Q$, 
$X$ a connected smooth projective variety over $k$, 
$Z\in\mathcal Z^d(X,*)$ an (higher) algebraic cycle homologically equivalent to 0 of codimension $d$, 
and $p\in\mathbb N$ a prime number, 
if $AJ^{et,p}(X)(Z)=0$ then $AJ_{\sigma}(X)(Z):=AJ(X_{\mathbb C})(Z_{\mathbb C})=0$ 
for any embedding $\sigma:k\hookrightarrow\mathbb C$ (c.f. theorem \ref{main}).
This implies by definition that for $f:X\to S$ a smooth projective morphism of smooth connected varieties over 
a field $k$ of finite type over $\mathbb Q$, 
and $Z\in\mathcal Z^d(X,*)$ a relative cycle on $S$ of codimension $d$ homologicaly trivial on fibers,
\begin{equation*}
V_{tors}(\nu_Z^{et,p})_{\mathbb C}\subset V_{tors}(\nu_Z) 
\end{equation*}
for any embedding $\sigma:k\hookrightarrow\mathbb C$ without any assumption (c.f. corollary \ref{maincor}(i)). 
We deduce that if $V(\nu_Z^{et,p})$ is not empty, 
$V(\nu_Z)$ is defined over the algebraic closure $\bar k$ of $k$ (c.f. corollary \ref{maincor}(ii)). 

\noindent The proof of theorem \ref{main} uses $p$-adic Hodge theory for open varieties 
to relate de Rham cohomology and its Hodge filtration 
to $p$-adic étale coholomogy with its Galois action. 
Theorem \ref{main} follows indeed from the fact that 
by proposition \ref{keypropCp}(i), proposition \ref{LatticeLog} and proposition \ref{keypropC},
we have for a field $k$ of finite type over $\mathbb Q$, $U\in\SmVar(k)$, 
and embeddings $\sigma:k\hookrightarrow\mathbb C$, $\sigma_p:k\hookrightarrow\mathbb C_p$ for a prime $p\in\mathbb N$, 
for each $j,l\in\mathbb Z$, a canonical injective map
\begin{equation*}
H^j\iota_{p,ev}^{G,l}(U):H_{et}^j(U_{\bar k},\mathbb Z_p)(l)^G
\hookrightarrow F^lH^j(U^{an}_{\mathbb C},2i\pi\mathbb Q)\otimes_{\mathbb Q}\mathbb Q_p, \;
\alpha\mapsto H^j\iota_{p,ev}^{G,l}(U)(\alpha):=ev(U)(w(\alpha)),
\end{equation*}
which is by construction functorial in $U\in\SmVar(k)$ (see theorem \ref{CevGd}).  
More precisely, if the étale Abel-Jacobi image of a cycle $Z$ is zero, 
\begin{itemize}
\item there exist a Galois invariant class $\alpha$ in the étale cohomology of 
$((X\times\square^n)\backslash Z)_{\bar k}$ with non-zero boundary, 
where $\bar k$ is an algebraic closure of $k$, 
\item then, by $p$ adic Hodge theory for $(X\times\square^n)\backslash Z$, 
$\alpha$ define a logarithmic de Rham class $w(\alpha)_L$ laying inside the right degree of the Hodge filtration of 
$((X\times\square^n)\backslash Z)_{\hat k_{\sigma_p}}$ where $\hat k_{\sigma_p}\subset\mathbb C_p$
is the $p$-adic completion of $k$ with respect to an embedding $\sigma_p:k\hookrightarrow\mathbb C_p$
(c.f.proposition \ref{keypropCp}(i) and proposition \ref{LatticeLog}) 
\item taking the image of this class by the complex period map with respect to an embedding 
$\sigma':k'\hookrightarrow\mathbb C$ extending the given embedding $\sigma:k\hookrightarrow\mathbb C$ 
where $k'\subset\hat k_{\sigma_p}$ is a subfield over which $w(\alpha)_L$ is defined, 
we get a Betti cohomology class $ev((X\times\square^n)\backslash Z)(w(\alpha)_L)$ 
of $((X\times\square^n)\backslash Z)^{an}_{\mathbb C}$ with $2i\pi\mathbb Q$ coefficients 
(c.f. proposition \ref{keypropC}), 
\item this last class $(1/2i\pi)ev((X\times\square^n)\backslash Z)(w(\alpha)_L)$
induce a splitting of the localization exact sequence of mixed Hodge structures :
\begin{eqnarray*}
0\rightarrow H^{2d-1-n}(X^{an}_{\mathbb C},\mathbb Q)\xrightarrow{j^*} 
H^{2d-1}(((X\times\square^n)\backslash Z)^{an}_{\mathbb C},\mathbb Q)^{[Z]} 
\xrightarrow{\partial}H^{2d}_Z((X\times\square^n)^{an}_{\mathbb C},\mathbb Q)^{[Z]}=\mathbb Q^{Hdg}(d)\rightarrow 0,
\end{eqnarray*}
which means that the complex Abel-Jacobi image de $Z$ is zero.
\end{itemize}
By the way, since for $X\in\PSmVar(k)$, $\iota_{p,ev}^{G,d}(X)$ is compatible with cycle class maps,
we get in particular that Hodge conjecture implies Tate conjecture. 
In particular, we get Tate conjecture for divisors for smooth projective varieties
over fields of characteristic zero.

\noindent In section 6, we show that, for  
$f:X\to S$ a smooth projective morphism of connected smooth varieties over a field $k$ of finite type over $\mathbb Q$, 
$Z\in\mathcal Z^d(X,*)$ an (higher) algebraic cycle homologically equivalent to 0 of codimension $d$, and $p$ a prime number
for all expect finitely many prime numbers $p\in\mathbb N$ 
\begin{equation*}
V_{tors}(\nu_Z^{et,p})=T\cap S_{(0)}\subset S, 
\end{equation*}
where $T\subset S$ is the image of closed points of a constructible algebraic subset of
\begin{equation*}
E_{DR}^{2d-1}(((X\times\square^*)\backslash|Z|)_{\hat k_{\sigma_p}}/S_{\hat k_{\sigma_p}})
:=H^{2d-1}\int_f(O_{((X\times\square^*)\backslash|Z|)_{\hat k_{\sigma_p}}},F_b)\in\Vect_{fil}(S_{\hat k_{\sigma_p}})
\end{equation*}
by the projection
$p_S:E_{DR}^{2d-1}(((X\times\square^n)\backslash|Z|)_{\hat k_{\sigma_p}}/S_{\hat k_{\sigma_p}})\to S_{\hat k_{\sigma_p}}$,
where $\hat k_{\sigma_p}\subset\mathbb C_p$ 
is the completion of $k$ with respect to $\sigma_p$ (c.f. definition \ref{nuZkalgdef} and theorem \ref{nuZkalg}(i)).
We also give a local version for $V_{tors}(\nu_{Z,\sigma_p}^{et,p})\subset S$ (c.f. theorem \ref{nuZkalg}(ii)).
The proof use on the one hand proposition \ref{keypropCp}(ii), and on the other hand
De Yong alterations to get a stratification 
$S=\sqcup_{\alpha\in\Lambda} S^{\alpha}$ by locally closed subsets
and alterations $\pi^{\alpha}:(X_{S^{\alpha}}\times\square^n)^{\alpha}\to X_{S^{\alpha}}\times\square^n$ such that
\begin{equation*}
f\circ\pi^{\alpha}:((X\times\square^n)^{\alpha}_{\hat k_{\sigma_p}},\pi^{\alpha,-1}(|Z|))
\to S_{\hat k_{\sigma_p}}^{\alpha} 
\end{equation*}
is a semi-stable morphism. We then use the $p$ adic semi-stable comparison theorem 
for semi-stable morphisms $f':(X',D')\to S'$, with $S',X',D'\in\SmVar(\hat k_{\sigma_p})$, that is satisfying
\begin{itemize}
\item $f':X'\to S$ is smooth projective, $D'\subset X'$ is a normal crossing divisor,
\item for all $s\in S'$, $D'_s\subset X'_s$ is a normal crossing divisor and
$(X'_s,D'_s)$ has integral model with semi-stable reduction, 
\end{itemize}
(or more generally for log smooth morphism of schemes over $(O_{\hat k_{\sigma_p}},N_{\mathcal O})$),
which gives, for each $j\in\mathbb Z$ a canonical filtered isomorphism
\begin{equation*}
H^jf'_*\alpha(U'):R^jf'_*\mathbb Z_{p,U_{\mathbb C_p}^{'et}}\otimes_{\mathbb Z_p}O\mathbb B_{st,S'}
\xrightarrow{\sim}R^jf'_{*Hdg}(O_{U'},F_b)\otimes_{O_{S'}}O\mathbb B_{st,S'}.
\end{equation*}
which is for each $s\in S'$ compatible with the action of the Galois group $Gal(\mathbb C_p/k(s))$,
the Frobenius and the monodromy:
for all but finitely many primes $p$, $\hat k_{\sigma_p}(s)$ is unramified 
for all embeddings $\sigma_p:k\hookrightarrow\mathbb C_p$ and all $s\in S'$, so that we get a Frobenius
action on $R^jf'_{*Hdg}(O_{U'},F_b)_s=H^j_{DR}(U'_s)$ for all $s\in S'$. \\

This also give (see theorem \ref{HLk}) together with proposition \ref{keypropCp}(ii) and proposition \ref{keypropC}, 
for $f:X\to S$ a smooth morphism, with $S,X$ smooth over a subfield 
$\sigma:k\hookrightarrow\mathbb C$ of finite type over $\mathbb Q$, and $p$ any but finitely many prime numbers
and $\sigma_p:k\hookrightarrow\mathbb C_p$ an embedding, 
the locus of Hodge-Tate classes 
$(H^jRf_{*\mathbb Q_{p,X_{\bar k}^{et}}}(d))^G\subset H^jRf_{*\mathbb Q_{p,X_{\bar k}^{et}}}$, 
where $G=Gal(\bar k/k)$ and $\bar k\subset\mathbb C$ is the algebraic closure of $k$ in $\mathbb C$,
\begin{eqnarray*}
\iota_{ev}^{G,d}(X/S):\pi_{k/\mathbb C}(S)^*(H^jRf_{*\mathbb Q_{p,X_{\bar k}^{et}}}(d))^G\xrightarrow{\sim} \\
<F^dE^j_{DR}(X/S)\cap(\sqcup_{\alpha\in\Lambda}
((E^j_{DR}(X_{\hat k_{\sigma_p}}^{\alpha}/S_{\hat k_{\sigma_p}}^{\alpha})
\otimes_{O_{S_{\hat k_{\sigma_p}}}}O\mathbb B_{st,S_{\hat k_{\sigma_p}}^{\alpha}})^{\phi,N}))>_{\mathbb Q_p} \\
\hookrightarrow 
(F^dE_{DR}(X_{\mathbb C}/S_{\mathbb C})\cap R^jf_*\mathbb Q_{X^{an}_{\mathbb C}})\otimes_{\mathbb Q}\mathbb Q_p
=:HL_{j,d}(X_{\mathbb C}/S_{\mathbb C})\otimes_{\mathbb Q}\mathbb Q_p, \\ 
\alpha_s\mapsto ev(X_{k(s)})((1/2i\pi)w(\alpha_s)_{k(s)}), \, 
s'=\pi_{k/\mathbb C}^{-1}(s):k(s)\hookrightarrow\mathbb C, \, s\in S
\end{eqnarray*}
inside the locus of Hodge classes 
$HL_{j,d,\sigma}(X/S):=HL_{j,d}(X_{\mathbb C}/S_{\mathbb C})\subset E^j_{DR}(X_{\mathbb C}/S_{\mathbb C})$, 
which is a countable union of algebraic subvarieties (\cite{BPS}), where 
\begin{itemize}
\item $f:X\xrightarrow{j}\bar X\xrightarrow{f}S$ is a compactification of $f$ where $\bar X\in\SmVar(k)$,
$j$ is an open embedding and $\bar f$ is proper,
\item $E^j_{DR}(X/S):=H^j\int_{\bar f}j_{*Hdg}(O_X,F_b)$ and 
$E^j_{DR}(X_{\mathbb C}/S_{\mathbb C}):=H^j\int_{\bar f}j_{*Hdg}(O_{X_{\mathbb C}},F_b)$
are the filtered algebraic holonomic $D$-modules over $S$ and $S_{\mathbb C}$ respectively,
\item $S=\sqcup_{\alpha\in\Lambda}S^{\alpha}$ is a stratification by locally closed algebraic subsets, 
$\Lambda$ being a finite set,
\item $\pi^{\alpha}:X^{\alpha}\to X_{S^{\alpha}}$ being alterations, in particular we get sub vector bundles
$\pi^{\alpha,*}:E^j_{DR}(X_{S^{\alpha}}/S^{\alpha}):=E^j_{DR}(X/S)_{|S^{\alpha}}
\hookrightarrow E^j_{DR}(X^{\alpha}/S^{\alpha})$,
\item $\pi_{k/\hat k_{\sigma_p}}(S^{\alpha})^*E^j_{DR}(X^{\alpha}/S^{\alpha})
\subset E^j_{DR}(X^{\alpha}_{\hat k_{\sigma_p}}/S^{\alpha}_{\hat k_{\sigma_p}})$ 
is the canonical subset of closed points, $\hat k_{\sigma_p}\subset\mathbb C_p$ being the $p$ adic completion of $k$
with respect to $\sigma_p$,
\item $\pi_{k/\mathbb C}(S):S_{\mathbb C}\to S$ and $\pi_{k/\hat k_{\sigma_p}}(S):S_{\hat k_{\sigma_p}}\to S$ 
are the projections.
\end{itemize} 

I am grateful for professor F.Mokrane for help and support during this work as well as O.Wittenberg
for mentioning me an article of Jannsen on $l$-adic cohomology.

\section{Preliminaries and Notations}

\begin{itemize}

\item Denote by $\Top$ the category of topological spaces and $\RTop$ the category of ringed spaces.
\item Denote by $\Cat$ the category of small categories and $\RCat$ the category of ringed topos.
\item For $\mathcal S\in\Cat$ and $X\in\mathcal S$, we denote $\mathcal S/X\in\Cat$ the category whose
objects are $Y/X:=(Y,f)$ with $Y\in\mathcal S$ and $f:Y\to X$ is a morphism in $\mathcal S$, and whose morphisms
$\Hom((Y',f'),(Y,f))$ consists of $g:Y'\to Y$ in $\mathcal S$ such that $f\circ g=f'$.
\item For $\mathcal S\in\Cat$ denote $\Grd\mathcal S:=\Fun(\mathbb Z,\mathcal S)$ is the category of graded objects.
\item Denote by $\Ab$ the category of abelian groups.
For $R$ a ring denote by $\Mod(R)$ the category of (left) $R$ modules. 
We have then the forgetful functor $o_R:\Mod(R)\to\Ab$.
\item Denote by $\AbCat$ the category of small abelian categories.

\item For $(\mathcal S,O_S)\in\RCat$ a ringed topos, we denote by 
\begin{itemize}
\item $\PSh(\mathcal S)$ the category of presheaves of $O_S$ modules on $\mathcal S$ and
$\PSh_{O_S}(\mathcal S)$ the category of presheaves of $O_S$ modules on $\mathcal S$, 
whose objects are $\PSh_{O_S}(\mathcal S)^0:=\left\{(M,m),M\in\PSh(\mathcal S),m:M\otimes O_S\to M\right\}$,
together with the forgetful functor $o:\PSh(\mathcal S)\to \PSh_{O_S}(\mathcal S)$,
\item $C(\mathcal S)=C(\PSh(\mathcal S))$ and $C_{O_S}(\mathcal S)=C(\PSh_{O_S}(\mathcal S))$ 
the big abelian category of complexes of presheaves of $O_S$ modules on $\mathcal S$,
\item $C_{O_S(2)fil}(\mathcal S):=C_{(2)fil}(\PSh_{O_S}(\mathcal S))\subset C(\PSh_{O_S}(\mathcal S),F,W)$,
the big abelian category of (bi)filtered complexes of presheaves of $O_S$ modules on $\mathcal S$ 
such that the filtration is biregular and $\PSh_{O_S(2)fil}(\mathcal S)=(\PSh_{O_S}(\mathcal S),F,W)$.
\end{itemize}

\item Let $(\mathcal S,O_S)\in\RCat$ a ringed topos with topology $\tau$. For $F\in C_{O_S}(\mathcal S)$,
we denote by $k:F\to E_{\tau}(F)$ the canonical flasque resolution in $C_{O_S}(\mathcal S)$ (see \cite{B5}).
In particular for $X\in\mathcal S$, $H^*(X,E_{\tau}(F))\xrightarrow{\sim}\mathbb H_{\tau}^*(X,F)$.

\item For $f:\mathcal S'\to\mathcal S$ a morphism with $\mathcal S,\mathcal S'\in\RCat$,
endowed with topology $\tau$ and $\tau'$ respectively, we denote for $F\in C_{O_S}(\mathcal S)$ and each $j\in\mathbb Z$,
\begin{itemize}
\item $f^*:=H^j\Gamma(\mathcal S,k\circ\ad(f^*,f_*)(F)):\mathbb H^j(\mathcal S,F)\to\mathbb H^j(\mathcal S',f^*F)$,
\item $f^*:=H^j\Gamma(\mathcal S,k\circ\ad(f^{*mod},f_*)(F)):\mathbb H^j(\mathcal S,F)\to\mathbb H^j(\mathcal S',f^{*mod}F)$, 
\end{itemize}
the canonical maps.

\item For $\mathcal X\in\Cat$ a (pre)site and $p$ a prime number, we consider the full subcategory
\begin{equation*}
\PSh_{\mathbb Z_p}(\mathcal X)\subset\PSh(\mathbb N\times\mathcal X), \; \; 
F=(F_n)_{n\in\mathbb N}, \; p^nF_n=0, \; F_{n+1}/p^n\xrightarrow{\sim}F_n
\end{equation*}
$C_{\mathbb Z_p}(\mathcal X):=C(\PSh_{\mathbb Z_p}(\mathcal X))\subset C(\mathbb N\times\mathcal X)$ and
\begin{equation*}
\mathbb Z_p:=\mathbb Z_{p,\mathcal X}:=((\mathbb Z/p^*\mathbb Z)_{\mathcal X})\in\PSh_{\mathbb Z_p}(\mathcal X)
\end{equation*}
the diagram of constant presheaves on $\mathcal X$.

\item Let $f:\mathcal X'\to\mathcal X$ a morphism of (pre)site with $\mathcal X,\mathcal X'\in\Cat$.
We will consider for $F=(F_n)_{n\in\mathbb N}\in C_{\mathbb Z_p}(\mathcal X)$ the canonical map in $C(\mathcal X')$
\begin{equation*}
T(f^*,\varprojlim)(F):f^*\varprojlim_{n\in\mathbb N}F_n\to \varprojlim_{n\in\mathbb N}f^*F_n
\end{equation*}
Recall that filtered colimits do NOT commute with infinite limits in general. In particular,
for $f:\mathcal X'\to\mathcal X$ a morphism of (pre)site and 
$F=(F_n)_{n\in\mathbb N}\in\PSh_{\mathbb Z_p}(\mathcal X)$, 
$\varprojlim_{n\in\mathbb N}f^*F_n$ is NOT isomorphic to $f^*\varprojlim_{n\in\mathbb N}F_n$ 
in $\PSh(\mathcal X')$ in general.

\item Denote by $\Sch\subset\RTop$ the subcategory of schemes (the morphisms are the morphisms of locally ringed spaces).
We denote by $\PSch\subset\Sch$ the full subcategory of proper schemes.
For $X\in\Sch$, we denote by 
\begin{itemize}
\item $\Sch^{ft}/X\subset\Sch/X$ the full subcategory consisting of objects
$X'/X=(X',f)\in\Sch/X$ such that $f:X'\to X$ is an morphism of finite type
\item $X^{et}\subset\Sch^{ft}/X$ the full subcategory consisting of objects
$U/X=(X,h)\in\Sch/X$ such that $h:U\to X$ is an etale morphism.
\item $X^{sm}\subset\Sch^{ft}/X$ the full subcategory consisting of objects
$U/X=(X,h)\in\Sch/X$ such that $h:U\to X$ is a smooth morphism.
\end{itemize}
For a field $k$, we consider $\Sch/k:=\Sch/\Spec k$ the category of schemes over $\Spec k$. We then denote by
\begin{itemize}
\item $\Var(k)=\Sch^{ft}/k\subset\Sch/k$ the full subcategory consisting of algebraic varieties over $k$, 
i.e. schemes of finite type over $k$,
\item $\PVar(k)\subset\QPVar(k)\subset\Var(k)$ the full subcategories consisting of quasi-projective varieties and projective varieties respectively, 
\item $\PSmVar(k)\subset\SmVar(k)\subset\Var(k)$ the full subcategories consisting of smooth varieties and smooth projective varieties respectively.
\end{itemize}
For a morphism of field $\sigma:k\hookrightarrow K$, we have the extention of scalar functor
\begin{eqnarray*}
\otimes_kK:\Sch/k\to\Sch/K, \; X\mapsto X_K:=X_{K,\sigma}:=X\otimes_kK, \; (f:X'\to X)\mapsto (f_K:=f\otimes I:X'_K\to X_K).
\end{eqnarray*}
which is left ajoint to the restriction of scalar 
\begin{eqnarray*}
Res_{k/K}:\Sch/K\to\Sch/k, \; X=(X,a_X)\mapsto X=(X,\sigma\circ a_X), \; (f:X'\to X)\mapsto (f:X'\to X)
\end{eqnarray*}
The adjonction maps are 
\begin{itemize}
\item for $X\in\Sch/k$, the projection $\pi_{k/K}(X):X_K\to X$ in $\Sch/k$,
for $X=\cup_iX_i$ an affine open cover with $X_i=\Spec(A_i)$ we have by definition $\pi_{k/K}(X_i)=n_{k/K}(A_i)$,
\item for $X\in\Sch/K$, $I\times\Delta_K:X\hookrightarrow X_K=X\times_KK\otimes_kK$ in $\Sch/K$,
where $\Delta_K:K\otimes_kK\to K$ is the diagonal which is given by for $x,y\in K$, $\Delta_K(x,y)=x-y$.
\end{itemize}
The extention of scalar functor restrict to a functor
\begin{eqnarray*}
\otimes_kK:\Var(k)\to\Var(K), \; X\mapsto X_K:=X_{K,\sigma}:=X\otimes_kK, \; (f:X'\to X)\mapsto (f_K:=f\otimes I:X'_K\to X_K).
\end{eqnarray*}
and for $X\in\Var(k)$ we have $\pi_{k/K}(X):X_K\to X$ the projection in $\Sch/k$.
An algebraic variety $X\in\Var(K)$ is said to be defined over $k$ if there exists $X_0\in\Var(k)$
such that $X\simeq X_0\otimes_kK$ in $\Var(K)$. 
For $X=(X,a_X)\in\Var(k)$, we have $\Sch^{ft}/X=\Var(k)/X$ since for $f:X'\to X$ a morphism of schemes of finite type,
$(X',a_X\circ f)\in\Var(k)$ is the unique structure of variety over $k$ of $X'\in\Sch$ such that $f$ becomes a morphism 
of algebraic varieties over $k$, in particular we have
\begin{itemize}
\item $X^{et}\subset\Sch^{ft}/X=\Var(k)/X$,
\item $X^{sm}\subset\Sch^{ft}/X=\Var(k)/X$.
\end{itemize}
A morphism $f:X'\to X$ with $X,X'\in\Var(K)$ is said to be defined over $k$ if 
$X\simeq X_0\otimes_kK$ and $X'\simeq X'_0\otimes_kK$ are defined over $k$ and 
$\Gamma_f=\Gamma_{f_0}\otimes_kK\subset X'\times X$ is defined over $k$, so that $f_0\otimes_kK=f$ with $f_0:X'_0\to X_0$.

\item For $X\in\Sch$ and $s\in\mathbb N$, we denote by $X_{(s)}\subset X$ its points of dimension $s$, 
in particular $X_{(0)}\subset X$ are the closed points of $X$.

\item For $X\in\Sch$ and $k$ a field we denote by $X(k):=\Hom_{Sch}(\Spec k,X)$ the $k$ points of $X$.
We get $X(k)_{in}\subset X$ the image of the $k$-points of $X$.
For $k\subset k'$ a subfield, $\mathbb A^N_{k'}(k)_{in}=k^N\subset k^{'N}\subset\mathbb A^N_{k'}$ and 
$\mathbb A^N_k(k')_{in}=\pi_{k/k'}(\mathbb A^N_k)(k^{'N})\subset\mathbb A^N_k$. 

\item For $X\in\Sch$, we denote $X^{pet}\subset\Sch/X$ the pro etale site (see \cite{BSch})
which is the full subcategory of $\Sch/X$ whose object consists of weakly etale maps $U\to X$ (that is flat maps
$U\to X$ such that $\Delta_U:U\to U\times_XU$ is also flat) and whose topology is generated by fpqc covers.
We then have the canonical morphism of site
\begin{equation*}
\nu_X:X^{pet}\to X, (U\to X)\mapsto (U\to X)
\end{equation*}  
For $F\in C(X^{et})$, 
\begin{equation*}
\ad(\nu_X^*,R\nu_{X*})(F):F\to R\nu_{X*}\nu_X^*F 
\end{equation*}
is an isomorphism in $D(X^{et})$, in particular, for each $n\in\mathbb Z$
\begin{equation*}
\nu_X^*:=H^n\Gamma(X,k):\mathbb H^n_{et}(X,F)\xrightarrow{\sim}\mathbb H^n_{pet}(X,\nu_X^*F) 
\end{equation*}
are isomorphisms, where
\begin{equation*}
k:=k\circ\ad(\nu_X^*,\nu_{X*})(E_{et}(F)):E_{et}(F)\to\nu_{X*}E_{pet}(\nu_X^*F)
\end{equation*}
is the canonical map in $C(X^{et})$ which is a quasi-isomorphism.
For $X\in\Sch$, we denote
\begin{itemize}
\item $\underline{\mathbb Z_p}_X:=\varprojlim_{n\in\mathbb N}\nu_X^*(\mathbb Z/p^n\mathbb Z)_{X^{et}}\in\PSh(X^{pet})$ 
the constant presheaf on $\mathcal X$,
\item $l_{p,\mathcal X}:=(p(*)):\underline{\mathbb Z_p}_{X}\to\nu_X^*(\mathbb Z/p\mathbb Z)_{X^{et}}$ 
the projection map in $\PSh(\mathbb N\times\mathcal X^{pet})$.
\end{itemize}

\item Let $k$ a field of characteristic zero and $k_0\subset k$ a subfield. 
We say that $k$ is of finite type over $k_0$ if $k$ is generated as a field by $k_0$ and a finite set 
$\left\{\alpha_1,\ldots,\alpha_r\right\}\subset k$ of elements of $k$, that is $k=k_0(\alpha_1,\ldots,\alpha_r)$.
If $k$ is of finite type over $k_0$ then it is of finite transcendence degree $d\in\mathbb N$ over $k_0$
and $k=k_0(\alpha_1,\ldots,\alpha_d)(\alpha_{d+1})$ with $\left\{\alpha_1,\ldots,\alpha_{d+1}\right\}\subset k$ 
such that $k_0(\alpha_1,\ldots,\alpha_d)=k_0(x_1,\ldots,x_d)$ and 
$\alpha_{d+1}$ is an algebraic element of $k$ over $k_0(\alpha_1,\ldots,\alpha_d)$.
Note that if $k$ is of finite type over $k_0$ then it is NOT algebraicly closed.
We denote $\bar k$ the algebraic closure of $k$. Then $\bar k$ is also transcendence degree $d$ over $k$.

\item Let $C$ a field of characteristic zero. Let $X\in\Var(C)$. Then there exist a subfield $k\subset C$
of finite type over $\mathbb Q$ such that $X$ is defined over $k$ that is $X\simeq X_0\otimes_k\mathbb C$
with $X_0\in\Var(k)$.

\item Let $k\subset\bar{\mathbb Q}$ be a number field, i.e. a finite extension of $\mathbb Q$.
There exists a finite set of prime number $\delta(k)$ such that for all prime number $p\in\mathbb N\backslash\delta(k)$ 
we have for all each embedding $\sigma_p:k\hookrightarrow\mathbb C_p$ $\mathbb Q_p\cap k=\mathbb Q$.

\item Let $X\in\Var(k)$. Considering its De Rham complex $\Omega_X^{\bullet}:=DR(X)(O_X):=\Omega_{X/k}^{\bullet}$,
we have for $j\in\mathbb Z$ its De Rham cohomology $H^j_{DR}(X):=\mathbb H^j(X,\Omega^{\bullet}_X)$.
If $X\in\SmVar(k)$, then $H^j_{DR}(X)=\mathbb H_{et}^j(X,\Omega^{\bullet}_X)$
since $\Omega^{\bullet}\in C(\SmVar(k))$ is $\mathbb A^1$ local and admits transfert (see \cite{B5}).

\item Let $X\in\Var(k)$. Let $X=\cup_{i=1}^sX_i$ an open affine cover. 
For $I\subset[1,\ldots,s]$, we denote $X_I:=\cap_{i\in I}X_i$. 
We get $X_{\bullet}\in\Fun(P([1,\ldots,s]),\Var(k))$.
Since quasi-coherent sheaves on affine noetherian schemes are acyclic, we have
for each $j\in\mathbb Z$, $H^j_{DR}(X)=\Gamma(X_{\bullet},\Omega^{\bullet}_{X^{\bullet}})$.

\item For $K$ a field which is complete with respect to a $p$-adic norm,
we consider $O_K\subset K$ the subring of $K$ consisting of integral elements, that is $x\in K$ such that $|x|\leq 1$.
\begin{itemize}
\item For $X\in\PVar(K)$, we will consider $X^{\mathcal O}\in\PSch/O_K$ a (non canonical) integral model of $X$, 
i.e. satisfying $X^{\mathcal O}\otimes_{O_K}K=X$
\item For $X\in\Var(K)$, we will consider $X^{\mathcal O}\in\Sch/O_K$ a (non canonical) integral model of $X$, 
i.e. $X^{\mathcal O}=\bar X^{\mathcal O}\backslash Z^{\mathcal O}$ for 
$\bar X\in\PVar(K)$ a compactification of $X$, $Z:=\bar X\backslash X$,
where $\bar X^{\mathcal O}\in\PSch/O_K$ is an integral model of $\bar X$ and 
$Z^{\mathcal O}\in\PSch/O_K$ is an integral model of $Z$. 
\end{itemize}
For $X\in\Var(K)$, we will consider $X^{\mathcal O}\in\Sch/O_K$ a (non canonical) integral model of $X$,
we then have the commutative diagram of sites
\begin{equation*}
\xymatrix{X^{pet}\ar[r]^{r}\ar[d]^{\nu_X} & X^{\mathcal O,pet}\ar[d]^{\nu_{X^{\mathcal O}}} \\
X^{et}\ar[r]^{r} & X^{\mathcal O,et}}, \; \; 
r(t:U\to X^{\mathcal O}):=(t\otimes_{O_K}K:U\otimes_{O_K}K\to X^{\mathcal O}\otimes_{O_K}K=X)
\end{equation*}

\item For $X\in\Sch$ and $p$ a prime number, 
we denote by $c:\hat X^{(p)}\to X$ the morphism in $\RTop$ which is the completion along the ideal generated by $p$.

\item Let $K$ a field which is complete with respect to a $p$-adic norm and $X\in\PVar(K)$ projective.
For $X^{\mathcal O}\in\PSch/O_K$ an integral model of $X$, i.e. satisfying $X^{\mathcal O}\otimes_{O_K}K=X$,
we consider the morphism in $\RTop$ 
\begin{equation*}
c:=(c\otimes I):\hat X^{(p)}:=\hat X^{\mathcal O,(p)}\otimes_{O_K}K\to X^{\mathcal O}\otimes_{O_K}K=X. 
\end{equation*}
We have then, by GAGA (c.f. EGA 3), for $F\in\Coh_{O_X}(X)$ a coherent sheaf of $O_X$ module, 
$c^*:H^k(X^{\mathcal O},F)\xrightarrow{\sim}H^k(\hat X^{\mathcal O,(p)},c^{*mod}F)$ for all $k\in\mathbb Z$,
in particular 
\begin{equation*}
c^*:\mathbb H^k(X,\Omega_X^{\bullet\geq l})\xrightarrow{\sim}
\mathbb H^k(\hat X^{(p)},\Omega_{\hat X^{(p)}}^{\bullet\geq l}) 
\end{equation*}
for all $k,l\in\mathbb Z$.

\item Denote by $\AnSp(\mathbb C)\subset\RTop$ the full subcategory of analytic spaces over $\mathbb C$,
and by $\AnSm(\mathbb C)\subset\AnSp(\mathbb C)$ the full subcategory of smooth analytic spaces (i.e. complex analytic manifold).
For $X\in\AnSp(\mathbb C)$, we denote by $X^{et}\subset\AnSp(\mathbb C)/X$ 
the full subcategory consisting of objects $U/X=(X,h)\in\AnSp(\mathbb C)/X$ such that $h:U\to X$ is an etale morphism.
By the Weirstrass preparation theorem (or the implicit function theorem if $U$ and $X$ are smooth),
a morphism $r:U\to X$ with $U,X\in\AnSp(\mathbb C)$ is etale if and only if it is an isomorphism local.
Hence for $X\in\AnSp(\mathbb C)$, the morphism of site $\pi_X:X^{et}\to X$ is an isomorphism of site.

\item Denote by $\CW\subset\Top$ the full subcategory of $CW$ complexes.
Denote by $\Diff(\mathbb R)\subset\RTop$ the full subcategory of differentiable (real) manifold.  

\item Let $k\subset\mathbb C$ a subfield. For $X\in\Var(k)$, we denote by 
\begin{equation*}
\an_X:X^{an}_{\mathbb C}\xrightarrow{\sim}X^{an,et}_{\mathbb C}
\xrightarrow{\an_{X^{et}_{\mathbb C}}}X^{et}_{\mathbb C}\xrightarrow{\pi_{k/\mathbb C}(X^{et})}X^{et},
\end{equation*}
the morphism of site given by the analytical functor.

\item Let $k\subset\mathbb C$ a subfield. For $X\in\Var(k)$, we denote by
\begin{equation*}
\alpha(X):\mathbb C_{X_{\mathbb C}^{an}}\hookrightarrow\Omega^{\bullet}_{X^{an}_{\mathbb C}} 
\end{equation*}
the canonical embedding in $C(X_{\mathbb C}^{an})$. It induces the embedding in $C(X_{\mathbb C}^{an})$
\begin{equation*}
\beta(X):2i\pi\mathbb Z_{X_{\mathbb C}^{an}}
\xrightarrow{\iota_{2i\pi\mathbb Z/\mathbb C,X_{\mathbb C}^{an}}}
\mathbb C_{X_{\mathbb C}^{an}}\xrightarrow{\alpha(X)}\Omega^{\bullet}_{X^{an}_{\mathbb C}} 
\end{equation*}
For $X\in\SmVar(k)$, $\alpha(X)$ is a quasi-isomorphism by the holomorphic Poincare lemma.

\item For $X\in\Sch$ noetherian irreducible and $d\in\mathbb N$, 
we denote by $\mathcal Z^d(X)$ the group of algebraic cycles of codimension $d$,
which is the free abelian group generated by irreducible closed subsets of codimension $d$.
\begin{itemize}
\item For $X\in\Sch$ noetherian irreducible and $d\in\mathbb N$, we denote by 
$\mathcal Z^d(X,\bullet)\subset\mathcal Z^d(X\times\square^{\bullet})$ the Bloch cycle complex
which is the subcomplex which consists of algebraic cycles which intersect $\square^*$ properly.
\item For $X\in\Var(k)$ irreducible and $d,n\in\mathbb N$, we denote by 
$\mathcal Z^d(X,n)^{\partial=0}_{hom}\subset\mathcal Z^d(X,\bullet)^{\partial=0}$ the 
the subabelian group consisting of algebraic cycles which are homologicaly trivial.
\item For $f:X\to S$ a dominant morphism with $S,X\in\Sch$ noetherian irreducible and $d\in\mathbb N$, we denote by 
$\mathcal Z^d(X,\bullet)^f\subset\mathcal Z^d(X,\bullet)$ 
the subcomplex consisting of algebraic cycles which are flat over $S$.
\item For $f:X\to S$ a dominant morphism with $S,X\in\Var(k)$ irreducible and $d,n\in\mathbb N$ we denote by
$\mathcal Z^d(X,n)_{fhom}^{f,\partial=0}\subset\mathcal Z^d(X,n)^{f,\partial=0}$ 
the subabelian group consisting of algebraic cycles which are flat over $S$ and homological trivial on fibers.
\end{itemize}

\item For $X\in\Var(k)$ and $Z\subset X$ a closed subset, denoting $j:X\backslash Z\hookrightarrow X$ the open complementary, 
we will consider
\begin{equation*}
\Gamma^{\vee}_Z\mathbb Z_X:=\Cone(\ad(j_{\sharp},j^*)(\mathbb Z_X):
\mathbb Z_X\hookrightarrow\mathbb Z_X)\in C(\Var(k)^{sm}/X)
\end{equation*}
and denote for short $\gamma^{\vee}_Z:=\gamma^{\vee}_Z(\mathbb Z_X):\mathbb Z_X\to\Gamma^{\vee}_Z\mathbb Z_X$
the canonical map in $C(\Var(k)^{sm}/X)$. Denote $a_X:X\to\Spec k$ the structural map.
For $X\in\Var(k)$ and $Z\subset X$ a closed subset, we have the motive of $X$ with support in $Z$ defined as 
\begin{equation*}
M_Z(X):=a_{X!}\Gamma^{\vee}_Za_X^!\mathbb Z\in\DA(k).
\end{equation*}
If $X\in\SmVar(k)$, we will also consider
\begin{equation*}
a_{X\sharp}\Gamma^{\vee}_Z\mathbb Z_X:=\Cone(a_{X\sharp}\circ\ad(j_{\sharp},j^*)(\mathbb Z_X):
\mathbb Z(U)\hookrightarrow\mathbb Z(X))=:\mathbb Z(X,X\backslash Z)\in C(\SmVar(k)).
\end{equation*}
Then for $X\in\SmVar(k)$ and $Z\subset X$ a closed subset
\begin{equation*}
M_Z(X):=a_{X!}\Gamma^{\vee}_Za_X^!\mathbb Z
=a_{X\sharp}\Gamma^{\vee}_Z\mathbb Z_X=:\mathbb Z(X,X\backslash Z)\in\DA(k).
\end{equation*}

\item We denote $\mathbb I^n:=[0,1]^n\in\Diff(\mathbb R)$ (with boundary).
For $X\in\Top$ and $R$ a ring, we consider its singular cochain complex
\begin{equation*}
C^*_{\sing}(X,R):=(\mathbb Z\Hom_{\Top}(\mathbb I^*,X)^{\vee})\otimes R 
\end{equation*}
and for $l\in\mathbb Z$ its singular cohomology $H^l_{\sing}(X,R):=H^nC^*_{\sing}(X,R)$.
In particular, we get by functoriality the complex 
\begin{equation*}
C^*_{X,R\sing}\in C_R(X), \; (U\subset X)\mapsto C^*_{\sing}(U,R)
\end{equation*}
We will consider the canonical embedding 
\begin{equation*}
C^*\iota_{2i\pi\mathbb Z/\mathbb C}(X):C^*_{\sing}(X,2i\pi\mathbb Z)\hookrightarrow C^*_{\sing}(X,\mathbb C), \,
\alpha\mapsto\alpha\otimes 1
\end{equation*}
whose image consists of cochains $\alpha\in C^j_{\sing}(X,\mathbb C)$ such that $\alpha(\gamma)\in 2i\pi\mathbb Z$
for all $\gamma\in\mathbb Z\Hom_{\Top}(\mathbb I^*,X)$.
We get by functoriality the embedding in $C(X)$
\begin{eqnarray*}
C^*\iota_{2i\pi\mathbb Z/\mathbb C,X}:C^*_{X,2i\pi\mathbb Z,\sing}\hookrightarrow C^*_{X,\mathbb C,\sing}, \\
(U\subset X)\mapsto 
(C^*\iota_{2i\pi\mathbb Z/\mathbb C}(U):C^*_{\sing}(U,2i\pi\mathbb Z)\hookrightarrow C^*_{\sing}(U,\mathbb C))
\end{eqnarray*}
We recall we have 
\begin{itemize}
\item For $X\in\Top$ locally contractile, e.g. $X\in\CW$, and $R$ a ring, the inclusion in $C_R(X)$
$c_X:R_X\to C^*_{X,R\sing}$ is by definition an equivalence top local and that we get 
by the small chain theorem, for all $l\in\mathbb Z$, an isomorphism 
$H^lc_X:H^l(X,R_X)\xrightarrow{\sim}H^l_{\sing}(X,R)$.
\item For $X\in\Diff(\mathbb R)$, the restriction map 
\begin{equation*}
r_X:\mathbb Z\Hom_{\Diff(\mathbb R)}(\mathbb I^*,X)^{\vee}\to C^*_{\sing}(X,R), \; 
w\mapsto w:(\phi\mapsto w(\phi))
\end{equation*}
is a quasi-isomorphism by Whitney approximation theorem.
\end{itemize}

\item Let $X\in\AnSm(\mathbb C)$. Let $X=\cup_{i=1}^r\mathbb D_i$ an open cover with $\mathbb D_i\simeq D(0,1)^d$. 
Since a convex open subset of $\mathbb C^d$ is biholomorphic to an open ball we have 
$\mathbb D_I:=\cap_{i\in I}\mathbb D_i\simeq D(0,1)^d$ (where $d$ is the dimension of a connected component of $X$).
We get $\mathbb D_{\bullet}\in\Fun(P([1,\ldots,r]),\AnSm(\mathbb C))$.

\item For $k$ a field, we denote by $\Vect(k)$ the category of vector spaces and 
$\Vect_{fil}(k)$ the category of filtered vector spaces.
Let $k\subset K$ a field extension of field of characteristic zero. 
\begin{itemize}
\item For $(V,F)\in\Vect_{fil}(k)$, we get a filtered $K$ vector space 
$(V\otimes_kK,F)\in\Vect_{fil}(K)$ by $F^j(V\otimes_kK):=(F^jV)\otimes_kK$.
In this case, we say that the filtration $F$ on $V\otimes_kK$ is defined over $k$.
\item For $(V',F)\in\Vect_{fil}(K)$ and $h:V\otimes_kK\xrightarrow{\sim}V'$ and isomorphism of $K$ vector space, we get
$(V,F_h)\in\Vect_{fil}(k)$ by $F_h^jV:=h^{-1}(F^jV')\cap V$ 
(considering the canonical embedding $n:V\hookrightarrow V\otimes_kK$, $n(v):=v\otimes 1$).
\item For $(V,F)\in\Vect_{fil}(k)$, we have $F^j(V\otimes_kK)\cap V=F^jV$.
\item For $(V',F)\in\Vect_{fil}(K)$ and $h:V\otimes_kK\xrightarrow{\sim}V'$ and isomorphism of $K$ vector space,
we have $h((F_h^jV)\otimes_kK)\subset F^jV'$. Of course this inclusion is NOT an equality in general.
The filtration $F$ on $V'$ is NOT defined over $k$ in general.
\end{itemize}

\item We also consider 
\begin{itemize}
\item $\Top_2$ the category whose objects are couples $(X,Y)$ with $X\in\Top$ and $Y\subset X$ a subset
and whose set of morphisms $\Hom((X',Y'),(X,Y))$ consists of $f:X'\to X$ continuous such that $Y'\subset f^{-1}(Y)$
(i.e. $f(Y')\subset Y$),
\item $\RTop_2$ the category whose objects are couples $(X,Y)$ with $X=(X,O_X)\in\RTop$ and $Y\subset X$ a subset
and whose set of morphisms $\Hom((X',Y'),(X,Y))$ consists of $f:X'\to X$ of ringed spaces such that $Y'\subset f^{-1}(Y)$,
\item $\Top^2$ the category whose objects are couples $(X,Z)$ with $X\in\Top$ and $Z\subset X$ a closed subset
and whose set of morphisms $\Hom((X',Z'),(X,Z))$ consists of $f:X'\to X$ continuous such that $f^{-1}(Z)\subset Z'$
(i.e. $f(X'\backslash Z')\subset X\backslash Z$),
in particular we have the canonical functor $\Top^2\to\Top_2$, $(X,Z)\mapsto (X,X\backslash Z)$,
\item $\RTop^2$ the category whose objects are couples $(X,Z)$ with $X=(X,O_X)\in\RTop$ and $Z\subset X$ a closed subset
and whose set of morphisms $\Hom((X',Z'),(X,Z))$ consists of $f:X'\to X$ of ringed spaces such that $f^{-1}(Z)\subset Z'$,
in particular we have the canonical functor $\RTop^2\to\RTop_2$, $(X,Z)\mapsto (X,X\backslash Z)$.
\end{itemize}
A (generalized) cohomology theory is in particular a functor $H^*:\Top_2\to\GrAb$, e.g singular cohomology 
\begin{equation*}
H^*_{\sing}:\Top^2\to\Gr\Ab,(X,Y)\mapsto H^*_{\sing}(X,Y,R).
\end{equation*}
where $R$ is a commutative ring.
It restrict to a functor $H^*:\Top^2\to\GrAb, (X,Z)\mapsto H^*_Z(X):=H^*(X,X\backslash Z)$. 

\item Denote $\Sch^2\subset\RTop^2$ the subcategory
whose objects are couples $(X,Z)$ with $X=(X,O_X)\in\Sch$ and $Z\subset X$ a closed subset
and whose set of morphisms $\Hom((X',Z'),(X,Z))$ consists of $f:X'\to X$ of locally ringed spaces
such that $f^{-1}(Z)\subset Z'$.

\item Let $k$ a field of characteristic zero. 
Denote $\SmVar^2(k)\subset\Var^2(k)\subset\Sch^2/k$ the full subcategories
whose objects are $(X,Z)$ with $X\in\Var(k)$, resp. $X\in\SmVar(k)$, and $Z\subset X$ is a closed subset,
and whose morphisms $\Hom((X',Z')\to(X,Z))$ consists of $f:X'\to X$ of schemes over $k$ 
such that $f^{-1}(Z)\subset Z'$.

\item Let $k$ a field of characteristic zero. Let 
\begin{equation*}
H^*:\SmVar^2(k)\to\Grd\AbCat, (X,Z)\mapsto H^*_Z(X) 
\end{equation*}
a mixed Weil cohomology theory in sense of \cite{CD} 
(e.g. (filtered) De Rham, etale or Betti cohomology, Hodge or $p$ adic realization). 
For $X\in\SmVar(k)$ and $Z\subset X$ a closed subset, we denote 
\begin{equation*}
H^*_Z(X)^0:=\ker(H^*_Z(X)\to H^*(X)).
\end{equation*}
For $X\in\SmVar(k)$ and $Z\in\mathcal Z^d(X,n)^{\partial=0}_{hom}$, 
we consider the subobject $H^{2d-1}(U)^{[Z]}\subset H^{2d-1}(U)$ where 
$j:U:=(X\times\square^n)\backslash|Z|\hookrightarrow X\times\square^n$ is the complementary open subset,
given by the pullback by $H^{2d}_Z(X\times\square^n)^{[Z]}:=[Z]\subset H^{2d}_Z(X\times\square^n)^0$
\begin{equation*}
\xymatrix{0\ar[r] & H^{2d-1}(X\times\square^n)=H^{2d-n-1}(X)\ar[r]^{j^*} & H^{2d-1}(U)\ar[r]^{\partial} &
H^{2d}_Z(X\times\square^n)^0\ar[r] & 0 \\
0\ar[r] & H^{2d-1}(X\times\square^n)=H^{2d-n-1}(X)\ar[r]^{j^*}\ar[u]^{=} & 
H^{2d-1}(U)^{[Z]}\ar[u]^{\subset}\ar[r]^{\partial} & H^{2d}_Z(X\times\square^n)^{[Z]}:=[Z]\ar[u]^{\subset}\ar[r] & 0}
\end{equation*}
of the first row exact sequence. In particular the second row is also an exact sequence.

\item We denote by $\log\Sch$ the category of log schemes whose objects are couples $(X,M):=(X,M,\alpha)$ 
where $X=(X,O_X)\in\Sch$, $M\in\Shv(X)$ is a sheaf of monoid and $\alpha:M\to O_X$ is a morphism of sheaves of monoid. 
In particular we have a canonical functor 
\begin{equation*}
\Sch^2\to\logSch, (X,Z)\mapsto (X,M_Z), M_Z:=(f\in O_X s.t. f_{|X\backslash Z}\in O_{X\backslash Z}^*)\subset O_X
\end{equation*}
Let $k$ a field of characteristic zero. 
We denote by $\log\Var(k)\subset\log\Sch/k$ the full subcategory of log varieties.

\item For $k$ a field, we denote by $\Vect(k)=\Mod(k)$ the category of vector spaces over $k$
and $C(k):=C(\Vect(k))$ the category of complexes of vector spaces over $k$,
by $\Vect_{fil}(k)$ the category of filtered vector spaces over $k$ and $C_{fil}(k):=C_{fil}(\Vect(k))$
the category of filtered complexes of vector spaces over $k$.

\item We assume all field have a transcendence basis at most countable, so that field of characteristic zero
admit embeddings in $\mathbb C$ and in $\mathbb C_p$ for each prime number $p\in\mathbb N$.

\end{itemize}

We have the followings facts :

\begin{itemize}
\item Let $k$ a field of characteristic zero. Denote $G:=Gal(\bar k/k)$ its absolute Galois group.
Then the functor
\begin{equation*}
\Gamma(\bar k)(-):\PSh(k^{et})\to\Mod(\bar k,G), \; F\mapsto\Gamma(\bar k,F)
\end{equation*}
is an equivalence of category whose inverse is
\begin{equation*}
G(-):\Mod(\bar k,G)\to\PSh(k^{et}), \; V\mapsto G(V):=V:=((k'/k)\mapsto V^{Aut(k'/k)}). 
\end{equation*}
In particular, for each $V\in\Mod(\bar k,G)$ and $j\in\mathbb Z$, we get an isomorphism
\begin{equation*}
H^jG(V):H^j(G,V)\xrightarrow{\sim}\Ext_G^j(\bar k,V).
\end{equation*}
\item Let $k\subset K$ a field extension. 
\begin{itemize}
\item Let $X\in\Var(k)$.
We have then the canonical isomorphism in $C_{Aut(K/k),fil}(X_K)$
\begin{equation*}
w(k/K):(\Omega^{\bullet}_X\otimes_kK,F_b)\xrightarrow{\sim}(\Omega^{\bullet}_{X_K},F_b) 
\end{equation*}
given by the universal property of derivation of a ring.
\item Let $X\in\SmVar(k)$. 
Let $\bar X\in\PSmVar(k)$ a smooth compactification of $X$ with $D:=\bar X\backslash X$ a normal crossing divisor.
We have then the canonical isomorphism in $C_{Aut(K/k),fil}(\bar X_K)$
\begin{equation*}
w(k/K):(\Omega^{\bullet}_{\bar X}(\log D)\otimes_kK,F_b)\xrightarrow{\sim}(\Omega^{\bullet}_{\bar X_K}(\log D),F_b) 
\end{equation*}
given by the preceeding point. In particular, we get for all $j,l\in\mathbb Z$,
\begin{itemize}
\item $F^lH^jw(k/K):F^lH^j_{DR}(X)\otimes_kK\xrightarrow{\sim}F^lH^j_{DR}(X_K)$,
\item $H^jw(k/K):H^j_{DR}(X)\xrightarrow{\sim}H^j_{DR}(X_K)^G$.
\end{itemize}
\end{itemize}
\item An affine scheme $U\in\Sch$ is said to be $w$-contractible if any faithfully flat weakly etale map $V\to U$, 
$V\in\Sch$, admits a section. We will use the facts that (see \cite{BSch}):
\begin{itemize}
\item Any scheme $X\in\Sch$ admits a pro-etale cover $(r_i:X_i\to X)_{i\in I}$ 
with for each $i\in I$, $X_i$ a $w$-contractile affine scheme and $r_i:X_i\to X$ a weakly etale map.
For $X\in\Var(k)$ with $k$ a field, we may assume $I$ finite since the topological space $X$ is then quasi-compact.
\item If $U\in\Sch$ is a $w$-contractible affine scheme, then for any sheaf $F\in\Shv(U^{pet})$, 
$H_{pet}^i(U,F)=0$ for $i\neq 0$ since $\Gamma(U,-)$ is an exact functor.
\end{itemize}
\end{itemize}

We introduce the logarithmic De Rham complexes

\begin{defi}\label{wlogdef}
\begin{itemize}
\item[(i)] Let $X=(X,O_X)\in\RCat$ a ringed topos, we have in $C(X)$ the subcomplex of presheaves of abelian groups
\begin{eqnarray*}
OL_X:\Omega^{\bullet}_{X,\log}\hookrightarrow\Omega_X^{\bullet}, \; 
\mbox{s.t. for} \; X^o\in X \; \mbox{and}\;  p\in\mathbb N, \, p\geq 1, \\
\Omega^p_{X,\log}(X^0):=
<df_{\alpha_1}/f_{\alpha_1}\wedge\cdots\wedge df_{\alpha_p}/f_{\alpha_p}, f_{\alpha_k}\in\Gamma(X^o,O_X)^*>
\subset\Omega^p_X(X^0),
\end{eqnarray*}
where $\Omega_X^{\bullet}:=DR(X)(O_X)\in C(X)$ is the De Rham complex and 
$\Gamma(X^o,O_X)^*\subset\Gamma(X^o,O_X)$ is the multiplicative subgroup 
consisting of invertible elements for the multiplication,
here $<,>$ stand for the sub-abelian group generated by. By definition, for $w\in\Omega^p_X(X^o)$, 
$w\in\Omega^p_{X,\log}(X^o)$ if and only if there exists $(n_i)_{1\leq i\leq s}\in\mathbb Z$ and 
$(f_{i,\alpha_k})_{1\leq i\leq s,1\leq k\leq p}\in\Gamma(X^o,O_X)^*$ such that
\begin{equation*}
w=\sum_{1\leq i\leq s}n_idf_{i,\alpha_1}/f_{i,\alpha_1}\wedge\cdots\wedge df_{i,\alpha_p}/f_{i,\alpha_p}
\in\Omega^p_X(X^o).
\end{equation*}
For $p=0$, we set $\Omega^0_{X,\log}:=\mathbb Z$.
Let $f:X'=(X',O_{X'})\to X=(X,O_X)$ a morphism with $X,X'\in\RCat$.
Consider the morphism $\Omega(f):\Omega_X^{\bullet}\to f_*\Omega_{X'}^{\bullet}$ in $C(X)$.
Then, $\Omega(f)(\Omega^{\bullet}_{X,\log})\subset f_*\Omega^{\bullet}_{X',\log}$.
\item[(ii)] For $k$ a field, we get from (i), for $X\in\Var(k)$, the embedding in $C(X)$
\begin{equation*}
OL_X:\Omega^{\bullet}_{X,\log}\hookrightarrow\Omega_X^{\bullet}:=\Omega^{\bullet}_{X/k},
\end{equation*}
such that, for $X^o\subset X$ an open subset and $w\in\Omega^p_X(X^o)$,
$w\in\Omega^p_{X,\log}(X^o)$ if and only if there exists $(n_i)_{1\leq i\leq s}\in\mathbb Z$ and 
$(f_{i,\alpha_k})_{1\leq i\leq s,1\leq k\leq p}\in\Gamma(X^o,O_X)^*$ such that
\begin{equation*}
w=\sum_{1\leq i\leq s}n_idf_{i,\alpha_1}/f_{i,\alpha_1}\wedge\cdots\wedge df_{i,\alpha_p}/f_{i,\alpha_p}
\in\Omega^p_X(X^o),
\end{equation*}
and for $p=0$, $\Omega^0_{X,\log}:=\mathbb Z$. Let $k$ a field. We get an embedding in $C(\Var(k))$
\begin{eqnarray*}
OL:\Omega^{\bullet}_{/k,\log}\hookrightarrow\Omega^{\bullet}_{/k}, \; 
\mbox{given by}, \, \mbox{for} \, X\in\Var(k), \\
OL(X):=OL_X:\Omega^{\bullet}_{/k,\log}(X):=\Gamma(X,\Omega^{\bullet}_{X,\log})
\hookrightarrow\Gamma(X,\Omega^{\bullet}_X)=:\Omega^{\bullet}_{/k}(X)
\end{eqnarray*}
and its restriction to $\SmVar(k)\subset\Var(k)$.
\item[(iii)]Let $K$ a field of characteristic zero which is complete for a $p$-adic norm. Let $X\in\Var(K)$. 
Let $X^{\mathcal O}\in\Sch/O_K$ an integral model of $X$, in particular
$X^{\mathcal O}\otimes_{O_K}K=X$. We have then the morphisms of sites $r:X^{et}\to X^{\mathcal O,et}$
and $r:X^{pet}\to X^{\mathcal O,pet}$ such that $\nu_{X^{\mathcal O}}\circ r=r\circ\nu_X$.
We then consider the embedding of $C(X^{et})$
\begin{equation*}
OL_X:=OL_X\circ\iota:
\Omega_{X^{et},\log,\mathcal O}^{\bullet}:=r^*\Omega_{X^{\mathcal O,et},\log}^{\bullet}
\hookrightarrow\Omega_{X^{et},\log}^{\bullet}\hookrightarrow\Omega_{X^{et}}^{\bullet}
\end{equation*}
consisting of integral logarithimc De Rham forms,
with $\iota:r^*\Omega_{X^{\mathcal O,et},\log}^{\bullet}\hookrightarrow\Omega_{X^{et},\log}^{\bullet}$.
We will also consider the embedding of $C(X^{\mathcal O,pet})$
\begin{eqnarray*}
OL_{\hat X^{\mathcal O,(p)}}:=(OL_{X^{\mathcal O}/p^n})_{n\in\mathbb N}:
\Omega_{\hat X^{(p)},\log,\mathcal O}^{\bullet}:=
\varprojlim_{n\in\mathbb N}\nu_{X^{\mathcal O}}^*\Omega^{\bullet}_{X^{\mathcal O,et}/p^n,\log} \\
\hookrightarrow\Omega_{\hat X^{\mathcal O,(p)}}^{\bullet}:=
\varprojlim_{n\in\mathbb N}\nu_{X^{\mathcal O}}^*\Omega^{\bullet}_{X^{\mathcal O,et}/p^n/(O_K/p^n)}
\end{eqnarray*}
where we recall $c:\hat X^{\mathcal O,(p)}\to X^{\mathcal O}$ the morphism in $\RTop$ 
is given by the completion along the ideal generated by $p$. 
We then get the embeddings of $C(X^{pet})$
\begin{equation*}
m\circ(OL_X\otimes I):
\Omega_{X^{pet},\log,\mathcal O}^{\bullet}\otimes\mathbb Z_p:=
\Omega_{X^{pet},\log,\mathcal O}^{\bullet}\otimes\underline{\mathbb Z_p}_X
\hookrightarrow\Omega_{X^{pet}}^{\bullet}, \; 
(w\otimes\lambda_n)_{n\in\mathbb N}\mapsto(\lambda_n)_{n\in\mathbb N}\cdot w
\end{equation*}
and
\begin{equation*}
OL_{\hat X^{(p)}}:=r^*OL_{\hat X^{\mathcal O,(p)}}:
r^*\Omega_{\hat X^{(p)},\log,\mathcal O}^{\bullet}\hookrightarrow r^*\Omega_{\hat X^{\mathcal O,(p)}}^{\bullet}
\hookrightarrow\Omega_{\hat X^{(p)}}^{\bullet}:=\Omega_{\hat X^{(p)}/K}^{\bullet},
\end{equation*}
where we recall $c:=(c\otimes_{O_K}K):\hat X^{(p)}\to X$ the morphism in $\RTop$. 
Note that the inclusion $\Omega^l_{X^{\mathcal O,et},\log}/p^n\subset\Omega^{\bullet}_{X^{\mathcal O,et}/p^n,\log}$
is strict in general.
Note that 
\begin{equation*}
\Omega_{X^{pet},\log,\mathcal O}^{\bullet}:=\nu_X^*\Omega_{X^{et},\log,\mathcal O}^{\bullet}\in C(X^{pet}),
\end{equation*}
but 
\begin{eqnarray*}
\underline{\mathbb Z_p}_X:=\varprojlim_{n\in\mathbb N}\nu_X^*(\mathbb Z/p^n\mathbb Z)_{X^{et}}\in C(X^{pet}), \, 
\Omega_{\hat X^{(p)}}^{\bullet}\in C(X^{pet}), \, \mbox{and} \\
r^*\Omega_{\hat X^{(p)},\log,\mathcal O}^{\bullet}
:=r^*\varprojlim_{n\in\mathbb N}\nu_{X^{\mathcal O}}^*\Omega_{X^{\mathcal O,et}/p^n,\log}
\in C(X^{pet}) 
\end{eqnarray*}
are NOT the pullback of etale sheaves by $\nu_X$. 
We consider $\Sch^{int}/O_K:=O(\PSch^2/O_K)\subset\Sch^{ft}/O_K$ the full subcategory
consisting of integrable models of algebraic varieties over $K$,
where $O:\PSch^2/O_K\to\Sch^{ft}/O_K, \; O(X,Z)=X\backslash Z$ is the canonical functor. 
We denote by $\Var(K)^{pet}$ and $(\Sch^{int}/O_K)^{pet}$ the big pro-etale sites.
We have then the morphism of sites $r:\Var(K)^{pet}\to(\Sch^{int}/O_K)^{pet}$.
We will consider the embedding of $C((\Sch^{int}/O_K)^{pet})$
\begin{eqnarray*}
OL_{/O_K,an}:\Omega^{\bullet,an}_{/K,\log,\mathcal O}\hookrightarrow\Omega_{/O_K}^{\bullet,an}, \;
\mbox{for} \, X^{\mathcal O}\in(\Sch^{int}/O_K)^{pet}, \\ 
OL_{/O_K,an}(X^{\mathcal O}):=OL_{\hat X^{\mathcal O,(p)}}(X^{\mathcal O}):
\Omega^{\bullet}_{\hat X^{(p)},\log,\mathcal O}(\hat X^{\mathcal O,(p)})
\hookrightarrow\Omega^{\bullet}_{\hat X^{\mathcal O,(p)}}(\hat X^{\mathcal O,(p)})
\end{eqnarray*}
We get the embedding of $C(\Var(K)^{pet})$
\begin{eqnarray*}
OL_{/K,an}:=r^*OL_{/O_K,an}:r^*\Omega^{\bullet,an}_{/K,\log,\mathcal O}
\hookrightarrow r^*\Omega_{/O_K}^{\bullet,an}\hookrightarrow\Omega_{/K}^{\bullet,an}
\end{eqnarray*}
and its restriction to $\SmVar(K)^{pet}\subset\Var(K)^{pet}$.
\end{itemize}
\end{defi}

Let $\sigma:k\hookrightarrow\mathbb C$ a subfield of finite type over $\mathbb Q$.
Consider $k\subset\bar k\subset\mathbb C$ the algebraic closure of $k$.
Let $X\in\Var(k)$. Let $p$ a prime number. Let $\sigma_p:k\hookrightarrow\mathbb C_p$ an embedding. 
We have then the following diagram in $C(\mathbb N\times X^{et})$
\begin{equation*}
\xymatrix{\an_{X*}\mathbb Z_{p,X_{\mathbb C}^{an}} & 
\an_{X*}\mathbb Z_{X_{\mathbb C}^{an}}\ar[r]^{\an_{X*}\beta(X)}\ar[l]_{(/p^*)} & 
\an_{X*}\Omega^{\bullet}_{X_{\mathbb C}^{an}} \\ 
\mathbb Z_{p,X_{\bar k}^{et}}\ar[u]^{\ad(\an_X^*,\an_{X*})(\mathbb Z_{p,X_{\bar k}^{et}})} 
\ar[r]^{\ad(\nu_X^*,\nu_{X*})(\mathbb Z_{p,X^{et}})} &
\nu_{X_{\bar k}*}\underline{\mathbb Z_p}_{X_{\bar k}^{pet}}
\ar[ldd]_{\ad(\pi_{\bar k/\mathbb C_p}(X_{\bar k}^{pet})^*,\pi_{\bar k/\mathbb C_p}(X_{\bar k}^{pet})_*)
(\underline{\mathbb Z_p}_{X_{\bar k}^{pet}})} & \,  \\
\, & \Omega^{\bullet}_{X^{et},\log}\ar[r]^{OL_{X^{et}}}\ar[d]_{\Omega(\pi_{k/\mathbb C_p}(X))}\ar[ruu] & 
\Omega^{\bullet}_{X^{et}}\ar[d]^{\Omega(\pi_{k/\hat k_{\sigma_p}}(X))}\ar[uu]_{\Omega(\an_X)} \\
\underline{\mathbb Z_p}_{X_{\mathbb C_p}^{pet}}\ar[r]^{\iota(X_{\mathbb C_p}^{pet})} &
\Omega^{\bullet}_{X^{pet}_{\mathbb C_p},\log,\mathcal O}\otimes\underline{\mathbb Z_p}_{X_{\mathbb C_p}^{pet}}
\ar[r]^{OL_{X_{\mathbb C_p}}\otimes I} & 
\Omega^{\bullet}_{X^{pet}_{\mathbb C_p}}\otimes_{O_{X_{\mathbb C_p}}}O\mathbb B_{dr,X_{\mathbb C_p}}}
\end{equation*}
with as above 
\begin{equation*}
\an_X:X_{\mathbb C}^{an}\simeq X_{\mathbb C}^{an,et}\xrightarrow{\an_X}
X_{\mathbb C}^{et}\xrightarrow{\pi_{\bar k/\mathbb C}(X_{\bar k}^{et})}X_{\bar k}^{et}, \; \;
\an_X:X_{\mathbb C}^{an}\simeq X_{\mathbb C}^{an,et}\xrightarrow{\an_X}
X_{\mathbb C}^{et}\xrightarrow{\pi_{k/\mathbb C}(X^{et})}X^{et}
\end{equation*}
In particular, we get for $X\in\Var(k)$ and $j\in\mathbb Z$, the canonical map
\begin{eqnarray*}
T(X):=H^jT(X):H^j_{\sing}(X_{\mathbb C}^{an},\mathbb Z)\xrightarrow{H^j(/p^*)}
H^j_{\sing}(X_{\mathbb C}^{an},\mathbb Z_p)\xrightarrow{\sim}H^j(X_{\mathbb C}^{an},\mathbb Z_p) \\
\xrightarrow{(\an_X^*)^{-1}}
H^j_{et}(X_{\mathbb C},\mathbb Z_p)\xrightarrow{(\pi_{\bar k/\mathbb C}(X)^*)^{-1}}H^j_{et}(X_{\bar k},\mathbb Z_p)
\end{eqnarray*}
for each $j\in\mathbb Z$, where $H^j(\an_X^*)$ is an isomorphism by the comparaison theorem between etale cohomology
and Betti cohomology with torsion coefficients (see SGA4).

Let $p\in\mathbb N$ a prime number. Let $k\subset\mathbb C_p$ a subfield. Let $X\in\SmVar(k)$. 
Take a compactification $j:X\hookrightarrow\bar X$ with $\bar X\in\PSmVar(k)$ 
such that $D:=\bar X\backslash X\subset\bar X$ is a normal crossing divisor. Then for each $j\in\mathbb Z$,
\begin{equation*}
H^j_{DR}(X_{\mathbb C_p})=H^j\Gamma(\bar X_{\mathbb C_p},
E_{zar}(\Omega^{\bullet}_{X_{\mathbb C_p}}(\log D_{\mathbb C_p}),F_b))
=H^j\Gamma(\bar X_{\mathbb C_p},E_{pet}(\Omega^{\bullet}_{X_{\mathbb C_p}}(\log D_{\mathbb C_p}),F_b))
\in\Vect_{fil}(\mathbb C_p).
\end{equation*}
By the complex case which follows from the Hodge decomposition, we see
after taking an isomorphism $\sigma:\mathbb C_p\xrightarrow{\sim}\mathbb C$ that  
the spectral sequence associated to 
$\Gamma(\bar X_{\mathbb C_p},E_{zar}(\Omega^{\bullet}_{X_{\mathbb C_p}}(\log D_{\mathbb C_p}),F_b))\in C_{fil}(\mathbb C_p)$
is $E_1$ degenerate. However there is no canonical splitting of the filtration on $H^j_{DR}(X_{\mathbb C_p})$
(it depend on the choice of a basis of this $\mathbb C_p$ vector space, or on the choice of such an isomorphism $\sigma$).
Then the canonical embedding in $C_{fil}(\bar X_{\mathbb C_p}^{pet})$
\begin{equation*}
OL_X:=j_*OL_X:j_*\Omega^{\bullet}_{X_{\mathbb C_p},\log,\mathcal O}
\hookrightarrow\Omega^{\bullet}_{X_{\mathbb C_p}}(\log D_{\mathbb C_p})\hookrightarrow 
j_*\Omega^{\bullet}_{X_{\mathbb C_p}}
\end{equation*}
induces in cohomology
\begin{eqnarray*}
H^jOL_X:\mathbb H_{pet}^j(X_{\mathbb C_p},\Omega^{\bullet\geq l}_{X_{\mathbb C_p},\log,\mathcal O})\to
\mathbb H_{pet}^j(X_{\mathbb C_p},\Omega^{\bullet\geq l}_{X_{\mathbb C_p}}(\log D_{\mathbb C_p}))
\xrightarrow{=}F^lH^j_{DR}(X_{\mathbb C_p})\hookrightarrow H^j_{DR}(X_{\mathbb C_p}).
\end{eqnarray*}
Similarly, the canonical embedding in $C_{fil}(\bar X_{\mathbb C_p}^{pet})$
\begin{equation*}
m\circ(OL_X\otimes I):j_*(\Omega^{\bullet}_{X_{\mathbb C_p},\log,\mathcal O}\otimes\mathbb Z_p)
\hookrightarrow\Omega^{\bullet}_{X_{\mathbb C_p}}(\log D_{\mathbb C_p})\hookrightarrow 
j_*\Omega^{\bullet}_{X_{\mathbb C_p}}
\end{equation*}
induces in cohomology
\begin{eqnarray*}
H^j(m\circ(OL_X\otimes I)):
\mathbb H_{pet}^j(X_{\mathbb C_p},\Omega^{\bullet\geq l}_{X_{\mathbb C_p},\log,\mathcal O}\otimes\mathbb Z_p)\to \\
\mathbb H_{pet}^j(X_{\mathbb C_p},\Omega^{\bullet\geq l}_{X_{\mathbb C_p}}(\log D_{\mathbb C_p}))
\xrightarrow{=}F^lH^j_{DR}(X_{\mathbb C_p})\hookrightarrow H^j_{DR}(X_{\mathbb C_p}).
\end{eqnarray*}
Consider, for each $j\in\mathbb Z$, a (non canonical) splitting
\begin{equation*}
\theta_j(X):H^j_{DR}(X_{\mathbb C_p})\xrightarrow{\sim}
\oplus_{0\leq l\leq j}H_{zar}^l(X_{\mathbb C_p},\Omega^{j-l}_{X_{\mathbb C_p}}(\log D_{\mathbb C_p}))=
\oplus_{0\leq l\leq j}H_{pet}^l(X_{\mathbb C_p},\Omega^{j-l}_{X_{\mathbb C_p}}(\log D_{\mathbb C_p}))
\end{equation*}
in $\Vect(\mathbb C_p)$. We then have the following map in $\Vect(\mathbb Z_p)$
\begin{eqnarray*}
OL_X^{\theta_j}:=\theta_j(X)^{-1}\circ(\oplus_{0\leq l\leq j}H^lOL_X^{l-j}): \\ 
\mathbb H_{pet}^j(X_{\mathbb C_p},\Omega^{\bullet}_{X_{\mathbb C_p},\log,\mathcal O}\otimes\mathbb Z_p)=
\oplus_{0\leq l\leq j}H_{pet}^l(X_{\mathbb C_p},\Omega^{j-l}_{X_{\mathbb C_p},\log,\mathcal O}\otimes\mathbb Z_p) \\
\xrightarrow{\oplus_{0\leq l\leq j}H^lOL_X^{l-j}}
\oplus_{0\leq l\leq j}H_{pet}^l(X_{\mathbb C_p},\Omega^{j-l}_{X_{\mathbb C_p}}(\log D_{\mathbb C_p}))
\xrightarrow{\theta_j(X)^{-1}}H^j_{DR}(X_{\mathbb C_p}),
\end{eqnarray*}
where the equality follows from the fact that the differential of 
$\Omega^{\bullet}_{X_{\mathbb C_p},\log,\mathcal O}$ vanishes (all the logarithmic forms are closed) so that
we get a canonical splitting.
Note that $H^jOL_X$ is NOT equal to $OL_X^{\theta_j}$. In fact we have for 
$w\in H_{pet}^l(X_{\mathbb C_p},\Omega^{j-l}_{X_{\mathbb C_p},\log,\mathcal O}\otimes\mathbb Z_p)$
\begin{equation*}
\pi_l\circ\theta_j(X)\circ H^jOL_X(w)=H^lOL_X^{j-l}(w)
\end{equation*}
where 
$\pi_l:\oplus_{0\leq l\leq j}H_{pet}^l(X_{\mathbb C_p},\Omega^{j-l}_{X_{\mathbb C_p}}(\log D_{\mathbb C_p}))
\to H_{pet}^l(X_{\mathbb C_p},\Omega^{j-l}_{X_{\mathbb C_p}}(\log D_{\mathbb C_p}))$ 
is the projection.

\section{Integral complex and $p$-adic periods of a smooth algebraic variety 
over a field $k$ of finite type over $\mathbb Q$}

\subsection{Complex integral periods}

Let $k$ a field of characteristic zero. 

Let $X\in\SmVar(k)$ a smooth variety. Let $X=\cup_{i=1}^sX_i$ an open affine cover. 
We have for $\sigma:k\hookrightarrow\mathbb C$ an embedding, the evaluation period embedding map
which is the morphism of bi-complexes 
\begin{eqnarray*}
ev(X)^{\bullet}_{\bullet}:\Gamma(X_{\bullet},\Omega^{\bullet}_{X_{\bullet}})\to 
\mathbb Z\Hom_{\Diff}(\mathbb I^{\bullet},X^{an}_{\mathbb C,\bullet})^{\vee}\otimes\mathbb C, \\
w^l_I\in\Gamma(X_I,\Omega^l_{X_I})\mapsto 
(ev(X)^l_I(w^l_I):\phi^l_I\in\mathbb Z\Hom_{\Diff}(\mathbb I^l,X^{an}_{\mathbb C,I})^{\vee}\otimes\mathbb C
\mapsto ev^l_I(w^l_I)(\phi^l_I):=\int_{\mathbb I^l}\phi_I^{l*}w^l_I)
\end{eqnarray*}
given by integration. By taking all the affine open cover $(j_i:X_i\hookrightarrow X)$ of $X$,
we get for $\sigma:k\hookrightarrow\mathbb C$, the evaluation period embedding map 
\begin{eqnarray*}
ev(X):=\varinjlim_{(j_i:X_i\hookrightarrow X)}ev(X)^{\bullet}_{\bullet}:
\varinjlim_{(j_i:X_i\hookrightarrow X)}\Gamma(X_{\bullet},\Omega^{\bullet}_{X_{\bullet}})
\to \varinjlim_{(j_i:X_i\hookrightarrow X)}
\mathbb Z\Hom_{\Diff(\mathbb R)}(\mathbb I^{\bullet},X^{an}_{\mathbb C,\bullet})^{\vee}\otimes\mathbb C
\end{eqnarray*}
It induces in cohomology, for $j\in\mathbb Z$, the evaluation period map
\begin{eqnarray*}
H^jev(X)=H^jev(X)^{\bullet}_{\bullet}:H^j_{DR}(X)=H^j\Gamma(X_{\bullet},\Omega^{\bullet}_{X_{\bullet}})\to 
H_{\sing}^j(X_{\mathbb C}^{an},\mathbb C)=
H^j(\Hom_{\Diff(\mathbb R)}(\mathbb I^{\bullet},X^{an}_{\mathbb C,\bullet})^{\vee}\otimes\mathbb C). 
\end{eqnarray*}
which does NOT depend on the choice of the affine open cover 
by acyclicity of quasi-coherent sheaves on affine noetherian schemes for the left hand side
and from Mayer-Vietoris quasi-isomorphism for singular cohomology of topological spaces
and Whitney approximation theorem for differential manifolds for the right hand side.

\begin{rem}
We also have for $\sigma:k\hookrightarrow\mathbb C$ the composition 
\begin{eqnarray*}
\bar ev(X)^{\bullet}_{\bullet}:\Gamma(X_{\bullet},\Omega^{\bullet}_{X_{\bullet}})
\xrightarrow{ev(X)^{\bullet}_{\bullet}} 
\mathbb Z\Hom_{\Diff(\mathbb R)}(\mathbb I^{\bullet},X^{an}_{\mathbb C\bullet})^{\vee}\otimes\mathbb C
\xrightarrow{\mathbb Z(X)(i)\circ\an^{-,-}} 
\Hom_{\Fun(\Delta^{\bullet},\Var(k))}(\mathbb D_{k,et}^{\bullet},X_{\bullet})^{\vee}\otimes\mathbb C 
\end{eqnarray*}
where $i:I^{\bullet}\hookrightarrow\mathbb D^{\bullet}_{k,et}$ is the embedding, which is given by integration : 
for $w^l_I\in\Gamma(X_I,\Omega^l_{X_I})$ and $\phi^l_I\in\Hom_{\Fun(\Delta^{\bullet},\Var(k))}(\mathbb D_{k,et}^j,X_I)$,
\begin{equation*}
\bar ev^l_I(w^l_I)(\phi^l_I)=\int_{\mathbb I^l}\phi_I^{l,an*}w^l_I.
\end{equation*}
\end{rem}

Let $X\in\SmVar(k)$. Note that 
\begin{equation*}
H^*ev(X_{\mathbb C}):H^*_{DR}(X_{\mathbb C})\xrightarrow{H^*R\Gamma(X_{\mathbb C}^{an},E_{zar}(\Omega(\an_X)))}
H^*_{DR}(X^{an}_{\mathbb C})\xrightarrow{H^*R\Gamma(X_{\mathbb C}^{an},\alpha(X))}
H^*_{\sing}(X_{\mathbb C}^{an},\mathbb C)
\end{equation*}
is the canonical isomorphism induced by the analytical functor and the quasi-isomorphism
$\alpha(X):\mathbb C_{X_{\mathbb C}^{an}}\hookrightarrow\Omega^{\bullet}_{X^{an}_{\mathbb C}}$ 
in $C(X_{\mathbb C}^{an})$. Hence,
\begin{equation*}
H^*ev(X)=H^*_{DR}(X)\xrightarrow{\Omega(\pi_{k/\mathbb C}(X))}H^*_{DR}(X_{\mathbb C})\xrightarrow{H^*ev(X_{\mathbb C})}
H^*_{\sing}(X_{\mathbb C}^{an},\mathbb C)
\end{equation*}
is injective. 
The elements of the image $H^*ev(X)(H^*_{DR}(X))\subset H^*_{\sing}(X_{\mathbb C}^{an},\mathbb C)$ are the periods of $X$.

Let $X\in\SmVar(k)$ a smooth variety. Let $X=\cup_{i=1}^sX_i$ an open affine cover with $X_i:=X\backslash D_i$
with $D_i\subset X$ smooth divisors with normal crossing. 
Let $\sigma:k\hookrightarrow\mathbb C$ an embedding and $X^{an}_{\mathbb C}=\cup_{i=1}^r\mathbb D_i$
an open cover with $\mathbb D_i\simeq D(0,1)^d$. 
Since a convex open subset of $\mathbb C^d$ is biholomorphic to an open ball we have 
$\mathbb D_I:=\cap_{i\in I}\mathbb D_i\simeq D(0,1)^d$ (where $d$ is the dimension of a connected component of $X$). 
Denote by $j_{\bullet}:X^{an}_{\bullet,\mathbb C}\cap\mathbb D_{\bullet}\to X^{an}_{\bullet,\mathbb C}$
is the open embeddings. We then have the period morphism of tri-complexes 
\begin{eqnarray*}
ev(X_{\mathbb C}^{an})^{\bullet}_{\bullet,\bullet}:
\Gamma(X^{an}_{\bullet,\mathbb C}\cap\mathbb D_{\bullet},\Omega^{\bullet}_{X_{\mathbb C}^{an}})\to 
\mathbb Z\Hom_{\Diff}(\mathbb I^{\bullet},X^{an}_{\mathbb C,\bullet})^{\vee}\otimes\mathbb C, \\
w^l_{I,J}\in\Gamma(X_{I,\mathbb C}^{an}\cap\mathbb D_J,\an_X^*\Omega^l_{X_I})\mapsto \\
(ev^l_{I,J}(w^l_{I,J}):
\phi^l_{I,J}\in\mathbb Z\Hom_{\Diff}(\mathbb I^l,X^{an}_{\mathbb C,I}\cap\mathbb D_J)^{\vee}\otimes\mathbb C
\mapsto ev^l_{I,J}(w^l_{I,J})(\phi^l_{I,J}):=\int_{\mathbb I^l}\phi_{I,J}^{l*}w^l_{I,J})
\end{eqnarray*}
given by integration. We have then the factorization
\begin{eqnarray*}
H^jev(X):H^j_{DR}(X):=\mathbb H^j(X,\Omega^{\bullet}_X)=\mathbb H_{et}^j(X,\Omega^{\bullet}_{X^{et}})
\xrightarrow{H^j\Omega(\pi_{k/\mathbb C}(X))} \\
H^j_{DR}(X_{\mathbb C}):=\mathbb H^j(X_{\mathbb C},\Omega^{\bullet}_X)=\mathbb H_{et}^j(X_{\mathbb C},\Omega^{\bullet}_{X^{et}})
\xrightarrow{j_{\bullet}^*\circ\an_{X_{\bullet}}^*}
H^j\Gamma(X^{an}_{\bullet,\mathbb C}\cap\mathbb D_{\bullet},\Omega^{\bullet}_{X_{\mathbb C}^{an}}) \\
\xrightarrow{H^jev(X_{\mathbb C}^{an})^{\bullet}_{\bullet,\bullet}} 
H_{\sing}^j(X_{\mathbb C}^{an}\cap\mathbb D_{\bullet},\mathbb C)=
H^j(\Hom_{\Diff(\mathbb R)}(\mathbb I^{\bullet},X^{an}_{\mathbb C,\bullet}\cap\mathbb D_{\bullet})^{\vee}\otimes\mathbb C). 
\end{eqnarray*}
where for the left hand side, the first equality follows from the fact that 
$\Omega^{\bullet}\in C(\SmVar(k))$ is $\mathbb A^1$ local and admits transferts,
and the equality of the right hand side follows 
from Mayer-Vietoris quasi-isomorphism for singular cohomology of topological spaces.

\begin{rem}
Let $X\in\SmVar(k)$ a smooth variety. Let $X=\cup_{i=1}^sX_i$ an open affine cover with $X_i:=X\backslash D_i$
with $D_i\subset X$ smooth divisors with normal crossing. 
Let $\sigma:k\hookrightarrow\mathbb C$ an embedding and $X^{an}_{\mathbb C}=\cup_{i=1}^r\mathbb D_i$
an open cover with $\mathbb D_i\simeq D(0,1)^d$. 
Since a convex open subset of $\mathbb C^d$ is biholomorphic to an open ball we have 
$\mathbb D_I:=\cap_{i\in I}\mathbb D_i\simeq D(0,1)^d$ (where $d$ is the dimension of a connected component of $X$). 
Denote by $j_{\bullet}:X^{an}_{\bullet,\mathbb C}\cap\mathbb D_{\bullet}\to X^{an}_{\bullet,\mathbb C}$
the open embeddings. Then,
\begin{eqnarray*}
j_{\bullet}^*\circ\an_{X_{\bullet}}^*:=\Omega(j_{\bullet}\circ\an_{X_{\bullet}}):
\Gamma(X_{\bullet,\mathbb C},\Omega^{\bullet}_{X^{et}})\to
\Gamma(X^{an}_{\bullet,\mathbb C}\cap\mathbb D_{\bullet},\Omega^{\bullet}_{X_{\mathbb C}^{an}}) 
\end{eqnarray*}
is a quasi-isomorphism by the Grothendieck comparaison theorem for De Rham cohomology
and the acyclicity of quasi-coherent sheaves on noetherian affine schemes.
\end{rem}

\begin{lem}\label{keypropClem}
Let $k\subset\mathbb C$ a subfield.
Let $X\in\SmVar(k)$ a smooth variety. Let $(r_i:X_i\to X)_{1\leq i\leq s}$ an affine etale cover. 
Let $X^{an}_{\mathbb C}=\cup_{i=1}^r\mathbb D_i$ an open cover with $\mathbb D_i\simeq D(0,1)^d$. Let 
\begin{equation*}
w=[\sum_{I,J,l,card I+card J+l=j}w^l_{I,J}]=\sum_{I,J,l,card I+card J+l=j}[w^l_{I,J}]
\in H^j\Gamma(X^{an}_{\bullet,\mathbb C}\cap\mathbb D_{\bullet},\Omega^{\bullet}_{X_{\mathbb C}^{an}}). 
\end{equation*}
Then the following assertions are equivalent :
\begin{itemize}
\item[(i)] $H^jev(X)(w)\in H^j_{\sing}(X_{\mathbb C}^{an},2i\pi\mathbb Q)$,
\item[(ii)] for all $I,J,l$ such that $card I+card J+l=j$ , there exist a lift 
\begin{equation*}
[\tilde w^l_{IJ}]\in H^l\Gamma(X^{an}_{I,\mathbb C}\cap\mathbb D_J,\Omega^{\bullet}_{X_{\mathbb C}^{an}})
\end{equation*}
of $[w^l_{IJ}]$ with respect to the spectral sequence associated to the filtration on the total complex 
associated to the bi-complex structure such that
\begin{equation*}
H^lev(X_{\mathbb C}^{an})^{\bullet}_{I,J}([\tilde w^l_{I,J}])
\in H^l_{\sing}(X_{I,\mathbb C}^{an}\cap\mathbb D_J,2i\pi\mathbb Q).
\end{equation*}
\end{itemize} 
\end{lem}

\begin{proof}
Follows immediately form the fact that 
$ev(X_{\mathbb C}^{an})^{\bullet}_{\bullet,\bullet}$ define by definition a morphism of spectral sequence for the filtration
given by the bi-complex structure.
\end{proof}

\begin{lem}\label{keypropClem2}
Let $k\subset\mathbb C$ a subfield. 
Let $\pi:X'\to X$ a proper generically finite (of degree $d$) morphism with $X,X'\in\SmVar(k)$.
Let $w\in H^j_{DR}(X)$. Then $ev(X)(w)\in H^j_{\sing}(X_{\mathbb C}^{an},2i\pi\mathbb Q)$ if and only if
$ev(X')(\pi^*w)\in H^j_{\sing}(X_{\mathbb C}^{'an},2i\pi\mathbb Q)$.
\end{lem}

\begin{proof}
Since $ev(X')(\pi^*w)=\pi^*ev(X)(w)\in H^j_{\sing}(X_{\mathbb C}^{'an},\mathbb C)$, 
the only if part is obvious whereas the if part follows from the formula of $\pi_*\pi^*=dI$.
\end{proof}

The main proposition of this section is the following :

\begin{prop}\label{keypropC}
Let $k$ a field of characteristic zero. Let $X\in\SmVar(k)$. 
Let $\sigma:k\hookrightarrow\mathbb C$ an embedding.
\begin{itemize}
\item[(i)]Let $w\in H^j_{DR}(X):=\mathbb H^j(X,\Omega^{\bullet}_X)=\mathbb H_{pet}^j(X,\Omega^{\bullet}_X)$. If 
\begin{equation*}
w\in H^jOL_X(\mathbb H_{pet}^j(X_{\mathbb C},\Omega^{\bullet}_{X_{\mathbb C}^{et},\log}))
\end{equation*}
then $H^jev(X)(w)\in H^j_{\sing}(X_{\mathbb C}^{an},2i\pi\mathbb Q)$. 
\item[(ii)]Let $p\in\mathbb N$ a prime number and $\sigma_p:k\hookrightarrow\mathbb C_p$ an embedding.
Let $j\in\mathbb Z$.
Let $w\in H^j_{DR}(X):=\mathbb H^j(X,\Omega^{\bullet}_X)=\mathbb H_{pet}^j(X,\Omega^{\bullet}_X)$. If 
\begin{equation*}
w:=\pi_{k/\mathbb C_p}(X)^*w\in H^jOL_X
(\mathbb H_{et}^j(X_{\mathbb C_p},\Omega^{\bullet}_{X_{\mathbb C_p}^{et},\log,\mathcal O}))
\end{equation*}
then $H^jev(X)(w)\in H^j_{\sing}(X_{\mathbb C}^{an},2i\pi\mathbb Q)$. Recall that
$\mathbb H_{pet}^j(X_{\mathbb C},\Omega^{\bullet}_{X_{\mathbb C_p}^{pet},\log,\mathcal O})=
\mathbb H_{et}^j(X_{\mathbb C},\Omega^{\bullet}_{X_{\mathbb C_p}^{et},\log,\mathcal O})$.
\end{itemize}
\end{prop}

\begin{proof}
\noindent(i):
Let 
\begin{equation*}
w\in H^j_{DR}(X):=\mathbb H^j(X,\Omega^{\bullet}_X)=H^j\Gamma(X_{\bullet},\Omega^{\bullet}_X).
\end{equation*}
where $(r_i:X_i\to X)_{1\leq i\leq s}$ is an affine etale cover.
Let $X^{an}_{\mathbb C}=\cup_{i=1}^r\mathbb D_i$ an open cover with $\mathbb D_i\simeq D(0,1)^d$. 
Denote $j_{IJ}:X_I\cap\mathbb D_J\hookrightarrow X_I$ the open embeddings.
Then by definition $H^jev(X)(w)=H^jev(X_{\mathbb C}^{an})(j_{\bullet}^*\circ\an_{X_{\bullet}}^*w)$ with
\begin{equation*}
j_{\bullet}^*\circ\an_{X_{\bullet}}^*w
\in H^j\Gamma(X^{an}_{\bullet,\mathbb C}\cap\mathbb D_{\bullet},\Omega^{\bullet}_{X_{\mathbb C}^{an}}). 
\end{equation*}
Now, if 
$w=H^jOL_X(\mathbb H_{et}^j(X_{\mathbb C},\Omega^{\bullet}_{X_{\mathbb C}^{et},\log}))$,
we have a canonical splitting
\begin{eqnarray*}
w=\sum_{l=0}^jw_L^{l,j-l}=\sum_{l=d}^jw^{l,j-l}\in H^j_{DR}(X_{\mathbb C}), \; 
w_L^{l,j-l}\in H^{j-l}(X_{\mathbb C},\Omega^l_{X^{et}_{\mathbb C},\log}), \;
w^{l,j-l}:=H^jOL_{X_{\mathbb C}^{et}}(w_L^{l,j-l}).
\end{eqnarray*}
Let $0\leq l\leq j$. Using an affine w-contractile pro-etale cover of $X$,
we see that there exists an affine etale cover 
$r=r(w^{l,j-l})=(r_i:X_i\to X)_{1\leq i\leq n}$ of $X$ (depending on $w^{l,j-l}$) such that
\begin{eqnarray*}
w^{l,j-l}=[(w^{l,j-l}_{L,I})_I]\in  
H^jOL_{X_{\mathbb C}^{et}}(H^{j-l}\Gamma(X_{\mathbb C,\bullet},\Omega^l_{X_{\mathbb C},\log}))
\subset\mathbb H^j\Gamma(X_{\mathbb C,\bullet},\Omega^{\bullet}_{X_{\mathbb C}}).
\end{eqnarray*}
Note that since $X$ is an algebraic variety, this also follows from a comparison theorem between
Chech cohomology of etale covers and etale cohomology.
By lemma \ref{keypropClem2}, we may assume, 
up to take a desingularization $\pi:X'\to X$ of $(X,\cup_i(r_i(X\backslash X_i)))$ and replace $w$ with $\pi^*w$,
that $r_i(X_i)=r_i(X_i(w))=X\backslash D_i$ with $D_i\subset X$ smooth divisors with normal crossing
For $1\leq l\leq j$, we get
\begin{eqnarray*}
w^{l,j-l}_{L,I}=\sum_{\nu}df_{\nu_1}/f_{\nu_1}\wedge\cdots\wedge df_{\nu_l}/f_{\nu_l}
\in\Gamma(X_{\mathbb C,I},\Omega^l_{X_{\mathbb C}}).
\end{eqnarray*}
For $l=0$, we get 
\begin{equation*}
w^{0,j}=[(\lambda_I)]\in H^j\Gamma(X_{\mathbb C,I},O_{X_{\mathbb C_p,I}}), \;
\lambda_I\in\Gamma(X_{\mathbb C,I},\mathbb Z_{X_{\mathbb C,I}^{et}})
\end{equation*}
There exists $k'\subset\mathbb C$ containing $k$ such that $w^{l,j-l}_{L,I}\in\Gamma(X_{k',I},\Omega^l_{X_{k'}})$
for all $0\leq l\leq j$.
Taking an embedding $\sigma':k'\hookrightarrow\mathbb C$ such that $\sigma'_{|k}=\sigma$, we then have
\begin{eqnarray*}
j_{\bullet}^*\circ\an_{X_{\bullet}}^*w=j_{\bullet}^*((m_l\cdot w^{l,j-l}_L)_{0\leq l\leq j})=(w^{l,j-l}_{L,I,J})_{l,I,J}
\in H^j\Gamma(X^{an}_{\bullet,\mathbb C}\cap\mathbb D_{\bullet},\Omega_{X_{\mathbb C}^{an}}^{\bullet}). 
\end{eqnarray*}
where for each $(I,J,l)$ with $card I+card J+l=j$, 
\begin{equation*}
w^{l,j-l}_{L,I,J}:=j_{IJ}^*w_{L,I}^{l,j-l}\in 
\Gamma(X^{an}_{I,\mathbb C}\cap\mathbb D_J,\Omega_{X_{\mathbb C}^{an}}^l).
\end{equation*} 
We have by a standard computation, for each $(I,J,l)$ with $card I+card J+l=j$,
\begin{equation*}
H^*_{\sing}(X_{I,\mathbb C}^{an}\cap\mathbb D_J,\mathbb Z)=<\gamma_1,\cdots,\gamma_{card I}>,
\end{equation*}
where for $1\leq i\leq card I$, $\gamma_i\in\Hom(\Delta^*,X_{I,\mathbb C}^{an}\cap\mathbb D_J)$
are products of loops around the origin inside the pointed disc $\mathbb D^1\backslash 0$.
On the other hand,
\begin{itemize}
\item $w^{l,j-l}_{L,I,J}=j_J^*(\sum_{\nu}df_{\nu_1}/f_{\nu_1}\wedge\cdots\wedge df_{\nu_l}/f_{\nu_l})
\in\Gamma(X_{I,\mathbb C}^{an}\cap\mathbb D_J,\Omega^l_{X_{\mathbb C}^{an}})$
for $1\leq l\leq j$, 
\item $w^{0,j}_{L,I,J}=\lambda_I$ is a constant.
\end{itemize}
Hence, for $\mu\in P([1,\cdots,s])$ with $card\mu=l$, we get, for $l=0$ 
$H^lev(X_{\mathbb C}^{an})_{I,J}(w_{L,I,J}^{0,j})=0$ and, for $1\leq l\leq j$,
\begin{eqnarray*}
H^lev(X_{\mathbb C}^{an})_{I,J}(w_{L,I,J}^{l,j-l})(\gamma_{\mu})=\sum_k\delta_{\nu,\mu}2i\pi\in 2i\pi\mathbb Z.
\end{eqnarray*}
where $\gamma_{\mu}:=\gamma_{\mu_1}\cdots\gamma_{\mu_l}$.
We conclude by lemma \ref{keypropClem}.

\noindent(ii):It is a particular case of (i).
\end{proof}

Let $k$ a field of characteristic zero.
Let $X\in\SmVar(k)$. Let $X=\cup_{i=1}^sX_i$ an open affine cover with $X_i:=X\backslash D_i$ 
with $D_i\subset X$ smooth divisors with normal crossing. Let $\sigma:k\hookrightarrow\mathbb C$ an embedding.
By proposition \ref{keypropC}, we have a commutative diagram of graded algebras
\begin{equation*}
\xymatrix{H^*_{DR}(X)\ar[rrr]^{H^*ev(X)} & \, & \, & H^*_{\sing}(X_{\mathbb C}^{an},\mathbb C) \\
H^*OL_X(\mathbb H_{et}^*(X,\Omega^{\bullet}_{X^{et},\log}))
\cap H^*_{DR}(X)\ar[u]^{\subset}\ar[rrr]^{H^*ev(X)} & \, & \, &
H^*_{\sing}(X_{\mathbb C}^{an},2i\pi\mathbb Q)\ar[u]_{H^*C^*\iota_{2i\pi\mathbb Q/\mathbb C}(X_{\mathbb C}^{an})}}
\end{equation*}
where 
\begin{equation*}
C^*\iota_{2i\pi\mathbb Q/\mathbb C}(X_{\mathbb C}^{an}):
C^{\bullet}_{\sing}(X_{\mathbb C}^{an},2i\pi\mathbb Q)\hookrightarrow C^{\bullet}_{\sing}(X_{\mathbb C}^{an},\mathbb C)
\end{equation*}
is the subcomplex consiting of $\alpha\in C^j_{\sing}(X_{\mathbb C}^{an},\mathbb C)$ 
such that $\alpha(\gamma)\in 2i\pi\mathbb Q$ for all $\gamma\in C_j^{\sing}(X_{\mathbb C}^{an},\mathbb Q)$.
Recall that 
\begin{equation*}
H^*ev(X_{\mathbb C})=H^*R\Gamma(X_{\mathbb C}^{an},\alpha(X))\circ\Gamma(X_{\mathbb C}^{an},E_{zar}(\Omega(\an_X))):
H^*_{DR}(X_{\mathbb C})\xrightarrow{\sim}H^*_{\sing}(X_{\mathbb C}^{an},\mathbb C)
\end{equation*}
is the canonical isomorphism induced by the analytical functor and 
$\alpha(X):\mathbb C_{X_{\mathbb C}^{an}}\hookrightarrow\Omega^{\bullet}_{X^{an}_{\mathbb C}}$, 
which gives the periods elements $H^*ev(X)(H^*_{DR}(X))\subset H^*_{\sing}(X_{\mathbb C}^{an},\mathbb C)$.
On the other side the induced map
\begin{equation*}
H^*ev(X_{\mathbb C}):
H^*OL_X(\mathbb H_{et}^*(X_{\mathbb C},\Omega^{\bullet}_{X_{\mathbb C}^{et},\log}))
\hookrightarrow H^*\iota_{2i\pi\mathbb Q/\mathbb C}H^*_{\sing}(X_{\mathbb C}^{an},2i\pi\mathbb Q)
\end{equation*}
is NOT surjective in general since the left hand side is invariant by the action of the group
$Aut(\mathbb C)$ (the group of field automorphism of $\mathbb C$) 
whereas the right hand side is not.
The fact for a de Rham cohomology class of being logarithmic is algebraic 
and invariant under isomorphism of (abstract) schemes.

\subsection{Rigid GAGA for logarithmic de Rham classes}

Let $K$ a field of characteristic zero which is complete for a $p$-adic norm.
We consider $\Sch^{int}/O_K:=O(\PSch^2/O_K)\subset\Sch^{ft}/O_K$ the full subcategory
consisting of integrable models of algebraic varieties over $K$,
where $O:\PSch^2/O_K\to\Sch^{ft}/O_K, \; O(X,Z)=X\backslash Z$ is the canonical functor. 
Denote by $\Var(K)^{pet}$ and $(\Sch^{int}/O_K)^{pet}$ the big pro-etale sites.
We then have the morphsim of sites $r:\Var(K)^{pet}\to(\Sch^{int}/O_K)^{pet}$
We will consider 
\begin{equation*}
\Omega^{\bullet,an}_{/K,\log,\mathcal O}\in C((\Sch^{int}/O_K)^{pet}), \;
X^{\mathcal O}\mapsto\Omega^{\bullet}_{\hat X^{(p)},\log,\mathcal O}(\hat X^{\mathcal O,(p)}), \, 
(f:X'\to X)\mapsto\Omega(f):=f^*
\end{equation*}
and the embedding of $C(\Var(K)^{pet})$ (see definition \ref{wlogdef}(iii))
\begin{eqnarray*}
OL:=OL_{/K,an}:r^*\Omega^{\bullet,an}_{/K,\log,\mathcal O}\hookrightarrow\Omega_{/K}^{\bullet,an}, \;
\mbox{for} \, X\in\Var(K)^{pet}, \\ 
OL_{/K,an}(X):=OL_{\hat X^{(p)}}(X):
\varinjlim_{X^{\mathcal O}\in(\Sch^{int}/O_K)^{pet},X^{\mathcal O}\times_{O_K}K=X}
\Omega^{\bullet}_{\hat X^{(p)},\log,\mathcal O}(\hat X^{(p)})\to\Omega^{\bullet}_{\hat X^{(p)}}(\hat X^{(p)}).
\end{eqnarray*}
and their restrictions to $\SmVar(K)^{pet}\subset\Var(K)^{pet}$.
Denote $\FSch/O_K$ is the category of formal schemes over $O_K$.
The morphism of site in $\RCat$ given by completion with respect to $(p)$
\begin{equation*}
c:\FSch/O_K\to\Var(K),\; X\mapsto\hat X:=\hat X^{\mathcal O,(p)}\otimes_{O_K}K, \; 
X^{\mathcal O}\in\Sch^{int}/O_K, \, \mbox{s.t} \, X^{\mathcal O}\otimes_{O_K}K=X
\end{equation*}
induces the map in $C(\Var(K)^{pet})$
\begin{equation*}
c^*:\Omega^{\bullet}_{/K}\to\Omega_{/K}^{\bullet,an}:=c_*\Omega_{/K}^{\bullet,an}, \;
\mbox{for} \, X\in\Var(K), \, c^*:\Omega^{\bullet}_{X}(X)\to\Omega^{\bullet}_{\hat X^{(p)}}(\hat X^{(p)}),
\end{equation*}
which induces for each $j\in\mathbb Z$ and $X\in\Var(K)$ the morphism 
\begin{equation*}
c^*:H^j_{DR}(X)=\mathbb H_{et}^j(X,\Omega^{\bullet}_X)\to
\mathbb H_{et}^j(X,\Omega^{\bullet}_{\hat X^{(p)}})=\mathbb H_{et}^j(\hat X^{(p)},\Omega^{\bullet}_{\hat X^{(p)}})
\end{equation*}
which is an isomorphism for $X\in\SmVar(K)$ or $X\in\PVar(K)$ by GAGA (c.f. EGA 3) and the fact that 
for the etale toplogy $c_*=Rc_*$.
On the on the other hand, it is well-known (see \cite{Sch}) that for $X\in\SmVar(K)$, by taking an open cover $X=\cup_iX_i$
such that there exists etale maps $X_i\to T^{d_X}\subset\mathbb A^{d_X}$, we have 
\begin{equation*}
\mathbb H_{pet}^j(\hat X^{(p)},\Omega^{\bullet}_{\hat X^{(p)}})=
\mathbb H_{et}^j(\hat X^{(p)},\Omega^{\bullet}_{\hat X^{(p)}}), \;
\mbox{that is}, \;
\mathbb H_{pet}^j(X,\Omega^{\bullet}_{\hat X^{(p)}})=
\mathbb H_{et}^j(X,\Omega^{\bullet}_{\hat X^{(p)}}).
\end{equation*}
We have also the sheaves
\begin{eqnarray*}
\hat O,\hat O^*\in\PSh(\Var(K)), \; X\in\Var(K)\mapsto 
\hat O(X):=\varprojlim_{n\in\mathbb N} O(X^{\mathcal O}_{/p^n}),\, 
\hat O^*(X):=\varprojlim_{n\in\mathbb N}O(X^{\mathcal O}_{/p^n})^*, \\ 
(g:Y\to X)\mapsto (a_g(X^{\mathcal O}_{/p^n}))_{n\in\mathbb N}:\hat O(X)\to \hat O(Y), \, \hat O(X)^*\to \hat O(Y)^*
\end{eqnarray*}
The sheaf $\hat O^*\in\PSh(\SmVar(k))$ admits transfers : for $W\subset X'\times X$ with $X,X'\in\SmVar(K)$
and $W$ finite over $X'$ and $f\in O(X^{\mathcal O}_{/p^n})^*$, $W^*f:=N_{W/X'}(p_X^*f)$ 
where $p_X:W\hookrightarrow X'\times X\to X$ is the projection and $N_{W/X'}:k(W)^*\to k(X')^*$ is the norm map.
This gives transfers on $\Omega^{1,an}_{/K,\log,\mathcal O}\in\PSh(\SmVar(K))$ compatible with transfers on 
$\Omega^{1,an}_{/K}\in\PSh(\SmVar(K))$ :
for $W\subset X'\times X$ with $X,X'\in\SmVar(k)$ and $W$ finite over $X'$ and $f\in O(X^{\mathcal O}_{/p^n})^*$, 
\begin{equation*}
W^*df/f:=dW^*f/W^*f=Tr_{W/X'}(p_X^*(df/f)), 
\end{equation*}
where where $p_X:W\hookrightarrow X'\times X\to X$ is the projection and $Tr_{W/X'}:O_W\to O_X$ is the trace map. 
Note that $d(fg)/fg=df/f+dg/g$. We get transfers on 
\begin{equation*}
\otimes^l_{\mathbb Q_p}\Omega^{1,an}_{/K,\log,\mathcal O}, \, \otimes_{\hat O}^l\Omega^{1,an}_{/K}\in\PSh(\SmVar(K))
\end{equation*}
since 
$\otimes^l_{\mathbb Q_p}\Omega^{1,an}_{/K,\log,\mathcal O}=H^0(\otimes^{L,l}_{\mathbb Q_p}\Omega^{1,an}_{/K,\log,\mathcal O})$ 
and $\otimes_{\hat O}^l\Omega^{1,an}_{/k}=H^0(\otimes_{\hat O}^{L,l}\Omega^{1,an}_{/K})$.
This induces transfers on  
\begin{eqnarray*}
\wedge^l_{\mathbb Q_p}\Omega^{1,an}_{/K,\log,\mathcal O}:=
\coker(\oplus_{I_2\subset[1,\ldots,l]}\otimes^{l-1}_{\mathbb Q_p}\Omega^{1,an}_{/K,\log,\mathcal O}
\xrightarrow{\oplus_{I_2\subset[1,\ldots,l]}\Delta_{I_2}:=(w\otimes w'\mapsto w\otimes w\otimes w')}
\otimes^l_{\mathbb Q_p}\Omega^{1,an}_{/K,\log,\mathcal O}) \\
\in\PSh(\SmVar(K)).
\end{eqnarray*}
and
\begin{equation*}
\wedge^l_{\hat O}\Omega^{1,an}_{/K}:=
\coker(\oplus_{I_2\subset[1,\ldots,l]}\otimes^{l-1}_{\hat O}\Omega^{1,an}_{/K}
\xrightarrow{\oplus_{I_2\subset[1,\ldots,l]}\Delta_{I_2}:=(w\otimes w'\mapsto w\otimes w\otimes w')})
\otimes^l_{\hat O}\Omega^{1,an}_{/K}
\in\PSh(\SmVar(K)).
\end{equation*}

We will use the following standard constructions and facts on closed pairs of algebraic varieties
(see \cite{CD}) :
\begin{itemize}
\item Let $(X,Z)\in\Sch^2$ with $X\in\Sch$ a noetherian scheme and $Z\subset X$ a closed subset.
We have the deformation $(D_ZX,\mathbb A^1_Z)\to\mathbb A^1$, $(D_ZX,\mathbb A^1_Z)\in\Sch^2$ 
of $(X,Z)$ by the normal cone $C_{Z/X}\to Z$, i.e. such that 
\begin{equation*}
(D_ZX,\mathbb A^1_Z)_s=(X,Z), \, s\in\mathbb A^1\backslash 0, \;  (D_ZX,\mathbb A^1_Z)_0=(C_{Z/X},Z).
\end{equation*}
We denote by $i_1:(X,Z)\hookrightarrow (D_ZX,\mathbb A^1_Z)$ and 
$i_0:(C_{Z/X},Z)\hookrightarrow (D_ZX,\mathbb A^1_Z)$ the closed embeddings in $\Sch^2$.
\item Let $k$ a field of characteristic zero. Let $X\in\SmVar(k)$. 
For $Z\subset X$ a closed subset of pure codimension $c$,
consider a desingularisation $\epsilon:\tilde Z\to Z$ of $Z$ and denote $n:\tilde Z\xrightarrow{\epsilon}Z\subset X$.
We have then the morphism in $\DA(k)$
\begin{equation*}
G_{Z,X}:M(X)\xrightarrow{D(\mathbb Z(n))}M(\tilde Z)(c)[2c]\xrightarrow{\mathbb Z(\epsilon)}M(Z)(c)[2c]
\end{equation*}
where $D:\Hom_{\DA(k)}(M_c(\tilde Z),M_c(X))\xrightarrow{\sim}\Hom_{\DA(k)}(M(X),M(\tilde Z)(c)[2c])$
is the duality isomorphism from the six functors formalism (moving lemma of Suzlin and Voevodsky)
and $\mathbb Z(n):=\ad(n_!,n^!)(a_X^!\mathbb Z)$, noting that $n_!=n_*$ since $n$ is proper and that
$a_X^!=a_X^*[d_X]$ and $a_{\tilde Z}^!=a_{\tilde Z}^*[d_Z]$ since $X$, resp. $\tilde Z$, are smooth
(considering the connected components, we may assume $X$ and $\tilde Z$ of pure dimension).
\item Let $X\in\SmVar(k)$ and $Z\subset X$ a smooth closed subvariety.
The closed embeddings $i_1:(X,Z)\hookrightarrow (D_ZX,\mathbb A^1_Z)$ and 
$i_0:(C_{Z/X},Z)\hookrightarrow (D_ZX,\mathbb A^1_Z)$ in $\SmVar^2(k)$ induces isomorphisms of motives
$\mathbb Z(i_1):M_Z(X)\xrightarrow{\sim}M_{\mathbb A^1_Z}(D_ZX)$ and 
$\mathbb Z(i_0):M_Z(N_{Z/X})\xrightarrow{\sim}M_{\mathbb A^1_Z}(D_ZX)$ in $\DA(k)$.
We get the excision isomorphism in $\DA(k)$
\begin{equation*}
P_{Z,X}:=\mathbb Z(i_0)^{-1}\circ\mathbb Z(i_1):M_Z(X)\xrightarrow{\sim}M_Z(N_{Z/X}).
\end{equation*}
We have 
\begin{equation*}
Th(N_{Z/X})\circ P_{Z,X}\circ\gamma^{\vee}_Z(\mathbb Z_X)=G_{Z,X}:=D(\mathbb Z(i)):M(X)\to M(Z)(d)[2d].
\end{equation*}
\end{itemize}

The result of this section is the following :

\begin{prop}\label{GAGAlog}
Let $K$ a field of characteristic zero which is complete for a $p$-adic norm. 
\begin{itemize}
\item[(i)]Let $X\in\PSmVar(K)$.
For each $j,l\in\mathbb Z$, the isomorphism
$c^*:\mathbb H_{pet}^j(X,\Omega^{\bullet}_X)\xrightarrow{\sim}\mathbb H_{pet}^j(X,\Omega^{\bullet}_{\hat X^{(p)}})$
and its inverse preserve logarithmic classes, that is
\begin{equation*}
c^*(H^j(m\circ(OL_X\otimes I))(H_{pet}^{j-l}(X,\Omega^l_{X^{pet},\log,\mathcal O}\otimes\mathbb Z_p)))
=H^jOL_{\hat X^{(p)}}(H_{pet}^{j-l}(X,\Omega^l_{\hat X^{(p)},\log,\mathcal O}))
\end{equation*}
is an isomorphism.
\item[(ii)]Let $X\in\SmVar(K)$.
For each $j,l\in\mathbb Z$, the isomorphism
$c^*:\mathbb H_{pet}^j(X,\Omega^{\bullet}_X)\xrightarrow{\sim}\mathbb H_{pet}^j(X,\Omega^{\bullet}_{\hat X^{(p)}})$
and its inverse preserve logarithmic classes, that is
\begin{equation*}
c^*(H^j(m\circ(OL_X\otimes I))(H_{pet}^{j-l}(X,\Omega^l_{X^{pet},\log,\mathcal O}\otimes\mathbb Z_p)))
=H^jOL_{\hat X^{(p)}}(H_{pet}^{j-l}(X,\Omega^l_{\hat X^{(p)},\log,\mathcal O}))
\end{equation*}
is an isomorphism.
\end{itemize}
\end{prop}

\begin{proof}
\noindent(i):Consider, for $j,l\in\mathbb Z$, the presheaf 
\begin{equation*}
L^{l,j-l}:=H^jOL_{/K,an}(H^{j-l}E^{\bullet}_{pet}\Omega_{/K,\log,\mathcal O}^{l,an})\in\PSh(\SmVar(K)). 
\end{equation*}
By definition, for $X\in\SmVar(K)$, 
\begin{equation*}
L^{l,j-l}(X):=H^jOL_{/K,an}(H^{j-l}\Hom(\mathbb Z(X),E^{\bullet}_{pet}(r^*\Omega_{/K,\log,\mathcal O}^{l,an})))
=H^jOL_{\hat X^{(p)}}(H_{pet}^{j-l}(X,\Omega^l_{\hat X^{(p)},\log,\mathcal O})).
\end{equation*}
Consider also for $X\in\SmVar(K)$ and $Z\subset X$ a closed subset,
\begin{equation*}
L_Z^{l,j-l}(X):=H^jOL_{/K,an}(H^{j-l}\Hom(\mathbb Z(X,X\backslash Z), 
E^{\bullet}_{pet}(r^*\Omega_{/K,\log,\mathcal O}^{l,an})))
=H^jOL_{\hat X^{(p)}}(H_{pet,Z}^{j-l}(X,\Omega_{\hat X^{(p)},\log,\mathcal O}^l)).
\end{equation*}
For the second equalities of $L^{l,j-l}(X)$ and $L_Z^{l,j-l}(X)$, note that 
\begin{equation*}
a_{et}OL_{/K,an}(r^*\Omega_{/K,\log,\mathcal O}^{l,an})_{|X^{pet}}
=a_{et}OL_{\hat X^{(p)}}(r^*\Omega_{\hat X^{(p)},\log,\mathcal O}),
\end{equation*}
where $X^{pet}\subset\Var(K)^{pet}$ is the small pro etale site and $a_{et}:\PSh(X^{pet})\to\Shv_{et}(X^{pet})$
is the sheaftification functor. 
For $l,n\in\mathbb N$, consider the presheaves in $(\Sch^{int}/O_K)^{pet}$
\begin{eqnarray*}
OL_{/O_K,p^n}:\Omega^l_{/(O_K/p^n),\log}\hookrightarrow\Omega^l_{/(O_K/p^n)}, \;
\mbox{for} \, X^{\mathcal O}\in(\Sch^{int}/O_K)^{pet}, \\
\Omega^l_{/(O_K/p^n),\log}(X^{\mathcal O}):=\Omega_{X^{\mathcal O}/p^n,\log}(X^{\mathcal O})
\hookrightarrow\Omega^l_{/(O_K/p^n)}(X^{\mathcal O}):=\Omega_{X^{\mathcal O}/p^n}(X^{\mathcal O})
\end{eqnarray*}
which induces the presheaves in $\SmVar(K)^{pet}$
\begin{equation*}
OL_{/(O_K/p^n)}:=r^*OL_{/O_K,p^n}:r^*\Omega^l_{/(O_K/p^n),\log}\hookrightarrow r^*\Omega^l_{/(O_K/p^n)}.
\end{equation*}
Note that by definition $\Omega^l_{/(O_K/p^n),\log}=\nu^*\nu_*\Omega^l_{/(O_K/p^n),\log}$ and 
$\Omega^l_{/(O_K/p^n)}=\nu^*\nu_*\Omega^l_{/(O_K/p^n)}$, where $\nu:(\Sch^{int}/O_K)^{pet}\to\Sch^{int}/O_K$.
We have then for $X\in\SmVar(K)$
\begin{eqnarray*}
L^{l,j-l}(X)
&=&H^jOL_{\hat X^{(p)}}(\varprojlim_{n\in\mathbb N}H_{et}^{j-l}(X,\Omega^l_{X^{\mathcal O}/p^n,\log,\mathcal O})) \\
:&=&H^jOL_{/K,an}(\varprojlim_{n\in\mathbb N}H^{j-l}\Hom(\mathbb Z(X,X\backslash Z), 
E^{\bullet}_{et}(r^*\Omega^l_{/(O_K/p^n),\log}))),
\end{eqnarray*}
since 
\begin{equation*}
R\varprojlim_{n\in\mathbb N}\Omega^l_{X^{\mathcal O}/p^n,\log,\mathcal O}
\simeq\varprojlim_{n\in\mathbb N}\Omega^l_{X^{\mathcal O}/p^n,\log,\mathcal O}\in D(X^{pet}) 
\end{equation*}
as the map in $\PSh(X^{pet})$
\begin{equation*}
r^*\Omega(/p^{n'}):r^*\Omega^l_{X^{\mathcal O}/p^{n-n'},\log,\mathcal O}\to r^*\Omega^l_{X^{\mathcal O}/p^n,\log,\mathcal O},
\end{equation*}
is surjective for the etale topology, where $/p^{n'}:X^{\mathcal O}/p^n\hookrightarrow X^{\mathcal O}/p^{n-n'}$
is induced by the quotient map $/p^{n'}:O_{X^{\mathcal O}}/p^{n-n'}\to O_{X^{\mathcal O}}/p^n$,
and as the pro-etale site is a replete topos (c.f. \cite{BSch}).
We have for $X\in\SmVar(K)$ and $Z\subset X$ a closed subset,
\begin{eqnarray*}
L_Z^{l,j-l}(X)
&=&H^jOL_{\hat X^{(p)}}(\varprojlim_{n\in\mathbb N}H_{et,Z}^{j-l}(X,\Omega^l_{X^{\mathcal O}/p^n,\log,\mathcal O})) \\
:&=&H^jOL_{/K,an}(\varprojlim_{n\in\mathbb N}H^{j-l}\Hom(\mathbb Z(X,X\backslash Z), 
E^{\bullet}_{et}(r^*\Omega^l_{/(O_K/p^n),\log}))).
\end{eqnarray*}
The presheaves $\Omega^l_{/(O_K/p^n),\log}\in\PSh(\SmVar(K))$, $l,n\in\mathbb N$, 
are $\mathbb A^1$ invariant and admit transfers.
Hence by a theorem of Voevodsky (c.f. \cite{CD} for example), 
$\Omega^l_{/(O_K/p^n),\log}\in\PSh(\SmVar(K))$, $l,n\in\mathbb N$, 
are $\mathbb A^1$ local since they are $\mathbb A^1$ invariant and admit transfers.
This gives in particular, for $Z\subset X$ a smooth subvariety of (pure) codimension $d$, an isomorphism
\begin{eqnarray*}
(\Omega^l_{/(O_K/p^n),\log}(P_{Z,X}))_{n\in\mathbb N}:
L_Z^{l,j-l}(X)\xrightarrow{\sim}L_Z^{l,j-l}(N_{Z/X})\xrightarrow{\sim}L^{l-d,j-l-d}(Z).
\end{eqnarray*}
Let 
\begin{eqnarray*}
\alpha=OL_{\hat X^{(p)}}(\alpha)
\in L^{l,j-l}(X):=H^jOL_{\hat X^{(p)}}(H_{pet}^{j-l}(X,\Omega^l_{\hat X^{(p)},\log,\mathcal O})) 
=H^{j-l}\Hom(\mathbb Z(X),E^{\bullet}_{pet}(r^*\Omega_{/K,\log,\mathcal O}^{l,an})).
\end{eqnarray*} 
Let $U\subset X$ be an affine open subset such that there exists an etale map $e:U\to T^{d_X}\subset\mathbb A^{d_X}$. 
Let $r=(r_i):Y:=\varprojlim_{i\in I}U_i\to U$ be a faithfully flat pro-etale map with $Y$ w-contractile. 
In particular $\bar r_i(Y^(0))=U$, where $Y^{(0)}\subset Y$ are the closed points.
As $\Omega_{/K,\log,\mathcal O}^{l,an}$ consists of single presheaf, we have 
\begin{equation*}
r^*\alpha=0\in H^j_{DR}(\hat Y^{(p)})=H^j\Omega^{\bullet}_{\hat X^{(p)}}(Y)
=\mathbb H_{pet}^j(Y,\Omega^{\bullet}_{\hat X^{(p)}}), 
\end{equation*}
that is
\begin{equation*}
r^*\alpha=0=[\partial(\eta_n)_{n\in\mathbb N}]=[(\partial\eta_n)_{n\in\mathbb N}]\in H^j_{DR}(\hat Y^{(p)}), 
\, \mbox{with} \,(\eta_n)_{n\in\mathbb N}\in\Omega^{j+1}_{\hat X^{(p)}}(Y). 
\end{equation*}
Denote $j:U\hookrightarrow X$ the open embedding. Consider 
\begin{equation*}
j^*\alpha=[(w_n)_{n\in\mathbb N}]\in H^j_{DR}(\hat U^{(p)})=H^j\Omega^{\bullet}_{\hat X^{(p)}}(U)
=\mathbb H_{pet}^j(U,\Omega^{\bullet}_{\hat X^{(p)}}).
\end{equation*}
Let $n\in\mathbb N$. There exists $i_n\in I$ (depending on $n$) such that $\eta_n=r_{i_n}^*\tilde\eta_n$,
with $\tilde\eta_n\in\Omega^{j+1}_{\hat X^{(p)}}(U_{i_n})$. 
Then, there exists $a_{ij}:U_{j_n}\to U_{i_n}$, $j_n,a_{ij}\in I$ such that
\begin{equation*}
r_{j_n}^*w_n=a_{ij}^*\partial\tilde\eta_n+\partial\beta_n\in\Omega^{\bullet}_{X^{\mathcal O}/p^n}(U_{j_n}) 
\end{equation*}
Since $r_{j_n}:U_{j_n}\to U$ is faithfully flat, we get  
\begin{equation*}
w_n=\partial(\tilde\eta_n+\beta_n)\in\Omega^{\bullet}_{X^{\mathcal O}/p^n}(U). 
\end{equation*}
Hence $j^*\alpha=0\in L^{l,j-l}(U)$.
Let $r_{\bullet}:Y_{\bullet}\to U$ be a pro-etale cover by $w$-contractile schemes. Then,
\begin{equation*}
j^*\alpha=(\theta_I)\in H_{pet}^{j-l}(U,\Omega^l_{\hat X^{(p)},\log,\mathcal O})
=H^{j-l}\Gamma(Y_{\bullet},\Omega^l_{\hat X^{(p)},\log,\mathcal O}).
\end{equation*}
Since $j^*\alpha=0\in L^{l,j-l}(U)$, we have 
\begin{equation*}
\theta_I=\tilde\theta_{I'|Y_I}+\partial\gamma_I\in\Gamma(Y_I,\Omega^l_{\hat X^{(p)}}),
\end{equation*}
with $I\subset I'$, $\tilde\theta_{I'}\in\Gamma(Y_{I'},\Omega^l_{\hat X^{(p)}})$ and 
$\gamma_I\in\Gamma(Y_I,\Omega^{l+1}_{\hat X^{(p)}})$.
Since $\Omega_U$ is a trivial vector bundle as $U\to\mathbb A^{d_X}$ is etale, 
and since a logarithmic form is exact if and only if it vanishes, we get $\theta_I=\tilde\theta_{I'|Y_I}$ and 
\begin{equation*}
j^*\alpha=\partial_{\bullet}(\tilde\theta_{I'})=0\in H_{pet}^{j-l}(U,\Omega^l_{\hat X^{(p)},\log,\mathcal O})
=H^{j-l}\Gamma(Y_{\bullet},\Omega^l_{\hat X^{(p)},\log,\mathcal O}).
\end{equation*}
Considering a divisor $X\backslash U\subset D\subset X$, we get 
\begin{equation*}
\alpha=H^{j-l}E^{\bullet}_{pet}(r^*\Omega_{/K,\log,\mathcal O}^{l,an})(\gamma^{\vee}_D)(\alpha), \, 
\alpha\in L_D^{l,j-l}(X).
\end{equation*}
We then get by induction (restricting to the smooth locus of the divisors)
a closed subset $Z\subset X$ of pure codimension $c=min(l,j-l)$ such that 
\begin{equation*}
\alpha=H^{j-l}E^{\bullet}_{pet}(r^*\Omega_{/K,\log,\mathcal O}^{l,an})(\gamma^{\vee}_Z)(\alpha), \, 
\alpha\in L_Z^{l,j-l}(X).
\end{equation*}
Hence, we get 
\begin{itemize}
\item if $j\neq 2l$, $\alpha=0$, we use the fact that $X$ is projective for $j<2l$, 
\item if $j=2l$, $\alpha\in\oplus_{1\leq t\leq s}\mathbb Z_p[Z_i]$,
where $(Z_i)_{1\leq i\leq t}\subset Z$ are the irreducible components of $Z$.
\end{itemize}

\noindent(ii): Take a compactification $\bar X\in\PSmVar(K)$ of $X$ with 
$\bar X\backslash X=\cup_{1\leq i\leq s}D_i\subset\bar X$ a normal crossing divisor.
Then (ii) follows from (i) applied to $\bar X$ and $D_I:=\cap_{i\in I} D_i$ for $I\subset[1,\ldots,s]$
by the distinguish triangle in $\DA(K)$
\begin{equation*}
M(X)\to M(\bar X)\to\oplus_{1\leq i\leq s}M(D_i)(-1)[-2]\to\cdots\to M(D_{[1,\cdots,s]})(-s)[-2s].
\end{equation*}
\end{proof}

\subsection{$p$ adic integral periods}

For $k$ a field of finite type over $\mathbb Q$ and $X\in\SmVar(k)$, 
we denote $\delta(k,X)\subset N$ the finite set consisting of prime numbers 
such that if $p\in\mathbb N\backslash\delta(k,X)$ is a prime number, $k$ is unramified at $p$ and 
there exists an integral model $X^{\mathcal O}_{\hat k_{\sigma_p}}\in\Sch^{int}/O_{\hat k_{\sigma_p}}$ 
of $X_{\hat k_{\sigma_p}}$ with good reduction modulo $p$ for all embeddings $\sigma_p:k\hookrightarrow\mathbb C_p$,
$\hat k_{\sigma_p}\subset\mathbb C_p$ being the $p$-adic completion of $k$ with respect to $\sigma_p$.

Let $k$ a field of finite type over $\mathbb Q$. 
Denote $\bar k$ the algebraic closure of $k$ and $G=Gal(\bar k/k)$ the absolute Galois group of $k$.
Let $X\in\SmVar(k)$ a smooth variety. 
Take a compactification $\bar X\in\PSmVar(k)$ of $X$ such that $D:=\bar X\backslash X\subset X$ is a normal crossing divisor,
and denote $j:X\hookrightarrow\bar X$ the open embedding.  
Let $p\in\mathbb N$ a prime number. Consider an embedding $\sigma_p:k\hookrightarrow\mathbb C_p$.
Then $k\subset\bar k\subset\mathbb C_p$ and $k\subset\hat k_{\sigma_p}\subset\mathbb C_p$,
where $\hat k_{\sigma_p}$ is the $p$-adic field 
which is the completion of $k$ with respect the $p$ adic norm given by $\sigma_p$.
Denote $\hat G_{\sigma_p}=Gal(\mathbb C_p/\hat k_{\sigma_p})=Gal(\bar{\mathbb Q_p}/\hat k_{\sigma_p})$ 
the Galois group of $\hat k_{\sigma_p}$.
Recall (see section 2) that 
$\underline{\mathbb Z_p}_{X_{\mathbb C_p}}:=
\varprojlim_{n\in\mathbb N}\nu_X^*(\mathbb Z/p^n\mathbb Z)_{X_{\mathbb C_p}^{et}}\in\Shv(X_{\mathbb C_p}^{pet})$ and 
$\Omega^{\bullet}_{X^{pet}_{\mathbb C_p},\log,\mathcal O}:=
\varprojlim_{n\in\mathbb N}\nu_X^*\Omega^{\bullet}_{X^{pet}_{O_{\mathbb C_p}/p^nO_{\mathbb C_p}},\log}
\in C(X_{\mathbb C_p}^{pet})$.
We have then the commutative diagram in $C_{\mathbb B_{dr}fil,\hat G_{\sigma_p}}(\bar X_{\mathbb C_p}^{an,pet})$
\begin{equation*}
\xymatrix{j_*E_{pet}(\mathbb B_{dr,X_{\mathbb C_p}},F)\ar[rrr]^{j_*E_{pet}(\alpha(X))} & \, & \, &
E_{pet}((\Omega^{\bullet}_{X_{\mathbb C_p}}(\log D_{\mathbb C_p}),F_b)
\otimes_{O_{\bar X_{\mathbb C_p}}}(O\mathbb B_{dr,\bar X_{\mathbb C_p},\log},F)) \\
j_*E_{pet}(\underline{\mathbb Z_p}_{X_{\mathbb C_p}})
\ar[u]^{E_{pet}(j_*\iota'_{X_{\mathbb C_p}^{pet}})_j:=E_{pet}(l\mapsto l.1)_j}
\ar[rrr]^{j_*E_{pet}(\iota_{X_{\mathbb C_p}^{pet}})} & \, & \, &
j_*E_{pet}(\Omega^{\bullet}_{X_{\mathbb C_p},\log,\mathcal O}\otimes\underline{\mathbb Z_p}_{X_{\mathbb C_p}},F_b)
\ar[u]_{E_{pet}(m\circ (OL_X\otimes I)):=E_{pet}((w\otimes\lambda)\mapsto (w\otimes\lambda))}},
\end{equation*}
where for $j':U'\hookrightarrow X'$ an open embedding with $X'\in\RTop$ and $\tau$ a topology on $\RTop$
we denote for $m:j_*Q\to Q'$ with $Q\in\PSh_{O}(U')$, $Q'\in\PSh_O(X')$ the canonical map in $C_O(X')$
\begin{equation*}
E_{\tau}^0(m)_j:j_*E^0_{\tau}(Q)\to E^0_{\tau}(j_*Q)\xrightarrow{E_{\tau}^0(m)}E^0_{\tau}(Q'), 
\end{equation*}
giving by induction the canonical map $E_{\tau}(m)_j:j_*E_{\tau}(Q)\to E_{\tau}(Q')$ in $C_O(X')$.
The main results of \cite{Chinois} state that 
\begin{itemize}
\item the map in $C_{\mathbb B_{dr}fil}(\bar X_{\hat k_{\sigma_p}}^{an,pet})$
\begin{eqnarray*}
\alpha(X):(\mathbb B_{dr,\bar X_{\hat k_{\sigma_p}},\log D_{\hat k_{\sigma_p}}},F)\hookrightarrow
(\Omega^{\bullet}_{\bar X_{\hat k_{\sigma_p}}}(\log D_{\hat k_{\sigma_p}}),F_b)\otimes_{O_{\bar X_{\hat k_{\sigma_p}}}}
(O\mathbb B_{dr,\bar X_{\hat k_{\sigma_p}},\log D_{\hat k_{\sigma_p}}},F)
\end{eqnarray*}
is a filtered quasi-isomorphism, that is, 
the induced map in $C_{\mathbb B_{dr}fil,\hat G_{\sigma_p}}(\bar X_{\mathbb C_p}^{an,pet})$
\begin{eqnarray*}
\alpha(X):=\alpha(X)_{\mathbb C_p}:(\mathbb B_{dr,\bar X_{\mathbb C_p},\log D_{\mathbb C_p}},F)\hookrightarrow
(\Omega^{\bullet}_{\bar X_{\mathbb C_p}}(\log D_{\mathbb C_p}),F_b)\otimes_{O_{\bar X_{\mathbb C_p}}}
(O\mathbb B_{dr,\bar X_{\mathbb C_p},\log D_{\mathbb C_p}},F)
\end{eqnarray*}
is thus a filtered quasi-isomorphism, 
\item the map in $D_{\mathbb Z_pfil}$
\begin{equation*}
T(a_X,a_X,\otimes)(Rj_*\mathbb Z_{p,X^{et}}):
R\Gamma(X_{\mathbb C_p},\mathbb Z_{p,X^{et}})\otimes_{\mathbb Z_p}(\mathbb B_{dr,\mathbb C_p},F)
\to R\Gamma(\bar X_{\mathbb C_p},(\mathbb B_{dr,\bar X_{\mathbb C_p},\log D_{\mathbb C_p}},F))
\end{equation*}
is an isomorphism.
\end{itemize}
Hence, we get the isomorphism in $D_{fil}(\mathbb B_{dr},\hat G_{\sigma_p})$
\begin{eqnarray*}
R\alpha(X):=R\Gamma(\bar X_{\mathbb C_p},\alpha(X))\circ T(a_X,a_X,\otimes)(Rj_*\mathbb Z_{p,X^{et}}): \\
R\Gamma(X_{\mathbb C_p},\mathbb Z_{p,X^{et}})\otimes_{\mathbb Z_p}(\mathbb B_{dr,\mathbb C_p},F)
\xrightarrow{\sim} 
R\Gamma(\bar X_{\mathbb C_p},
(\Omega^{\bullet}_{\bar X^{et}_{\mathbb C_p}}(\log D_{\mathbb C_p}),F_b)\otimes_{O_{\bar X_{\mathbb C_p}}}
(O\mathbb B_{dr,\bar X_{\mathbb C_p},\log D_{\mathbb C_p}},F)) \\
\xrightarrow{=} 
R\Gamma(\bar X_{\mathbb C_p},
(\Omega^{\bullet}_{\bar X^{et}_{\mathbb C_p}},F_b)\otimes_{O_{\bar X_{\mathbb C_p}}}
j_{*Hdg}(O_{X_{\mathbb C_p}},F_b)\otimes_{O_{X_{\mathbb C_p}}}(O\mathbb B_{dr,X_{\mathbb C_p}},F))
\end{eqnarray*}
which gives for each $n\in\mathbb Z$ a filtered isomorphism of $\hat G_{\sigma_p}$-modules
\begin{eqnarray*}
H^nR\alpha(X):
H^n_{et}(X_{\mathbb C_p},\mathbb Z_{p,X^{et}})\otimes\mathbb B_{dr,\mathbb C_p}\xrightarrow{\sim}
H^n_{DR}(X_{\hat k_{\sigma_p}})\otimes_{\hat k_{\sigma_p}}\mathbb B_{dr,\mathbb C_p}
\end{eqnarray*}
so that we can recover the Hodge filtration on $H^*_{DR}(X)$ by the action of $\hat G_{\sigma_p}$.
Let $p\in\mathbb N\in\delta(k,X)$.
Take a compactification $\bar X\in\PSmVar(k)$ of $X$ such that $D:=\bar X\backslash X=\cup_iD_i\subset X$ 
is a normal crossing divisor such that $D_i$ admit an integral model.
The main result of \cite{ChinoisCrys} say that the embedding in $C((\bar X^{\mathcal O}_{\hat k_{\sigma_p}})^{Falt})$
\begin{equation*}
\alpha(X):\mathbb B_{st,\bar X_{\hat k_{\sigma_p}},\log}\hookrightarrow 
a_{\bullet*}\Omega^{\bullet}_{X^{\mathcal O,\bullet}_{\hat k_{\sigma_p}}}(\log D^{\mathcal O}_{\hat k_{\sigma_p}})
\otimes_{O_{X^{\mathcal O}}}O\mathbb B_{st,\bar X^{\bullet}_{\hat k_{\sigma_p}},\log D_{\hat k_{\sigma_p}}}
\end{equation*}
is a filtered quasi-isomorphism compatible with the action of the Frobenius $\phi_p$ and the monodromy $N$,
note that we have a commutative diagram in $C_{fil}(X^{an,pet}_{\hat k_{\sigma_p}})$
\begin{equation*}
\xymatrix{\mathbb B_{st,\bar X_{\hat k_{\sigma_p}},\log D_{\hat k_{\sigma_p}}}\ar[rr]^{\alpha(X)}\ar[d]^{\subset} & \, &
O\mathbb B_{st,X_{\hat k_{\sigma_p}},\log D_{\hat k_{\sigma_p}}}\otimes_{O_X}
\Omega_{X_{\hat k_{\sigma_p}}}^{\bullet}(\log D_{\hat k_{\sigma_p}})
\ar[d]^{\subset} \\
\mathbb B_{dr,X_{\mathbb C_p},\log D_{\mathbb C_p}}\ar[rr]^{\alpha(X)} & \, &
O\mathbb B_{dr,\bar X_{\mathbb C_p},\log D_{\mathbb C_p}}\otimes_{O_X}
\Omega_{X_{\mathbb C_p}}^{\bullet}(\log D_{\mathbb C_p})}.
\end{equation*}
This gives if $(X^{\mathcal O}_{\hat k_{\sigma_p}},N_{U,\mathcal O})$ is log smooth, 
for each $j\in\mathbb Z$, a filtered isomorphism of filtered abelian groups
\begin{eqnarray*}
H^jR\alpha(X):H_{et}^j(X_{\mathbb C_p},\mathbb Z_p)\otimes_{\mathbb Z_p}\mathbb B_{st,\hat k_{\sigma_p}}
\xrightarrow{H^jT(a_X,\mathbb B_{st})^{-1}}
H_{et}^j((X,N)^{Falt})(\mathbb B_{st,\bar X_{\hat k_{\sigma_p}},\log D_{\hat k_{\sigma_p}}}) \\
\xrightarrow{H^jR\Gamma((\bar X_{\hat k_{\sigma_p}}^{\mathcal O},D_{\hat k_{\sigma_p}}),\alpha(X))}
H^j_{DR}(X_{\hat k_{\sigma_p}})\otimes_{\hat k_{\sigma_p}}\mathbb B_{st,\hat k_{\sigma_p}}
\end{eqnarray*}
compatible with the action of $Gal(\mathbb C_p/\hat k_{\sigma_p})$, of the Frobenius $\phi_p$ and the monodromy $N$.

\begin{defi}\label{walpha}
Let $k$ a field of finite type over $\mathbb Q$. 
Denote $\bar k$ the algebraic closure of $k$ and $G=Gal(\bar k/k)$ the absolute Galois group of $k$.
Let $X\in\SmVar(k)$ a smooth variety. 
Take a compactification $\bar X\in\PSmVar(k)$ of $X$ such that $D:=\bar X\backslash X\subset X$ is a normal crossing divisor,
and denote $j:X\hookrightarrow\bar X$ the open embedding.  
Let $p\in\mathbb N$ a prime number. 
Consider an embedding $\sigma_p:k\hookrightarrow\mathbb C_p$.
Then $k\subset\bar k\subset\mathbb C_p$ and $k\subset\hat k_{\sigma_p}\subset\mathbb C_p$,
where $\hat k_{\sigma_p}$ is the $p$-adic field 
which is the completion of $k$ with respect the $p$ adic norm given by $\sigma_p$.
For $\alpha\in H^j_{et}(X_{\mathbb C_p},\mathbb Z_p)$, we denote
\begin{eqnarray*}
w(\alpha):=H^jR\alpha(X)(\alpha\otimes 1)\in 
H^j_{DR}(X_{\hat k_{\sigma_p}})\otimes_{\hat k_{\sigma_p}}\mathbb B_{dr,\mathbb C_p}
\end{eqnarray*}
and if $p\in\mathbb N\backslash\delta(k,X)$
\begin{eqnarray*}
w(\alpha):=H^jR\alpha(X)(\alpha\otimes 1)\in
H^j_{DR}(X_{\hat k_{\sigma_p}})\otimes_{\hat k_{\sigma_p}}\mathbb B_{st,\hat k_{\sigma_p}}.
\end{eqnarray*}
the associated de Rham class by the $p$ adic periods. We recall
\begin{eqnarray*}
H^jR\alpha(X):
H^j_{et}(X_{\mathbb C_p},\mathbb Z_{p,X^{et}})\otimes\mathbb B_{dr,\mathbb C_p}\xrightarrow{\sim}
H^j_{DR}(X_{\hat k_{\sigma_p}})\otimes_{\hat k_{\sigma_p}}\mathbb B_{dr,\mathbb C_p}
\end{eqnarray*}
is the canonical filtered isomorphism of $\hat G_{\sigma_p}$-modules, and
\begin{eqnarray*}
H^jR\alpha(X):H_{et}^j(X_{\mathbb C_p},\mathbb Z_p)\otimes_{\mathbb Z_p}\mathbb B_{st,\hat k_{\sigma_p}}
\xrightarrow{\sim}H^j_{DR}(X_{\hat k_{\sigma_p}})\otimes_{\hat k_{\sigma_p}}\mathbb B_{st,\hat k_{\sigma_p}}
\end{eqnarray*}
is the canonical filtered isomorphism compatible with the action of $\hat G_{\sigma_p}$, 
of the Frobenius $\phi_p$ and the monodromy $N$.
\end{defi}

We recall the following result from Illusie:

\begin{prop}\label{Ilprop}
Let $k$ a field of finite type over $\mathbb Q$. Let $X\in\SmVar(k)$.  
Let $p\in\mathbb N\backslash\delta(k,X)$ a prime number. Consider an embedding $\sigma_p:k\hookrightarrow\mathbb C_p$.
Denote $\hat k_{\sigma_p}\subset\mathbb C_p$ the $p$-adic completion of $k$ with respect to $\sigma_p$.
Consider $X_{\hat k_{\sigma_p}}^{\mathcal O}\in\Sch^{int}/O_{\hat k_{\sigma_p}}$ a smooth model of $X_{\hat k_{\sigma_p}}$,
in particular $X_{\hat k_{\sigma_p}}^{\mathcal O}\otimes_{O_{\hat k_{\sigma_p}}}\hat k_{\sigma_p}=X_{\hat k_{\sigma_p}}$
and $X_{\hat k_{\sigma_p}}^{\mathcal O}$ is smooth with smooth special fiber.
Assume there exist lifts 
$\phi_n:X^{\mathcal O}_{\hat k_{\sigma_p}}/p^n\to X^{\mathcal O}_{\hat k_{\sigma_p}}/p^n$ 
of the Frobenius $\phi:X^{\mathcal O}_{\hat k_{\sigma_p}}/p\to X^{\mathcal O}_{\hat k_{\sigma_p}}/p$,
such that for $n'>n$ the following diagram commutes
\begin{equation*}
\xymatrix{0\ar[r] & O_{X_{\hat k_{\sigma_p}}^{\mathcal O}/p^{n'-n}}\ar[r]^{p^n\cdot} & 
O_{X_{\hat k_{\sigma_p}}^{\mathcal O}/p^{n'}}\ar[r]^{/p^{n'-n}} &
O_{X_{\hat k_{\sigma_p}}^{\mathcal O}/p^n}\ar[r] & 0 \\
0\ar[r] & O_{X_{\hat k_{\sigma_p}}^{\mathcal O}/p^{n'-n}}\ar[r]^{p^n\cdot}\ar[u]^{\phi_{n'-n}} & 
O_{X_{\hat k_{\sigma_p}}^{\mathcal O}/p^{n'}}\ar[r]^{/p^{n'-n}}\ar[u]^{\phi_{n'}} &
O_{X_{\hat k_{\sigma_p}}^{\mathcal O}/p^n}\ar[u]^{\phi_n}\ar[r] & 0}
\end{equation*}
Let $l\in\mathbb Z$. For each $n\in\mathbb N$, the sequence in $C(X^{\mathcal O,et}_{\hat k_{\sigma_p}})$
\begin{eqnarray*}
0\to \Omega^{\bullet\geq l}_{X^{\mathcal O}_{\hat k_{\sigma_p}}/p^n,\log} 
\xrightarrow{OL_{X^{\mathcal O}/p^n}}
\Omega^{\bullet\geq l}_{X^{\mathcal O}_{\hat k_{\sigma_p}}/p^n}
\xrightarrow{\phi_n-I}
\Omega^{\bullet\geq l}_{X^{\mathcal O}_{\hat k_{\sigma_p}}/p^n}\to 0
\end{eqnarray*}
is exact as a sequence of etale sheaves (i.e. we only have local surjectivity on the right).
\end{prop}

\begin{proof}
It follows from \cite{Illusie} for $n=1$. 
It then follows for $n\geq 2$ by induction on $n$ by a trivial devissage.
\end{proof}

We have the following key proposition :

\begin{prop}\label{LatticeLog}
Let $k$ a field of finite type over $\mathbb Q$. Let $X\in\SmVar(k)$. 
Let $p\in\mathbb N\backslash\delta(k,X)$ a prime number. Consider an embedding $\sigma_p:k\hookrightarrow\mathbb C_p$.
Denote $k\subset\hat k_{\sigma_p}\subset\mathbb C_p$ the $p$-adic completion of $k$ with respect to $\sigma_p$.
Consider an integral model $(\bar X,D)_{\hat k_{\sigma_p}}^{\mathcal O})$ an integral model of a compactification
$\bar X\in\PSmVar(k)$ of $X$ with $D:=\bar X\backslash X\subset\bar X$ a normal crossing divisor, i.e. 
\begin{itemize}
\item $\bar X_{\hat k_{\sigma_p}}^{\mathcal O}\in\PSch/O_K$ and $(\bar X,D)_{\hat k_{\sigma_p}}^{\mathcal O}$ is log smooth pair
(in particular $\bar X^{\mathcal O}$ is smooth with smooth special fiber and 
$D^{\mathcal O}\subset X^{\mathcal O}$ is a normal crossing divisor), 
\item $(\bar X,D)_{\hat k_{\sigma_p}}^{\mathcal O}\otimes_{O_{\hat k_{\sigma_p}}}\hat k_{\sigma_p}=(\bar X,D)_{\hat k_{\sigma_p}}$.
\end{itemize}
so that $X_{\hat k_{\sigma_p}}^{\mathcal O}\in\Sch^{int}/O_{\hat k_{\sigma_p}}$ an integral model of $X_{\hat k_{\sigma_p}}$.
\begin{itemize}
\item[(i)]Let $j,l\in\mathbb Z$. We have, see definition \ref{wlogdef}(iii), 
\begin{eqnarray*}
F^lH^j_{DR}(X_{\hat k_{\sigma_p}})\cap H^jR\alpha(X)(H^j_{et}(X_{\mathbb C_p},\mathbb Z_p))
=(c^*)^{-1}(H^jOL_{\hat X_{\hat k_{\sigma_p}}^{(p)}}
(\mathbb H^j_{pet}(X,\Omega^{\bullet\geq l}_{\hat X^{(p)}_{\hat k_{\sigma_p}},\log,\mathcal O}))) \\
\subset H^j_{DR}(X_{\hat k_{\sigma_p}})\otimes_{\hat k_{\sigma_p}}\mathbb B_{st,\hat k_{\sigma_p}},
\end{eqnarray*}
where we recall $c:\hat{\bar X}_{\hat k_{\sigma_p}}^{(p)}\to\bar X_{\hat k_{\sigma_p}}$ 
is the completion with respect to $(p)$, and 
\begin{equation*}
c^*:\mathbb H^j_{pet}(X_{\hat k_{\sigma_p}},\Omega_{X_{\hat k_{\sigma_p}}}^{\bullet})\xrightarrow{\sim}
\mathbb H^j_{pet}(X_{\hat k_{\sigma_p}},\Omega^{\bullet}_{\hat X^{(p)}_{\hat k_{\sigma_p}}}) 
\end{equation*}
is an isomorphism by GAGA and by considering an open cover $X_{\hat k_{\sigma_p}}=\cup_iX_i$
such that we have etale maps $X_i\to T^{d_X}\subset\mathbb A^{d_X}$. 
\item[(ii)]Let $j,l\in\mathbb Z$. We have, see definition \ref{wlogdef}(iii),
\begin{eqnarray*}
F^lH^j_{DR}(X_{\hat k_{\sigma_p}})\cap H^jR\alpha(X)(H^j_{et}(X_{\mathbb C_p},\mathbb Z_p))
=H^j(m\circ(OL_X\otimes I))
(\mathbb H^j_{pet}(X,\Omega^{\bullet\geq l}_{X_{\hat k_{\sigma_p}},\log,\mathcal O}\otimes\mathbb Z_p)) \\
\subset H^j_{DR}(X_{\hat k_{\sigma_p}})\otimes_{\hat k_{\sigma_p}}\mathbb B_{st,\hat k_{\sigma_p}},
\end{eqnarray*} 
\item[(ii)'] For $\alpha\in H^j_{et}(X_{\mathbb C_p},\mathbb Z_p)$ 
such that $w(\alpha)\in F^lH^j_{DR}(X_{\hat k_{\sigma_p}})$ (see definition \ref{walpha}), there exist 
\begin{equation*}
(\lambda_i)_{1\leq i\leq n}\in\mathbb Z_p \; \mbox{and} \;
(w_{Li})_{1\leq i\leq n}\in\mathbb H_{pet}^j(X_{\mathbb C_p},\Omega^{\bullet\geq l}_{X_{\mathbb C_p},\log,\mathcal O})
\end{equation*}
such that 
\begin{equation*}
w(\alpha)=\sum_{1\leq i\leq n}\lambda_i\cdot w_{Li}\in
\mathbb H_{pet}^j(X_{\mathbb C_p},\Omega^{\bullet\geq l}_{X_{\mathbb C_p}})=F^lH^j_{DR}(X_{\mathbb C_p}),
\; w_{Li}:=H^jOL_X(w_{Li}).
\end{equation*}
\end{itemize}
\end{prop}

\begin{proof} 
\noindent(i): Consider 
$c:(\hat\bar X^{\mathcal O,(p)}_{\hat k_{\sigma_p}},\hat D^{\mathcal O,(p)}_{\hat k_{\sigma_p}})\to 
(\bar X_{\hat k_{\sigma_p}}^{\mathcal O},D^{\mathcal O}_{\hat k_{\sigma_p}})$ 
the morphism in $\RTop$ which is the formal completion along the ideal $(p)$.
Take a Zariski or etale cover $r=(r_i:X_i\to X)_{1\leq i\leq r}$ such that for each $i$ there exists an etale map
$X_i\to\mathbb A^{d_{X_i}}$. Then, by \cite{ChinoisCrys}, we have for each $i$ explicit lifts of Frobenius
$\phi^i_n:X^{\mathcal O}_{i,\hat k_{\sigma_p}}/p^n\to X^{\mathcal O}_{i,\hat k_{\sigma_p}}/p^n$ 
of the Frobenius $\phi_1^i:X^{\mathcal O}_{i,\hat k_{\sigma_p}}/p\to X^{\mathcal O}_{i,\hat k_{\sigma_p}}/p$,
such that for $n'>n$ the following diagram commutes
\begin{equation*}
\xymatrix{0\ar[r] & O_{X_{i,\hat k_{\sigma_p}}^{\mathcal O}/p^{n'-n}}\ar[r]^{p^n\cdot} & 
O_{X_{i,\hat k_{\sigma_p}}^{\mathcal O}/p^{n'}}\ar[r]^{/p^{n'-n}} &
O_{X_{i,\hat k_{\sigma_p}}^{\mathcal O}/p^n}\ar[r] & 0 \\
0\ar[r] & O_{X_{i,\hat k_{\sigma_p}}^{\mathcal O}/p^{n'-n}}\ar[r]^{p^n\cdot}\ar[u]^{\phi^i_{n'-n}} & 
O_{X_{i,\hat k_{\sigma_p}}^{\mathcal O}/p^{n'}}\ar[r]^{/p^{n'-n}}\ar[u]^{\phi^i_{n'}} &
O_{X_{i,\hat k_{\sigma_p}}^{\mathcal O}/p^n}\ar[u]^{\phi^i_n}\ar[r] & 0}
\end{equation*}
and such that the action of $\phi^i_n$ on $\Omega^{\bullet}_{X_{i,\hat k_{\sigma_p}}^{\mathcal O}/p^n}$ 
is a morphism of complex, i.e. commutes with the differentials. 
On the other hand, by \cite{Illusie}, we have action of the Frobenius on 
$H^j_{DR}(X_{\hat k_{\sigma_p}})=H^j_{DR}(X^{\mathcal O}_{\hat k_{\sigma_p}})\otimes_{O_{\hat k_{\sigma_p}}}\hat k_{\sigma_p}$
by
\begin{equation*}
\phi:H^j_{DR}(X^{\mathcal O}_{\hat k_{\sigma_p}})\xrightarrow{\sim}
\mathbb H^j(X_{\hat k_{\sigma_p}},W\Omega^{\bullet}_{X^{\mathcal O}_{\hat k_{\sigma_p}}})
\xrightarrow{\phi_{W(X_{\hat k_{\sigma_p}})}^*}
\mathbb H^j(X_{\hat k_{\sigma_p}},W\Omega^{\bullet}_{X^{\mathcal O}_{\hat k_{\sigma_p}}})
\xrightarrow{\sim}H^j_{DR}(X^{\mathcal O}_{\hat k_{\sigma_p}})
\end{equation*}
We then have the following commutative diagram, where $R:=[1,\ldots,r]$ and $X_R:=X_1\times_X\cdots\times_X X_r$,
\begin{equation}\label{Frob}
\xymatrix{\cdots\ar[r] & F^lH^{j-1}_{DR}(X^{\mathcal O}_{R,\hat k_{\sigma_p}})\ar[r]^{\partial} & 
F^lH^j_{DR}(X^{\mathcal O}_{\hat k_{\sigma_p}})\ar[r]^{r_i^*} &
\oplus_{i=1}^rF^lH^j_{DR}(X^{\mathcal O}_{i,\hat k_{\sigma_p}})\ar[r]^{r_I^*} & \cdots \\
\cdots\ar[r] & F^lH^{j-1}_{DR}(X^{\mathcal O}_{R,\hat k_{\sigma_p}})\ar[r]^{\partial}\ar[u]^{I-\phi^R} & 
F^lH^j_{DR}(X^{\mathcal O}_{\hat k_{\sigma_p}})\ar[r]^{r_i^*}\ar[u]^{I-\phi} &
\oplus_{i=1}^rF^lH^j_{DR}(X^{\mathcal O}_{i,\hat k_{\sigma_p}})\ar[r]^{r_I^*}\ar[u]^{I-\phi^i} & \cdots \\
\cdots\ar[r] & 
\mathbb H_{pet}^{j-1}(X_{R,\hat k_{\sigma_p}},\Omega^{\bullet\geq l}_{\hat X^{(p)}_{\hat k_{\sigma_p}},\log,\mathcal O})
\ar[r]^{\partial}\ar[u]^{H^{j-1}r^*OL_{\hat X^{\mathcal O,(p)}_{R,\hat k_{\sigma_p}}}} &
\mathbb H_{pet}^j(X_{\hat k_{\sigma_p}},\Omega^{\bullet\geq l}_{\hat X^{(p)}_{\hat k_{\sigma_p}},\log,\mathcal O})
\ar[r]^{r_i^*}\ar[u]^{H^jr^*OL_{\hat X^{\mathcal O,(p)}_{i,\hat k_{\sigma_p}}}} &
\oplus_{i=1}^r\mathbb H_{pet}^j(X_{i,\hat k_{\sigma_p}},
\Omega^{\bullet\geq l}_{\hat X^{(p)}_{\hat k_{\sigma_p}},\log,\mathcal O})
\ar[r]^{r_I^*}\ar[u]^{H^jr^*OL_{\hat X_{i,\hat k_{\sigma_p}}^{\mathcal O,(p)}}} & \cdots}
\end{equation}
whose rows are exact sequences.
By \cite{ChinoisCrys},
$\alpha\in H^j_{et}(X_{\mathbb C_p},\mathbb Z_p)$ is such that $w(\alpha)\in F^lH^j_{DR}(X_{\hat k_{\sigma_p}})$
if and only if  
\begin{equation*}
w(\alpha)\in\ker(I-\phi:F^lH^j_{DR}(X^{\mathcal O}_{\hat k_{\sigma_p}})\to F^lH^j_{DR}(X^{\mathcal O}_{\hat k_{\sigma_p}})).
\end{equation*}  
On the other hand, for each $I\subset[1,\ldots,r]$, the sequence in $C(X_I^{\mathcal O,pet})$
\begin{eqnarray*}
0\to\Omega^{\bullet\geq l}_{\hat X^{(p)}_{I,\hat k_{\sigma_p}},\log,\mathcal O}
\xrightarrow{OL_{\hat X_{\hat k_{\sigma_p}}^{\mathcal O,(p)}}}
\Omega^{\bullet\geq l}_{\hat X^{\mathcal O,(p)}_{I,\hat k_{\sigma_p}}}(\log D^{\mathcal O}_{I,\hat k_{\sigma_p}}) 
\xrightarrow{\phi^I-I:=(\phi^I_n-I)_{n\in\mathbb N}}
\Omega^{\bullet\geq l}_{\hat X^{\mathcal O,(p)}_{I,\hat k_{\sigma_p}}}(\log D^{\mathcal O}_{I,\hat k_{\sigma_p}})\to 0
\end{eqnarray*}
is exact for the pro-etale topology by proposition \ref{Ilprop} 
and since for each $l,n\in\mathbb N$ the maps in $\PSh(X^{pet})$
\begin{equation*}
\Omega(/p^{n'}):\Omega^l_{X^{\mathcal O}_{\hat k_{\sigma_p}}/p^{n-n'}}\to\Omega^l_{X^{\mathcal O}_{\hat k_{\sigma_p}}/p^n} 
\end{equation*}
are surjective for the etale topology and since the pro-etale site is a replete topos (\cite{BSch}). 
Hence, for each $I\subset[1,\ldots,r]$, by applying $r^*$ where $r:X_I^{pet}\to X_I^{\mathcal O,pet}$,
the sequence in $C(X_I^{pet})$
\begin{eqnarray*}
0\to r^*\Omega^{\bullet\geq l}_{\hat X^{(p)}_{I,\hat k_{\sigma_p}},\log,\mathcal O}
\xrightarrow{r^*OL_{\hat X_{\hat k_{\sigma_p}}^{\mathcal O,(p)}}}
r^*\Omega^{\bullet\geq l}_{\hat X^{\mathcal O,(p)}_{I,\hat k_{\sigma_p}}}(\log D^{\mathcal O}_{I,\hat k_{\sigma_p}}) 
\xrightarrow{\phi^I-I:=(\phi^I_n-I)_{n\in\mathbb N}}
r^*\Omega^{\bullet\geq l}_{\hat X^{\mathcal O,(p)}_{I,\hat k_{\sigma_p}}}(\log D^{\mathcal O}_{I,\hat k_{\sigma_p}})\to 0
\end{eqnarray*}
is exact for the pro-etale topology.
Hence the columns of the diagram (\ref{Frob}) are exact. This proves (i).

\noindent(ii): Follows from (i) and proposition \ref{GAGAlog}.

\noindent(ii)': By (ii), if $(r_i:X_i\to X)_{1\leq i\leq r}$ is a w-contractile affine pro-etale cover 
\begin{equation*}
w(\alpha)=(\sum_{1\leq i\leq n_J}\lambda_{iJ}\cdot w_{LiJ})_{J\subset[1,\ldots,r],card J=j}
\in H^j\Gamma(X_{\bullet,\hat k_{\sigma_p}},m\circ(OL_X\otimes I))(H^j\Gamma(X_{\bullet,\hat k_{\sigma_p}},
\Omega^{\bullet\geq l}_{X_{\hat k_{\sigma_p}},\log,\mathcal O}\otimes\mathbb Z_p)),
\end{equation*}
with $w_{LiJ}\in\Gamma(X_J,\Omega^{\bullet\geq l'}_{X_{\hat k_{\sigma_p}},\log,\mathcal O})$, $l'\geq l$
and $\lambda_{iJ}\in\mathbb Z_p$. Since $w(\alpha)$ is a Chech etale cycle (i.e. closed for the Chech differential)
without torsion and since the topology consists of etale covers of $X\in\PSmVar(k)$ (which we way assume connected), 
$n_J=n$ and $\lambda_{iJ}=\lambda_i\in\mathbb Z_p$ for each $J$, which gives
\begin{equation*}
w(\alpha)=\sum_{1\leq i\leq n}\lambda_i\cdot w_{Li}\in
\mathbb H_{pet}^j(X_{\mathbb C_p},\Omega^{\bullet\geq l}_{X_{\mathbb C_p}})=F^lH^j_{DR}(X_{\mathbb C_p}),
\; w_{Li}:=H^jOL_X(w_{Li}).
\end{equation*}
\end{proof}

\begin{prop}\label{keypropCp}
Let $k$ a field of finite type over $\mathbb Q$. Denote $\bar k$ the algebraic closure of $k$.
Denote $G:=Gal(\bar k/k)$ its absolute Galois group.
Let $X\in\SmVar(k)$ a smooth variety. 
Take a compactification $\bar X\in\PSmVar(k)$ of $X$ such that $D:=\bar X\backslash X\subset X$ is a normal crossing divisor,
and denote $j:X\hookrightarrow\bar X$ the open embedding.    
Let $p\in\mathbb N$ a prime number. Consider an embedding $\sigma_p:k\hookrightarrow\mathbb C_p$.
Denote $k\subset\hat k_{\sigma_p}\subset\mathbb C_p$ being the $p$-adic competion
with respect to the $p$ adic norm induced by $\sigma_p$. 
Then $\hat G_{\sigma_p}:=Gal(\mathbb C_p/\hat k_{\sigma_p})\subset G:=Gal(\bar k/k)$.
\begin{itemize}
\item[(i)]Let $\alpha\in H_{et}^j(X_{\mathbb C_p},\mathbb Z_p)$. 
Consider then its associated De Rham class (see definition \ref{walpha})
\begin{eqnarray*}
w(\alpha):=H^jR\alpha(X)(\alpha\otimes 1)\in H^j_{DR}(X_{\mathbb C_p})\otimes_{\mathbb C_p}\mathbb B_{dr,\mathbb C_p}.
\end{eqnarray*}
Then $\alpha\in H_{et}^j(X_{\mathbb C_p},\mathbb Z_p)(l)^{\hat G_{\sigma_p}}$ if and only if
$w(\alpha)\in F^lH^j_{DR}(X_{\hat k_{\sigma_p}})=F^lH^j_{DR}(X_{\mathbb C_p})\cap H^j_{DR}(X_{\hat k_{\sigma_p}})$.
That is we have 
\begin{equation*}
H_{et}^j(X_{\mathbb C_p},\mathbb Z_p)(l)^{\hat G_{\sigma_p}}\otimes_{\mathbb Z_p}\mathbb Q_p
=<F^lH^j_{DR}(X_{\hat k_{\sigma_p}})\cap H_{et}^j(X_{\mathbb C_p},\mathbb Z_p)>_{\mathbb Q_p}\subset
H^j_{et}(X_{\mathbb C_p},\mathbb Z_{p,X^{et}})\otimes\mathbb B_{dr,\mathbb C_p},
\end{equation*}
where $<->_{\mathbb Q_p}$ denote the $\mathbb Q_p$ vector space generated by $(-)$.
Note that $H_{et}^j(X_{\mathbb C_p},\mathbb Z_p)$ and $F^lH^j_{DR}(X_{\hat k_{\sigma_p}})$ 
are canonically embedded as subabelian groups of 
$H^j_{et}(X_{\mathbb C_p},\mathbb Z_{p,X^{et}})\otimes\mathbb B_{dr,\mathbb C_p}$ 
by $(-)\otimes 1$ and $\alpha(X)\circ((-)\otimes 1)$ respectively.
\item[(ii)] Let $\alpha\in H_{et}^j(X_{\bar k},\mathbb Z_p)(l)$. Consider, see (i),
\begin{equation*}
w(\alpha):=w(\pi_{\bar k/\mathbb C_p}(X)^*\alpha)\in H^j_{DR}(X_{\mathbb C_p})
\otimes_{\mathbb C_p}\mathbb B_{dr,\mathbb C_p}.
\end{equation*}
Then $\alpha\in H_{et}^j(X_{\bar k},\mathbb Z_p)(l)^G$ if and only if
$w(\alpha)\in F^lH^j_{DR}(X)=F^lH^j_{DR}(X_{\mathbb C_p})\cap H^j_{DR}(X)$.
That is we have 
\begin{equation*}
H_{et}^j(X_{\bar k},\mathbb Z_p)(l)^G\otimes_{\mathbb Z_p}\mathbb Q_p
=<F^lH^j_{DR}(X)\cap H_{et}^j(X_{\bar k},\mathbb Z_p)>_{\mathbb Q_p}\subset
H^j_{et}(X_{\mathbb C_p},\mathbb Z_{p,X^{et}})\otimes\mathbb B_{dr,\mathbb C_p},
\end{equation*}
where $<->_{\mathbb Q_p}$ denote the $\mathbb Q_p$ vector space generated by $(-)$. 
Note that $H_{et}^j(X_{\bar k},\mathbb Z_p)$ and $F^lH^j_{DR}(X)$ are canonically embedded as subabelian groups of 
$H^j_{et}(X_{\mathbb C_p},\mathbb Z_{p,X^{et}})\otimes\mathbb B_{dr,\mathbb C_p}$ 
by $(-)\otimes 1$ and $\alpha(X)\circ((-)\otimes 1)$ respectively.
\end{itemize}
\end{prop}

\begin{proof}
\noindent(i):Follows immediately from the fact that $H^jR\alpha(X)$ 
is a filtered quasi-isomorphism compatible with the Galois action of $\hat G_{\sigma_p}$ by \cite{Chinois}.

\noindent(ii):Let $\alpha\in H^j_{et}(X_{\mathbb C_p},\mathbb Z_p)^G$
Take a basis $((\alpha_i)_{1\leq i\leq t},(\alpha_i)_{t+1\leq i\leq s})\in H^j_{et}(X_{\mathbb C_p},\mathbb Z_p)$ such
that for $1\leq i\leq t$, $\alpha_i\in H^j_{et}(X_{\mathbb C_p},\mathbb Z_p)^G$. We have
\begin{equation*}
w(\alpha)=\sum_{1\leq i\leq s}\lambda_{i,w(\alpha)}w(\alpha_i)\in H^j_{DR}(X_{\mathbb C_p})
\end{equation*} 
Assume by absurd that $w(\alpha)\notin H^j_{DR}(X)$. 
There exist a finite type extenion $k'/k$, $k'\subset\mathbb C_p$ (depending on $w(\alpha)$) 
such that $w(\alpha)\in H^j_{DR}(X_{k'})$. 
By assumption the orbit $Aut(k'/k)\cdot w(\alpha)\subset H_{DR}(X_{k'})$ of $w(\alpha)$ under $Aut(k'/k)$
is non trivial (i.e. contain more then one element). 
Then there exist a prime number $l$ and an embedding $\sigma'_l:k'\hookrightarrow\mathbb C_l$ 
such that the extension $\hat k'_{\sigma'_l}/\hat k_{\sigma_l}$ is non trivial 
(i.e. $\hat k'_{\sigma'_l}\neq\hat k_{\sigma_l}$) where $\sigma_l=\sigma'_{l|k}$
and such that $w(\alpha):=\pi_{k'/\hat k'_{\sigma'_l}}(X_{k'})^*w(\alpha)\notin H^j_{DR}(X_{\hat k_{\sigma_l}})$, 
where $\pi_{k'/\hat k'_{\sigma'_l}}(X_{k'}):X_{\hat k'_{\sigma'_l}}\to X_{k'}$ is the projection,
recall that $H^j_{DR}(X_{k'})=H^j_{DR}(X)\otimes_kk'$.
By injectivity of 
\begin{equation*}
\pi_{k'/\hat k'_{\sigma'_l}}(X_{k'})^*:H^j_{DR}(X_{k'})\to H^j_{DR}(X_{\hat k'_{\sigma'_l}}),
\end{equation*}
the orbit $Gal(\hat k'_{\sigma'_l}/\hat k_{\sigma_l})\cdot\lambda w(\alpha)\subset H_{DR}(X_{\hat k'_{\sigma'_l}})$ 
of $w(\alpha):=\pi_{k'/\hat k'_{\sigma'_l}}(X_{k'})^*w(\alpha)$ 
under $Gal(\hat k'_{\sigma'_l}/\hat k_{\sigma'_l})\subset Aut(k'/k)$ is non trivial (i.e. contain more then one element).
By the main result of \cite{Chinois} (see above), we get
\begin{equation*}
H^jR\alpha(X_{\mathbb C_l})^{-1}(\lambda w(\alpha))\notin 
(H^j_{et}(X_{\bar k'},\mathbb Z_l)\otimes_{\mathbb Z_l}\mathbb B_{dr,\mathbb C_l})^{\hat G_{\sigma'_l}}.
\end{equation*}
Consider the pairing of $G$ modules
\begin{equation*}
\delta(-,-):H^k_{pet}(X_{\mathbb C_p},\mathbb B_{dr,\mathbb C_p})\otimes_{k'}
H^l_{pet}(X_{\mathbb C_l},\mathbb B_{dr,\mathbb C_l})\to 
H^{k+l}_{pet}(X_{\mathbb C_p\otimes_{k'}\mathbb C_l},
\mathbb B_{dr,\mathbb C_p}\otimes_{k'}\mathbb B_{dr,\mathbb C_l})
\end{equation*}
and    
\begin{eqnarray*}
\alpha(X_{\mathbb C_p\otimes_{k'}\mathbb C_l}):
\pi_p^*\mathbb B_{dr,X_{\mathbb C_p}}\otimes_{k'}\pi_l^*\mathbb B_{dr,X_{\mathbb C_l}}
\xrightarrow{\alpha(X_{\mathbb C_p})\otimes\alpha(X_{\mathbb C_l})} \\ 
\pi_p^*(\Omega^{\bullet}_{X_{\mathbb C_p}}\otimes_{O_{X_{\mathbb C_p}}}O\mathbb B_{dr,X_{\mathbb C_p}})\otimes_{k'}
\pi_l^*(\Omega^{\bullet}_{X_{\mathbb C_l}}\otimes_{O_{X_{\mathbb C_l}}}O\mathbb B_{dr,X_{\mathbb C_l}}) 
\xrightarrow{w(\Omega)} \\
\Omega^{\bullet}_{X_{\mathbb C_p\otimes_{k'}\mathbb C_l}}
\otimes_{\pi_p^*O_{X_{\mathbb C_p}}\otimes_{k'}\pi_l^*O_{X_{\mathbb C_l}}}
O\mathbb B_{dr,X_{\mathbb C_p}}\otimes_{k'}O\mathbb B_{dr,X_{\mathbb C_l}}
=:DR(X)(O\mathbb B_{dr,X_{\mathbb C_p}}\otimes_{k'}O\mathbb B_{dr,X_{\mathbb C_l}})
\end{eqnarray*}
is the canonical map in 
$C_{\hat G_{\sigma'_p}\times\hat G_{\sigma'_l}}(X^{pet}_{\mathbb C_p\otimes_{k'}\mathbb C_l})$ and 
\begin{itemize}
\item $\pi_p:=\pi_{\mathbb C_p/\mathbb C_p\otimes_{k'}\mathbb C_l}(X_{\mathbb C_p}):
X_{\mathbb C_p\otimes_{k'}\mathbb C_l}\to X_{\mathbb C_p}$ 
\item $\pi_l:=\pi_{\mathbb C_l/\mathbb C_p\otimes_{k'}\mathbb C_l}(X_{\mathbb C_l}):
X_{\mathbb C_p\otimes_{k'}\mathbb C_l}\to X_{\mathbb C_l}$ 
\end{itemize}
are the base change maps. By the commutative diagram of $\hat G_{\sigma'_p}\times\hat G_{\sigma'_l}$ modules
\begin{equation*}
\xymatrix{H^j_{pet}(X_{\bar k'},\underline{\mathbb Z_p})\otimes_{\mathbb Z_p}\mathbb B_{dr,\mathbb C_p}
\ar[r]^{((-)\otimes 1)\circ\pi_{\bar k'/\mathbb C_p\otimes_{k'}\mathbb C_l}(X_{\bar k'})^*}
\ar[dd]_{H^jR\alpha(X_{\mathbb C_l})} & 
H_{pet}^j(X_{\mathbb C_p\otimes_{k'}\mathbb C_l},
\mathbb B_{dr,X_{\mathbb C_p}}\otimes_{k'}\mathbb B_{dr,X_{\mathbb C_l}})
\ar[d]^{H^j\alpha(X_{\mathbb C_p\otimes_{k'}\mathbb C_l})} &
H^j_{pet}(X_{\bar k'},\underline{\mathbb Z_l})\otimes_{\mathbb Z_l}\mathbb B_{dr,\mathbb C_l}
\ar[l]_{((-)\otimes 1)\circ\pi_{\bar k'/\mathbb C_p\otimes_{k'}\mathbb C_l}(X_{\bar k'})^*}
\ar[dd]^{H^jR\alpha(X_{\mathbb C_l})} \\
\, & \mathbb H_{pet}^j(X_{\mathbb C_p\otimes_{k'}\mathbb C_l},
DR(X)(O\mathbb B_{dr,X_{\mathbb C_p}}\otimes_{k'}O\mathbb B_{dr,X_{\mathbb C_l}})) & \, \\
H^j_{DR}(X_{\mathbb C_p})\otimes_{\mathbb C_p}\mathbb B_{dr,\mathbb C_p}
\ar[ru]^{((-)\otimes 1)\circ\pi_{\hat k'_{\sigma_p}/\mathbb C_p}(X_{k'})^*} & 
H^j_{DR}(X_{k'})\ar[l]^{\pi_{k'/\mathbb C_p}(X_{k'})^*}
\ar[u]^{((-)\otimes 1)\circ\pi_{k'/(\mathbb C_p\otimes_{k'}\mathbb C_l)}(X_{k'})^*}
\ar[r]^{((-)\otimes 1)\circ\pi_{k'/\mathbb C_l}(X_{k'})^*} &
H^j_{DR}(X_{\mathbb C_l})\otimes_{\mathbb C_l}\mathbb B_{dr,\mathbb C_l}
\ar[lu]_{((-)\otimes 1)\circ\pi_{\hat k'_{\sigma_l}/\mathbb C_l}(X_{k'})^*}},
\end{equation*}
we have for $\beta_p\in H^j_{pet}(X_{\bar k'},\mathbb Z_p)\otimes_{\mathbb Z_p}\mathbb B_{dr,\mathbb C_p}$
$\beta_l\in H^j_{pet}(X_{\bar k'},\mathbb Z_l)\otimes_{\mathbb Z_l}\mathbb B_{dr,\mathbb C_l}$,
\begin{equation*}
\alpha(X_{\mathbb C_p\otimes_{k'}\mathbb C_l})(\delta(\beta_p,\beta_l))=
\alpha(X_{\mathbb C_l})(\beta_p)\cdot\alpha(X_{\mathbb C_l})(\beta_l)
\in\mathbb H^j(X_{\mathbb C_p\otimes_{k'}\mathbb C_l},
DR(X)(O\mathbb B_{dr,\mathbb C_p}\otimes_{k'}O\mathbb B_{dr,\mathbb C_l})).
\end{equation*}
Note that since 
$\pi_{\bar k'/(\mathbb C_p\otimes_{k'}\mathbb C_l)}(X_{\bar k'}):X_{\mathbb C_p\otimes_{k'}\mathbb C_l}\to X_{\bar k'}$
is flat $((-)\otimes 1)\circ\pi_{\bar k'/(\mathbb C_p\otimes_{k'}\mathbb C_l)}(X_{\bar k'})^*$ and
$((-)\otimes 1)\circ\pi_{\bar k'/(\mathbb C_p\otimes_{k'}\mathbb C_l)}(X_{\bar k'})^*$ are injective,
(the morphism involved in the base change are without torsion).
Denote $d=\dim(X)$. Consider the canonical projection
\begin{eqnarray*}
\pi:X_{\mathbb C_p}\times X_{\mathbb C_l}\to X_{\mathbb C_p\otimes_{k'}\mathbb C_l}, \;
\mbox{given by on} \; X^o\subset X \; \mbox{open affine} \\
\pi((x_1,\cdots,x_d),(x'_1,\cdots,x'_d)):=(x_1\otimes x'_1,\cdots,x_d\otimes x'_d),
\end{eqnarray*}
where $X_{\mathbb C_p}\times X_{\mathbb C_l}$ is endowed with the product topology.
Then the commutative diagram
\begin{equation*}
\xymatrix{
H^j_{pet}(X_{\mathbb C_p\otimes_{k'}\mathbb C_l},\mathbb B_{dr,\mathbb C_p}\otimes_{k'}\mathbb B_{dr,\mathbb C_l})
\ar[r]^{\pi^*}\ar[d]^{\alpha(X_{\mathbb C_p\otimes_{k'}\mathbb C_l})} &
H^j_{pet}(X_{\mathbb C_p}\times X_{\mathbb C_l},\pi^*(\mathbb B_{dr,\mathbb C_p}\otimes_{k'}\mathbb B_{dr,\mathbb C_l}))
\ar[d]_{w(-)\circ(\pi^*\alpha(X_{\mathbb C_p})\otimes\pi^*\alpha(X_{\mathbb C_l}))} \\
\mathbb H_{pet}^j(X_{\mathbb C_p\otimes_{k'}\mathbb C_l},
DR(X)(O\mathbb B_{dr,X_{\mathbb C_p}}\otimes_{k'}O\mathbb B_{dr,X_{\mathbb C_l}}))\ar[r]^{\pi^*} &
\mathbb H_{pet}^j(X_{\mathbb C_p}\times X_{\mathbb C_l},
DR(X)(O\mathbb B_{dr,X_{\mathbb C_p}}\otimes_{k'}O\mathbb B_{dr,X_{\mathbb C_l}}))}
\end{equation*}
together with the $p$ adic Poincare lemma on $X_{\mathbb C_p}$ and the $l$ adic Poincare lemma on $X_{\mathbb C_l}$,
the fact that $\pi^*$ is injective 
(note that the product topology is less fine then the pro-etale topology on $X_{\mathbb C_p}\times_{k'}X_{\mathbb C_l}$
and that the map 
$O_{X_{\mathbb C_p\otimes_{k'}\mathbb C_l}}\to O_{X_{\mathbb C_p}\times_{k'}X_{\mathbb C_l}}$ is torsion free),
show that $\alpha(X_{\mathbb C_p\otimes_{k'}\mathbb C_l})$ is injective. Hence
\begin{equation*}
\alpha\otimes 1=\delta(\alpha,1)=\delta(1,H^jR\alpha(X_{\mathbb C_l})^{-1}(w(\alpha)))
\in H^j_{pet}(X_{\mathbb C_p\otimes_{k'}\mathbb C_l},
\mathbb B_{dr,\mathbb C_p}\otimes_{k'}\mathbb B_{dr,\mathbb C_l}),
\end{equation*} 
that is there exists $\lambda_{\alpha}\in k'$ such that
\begin{equation*}
\alpha=\lambda_{\alpha}H^jR\alpha(X_{\mathbb C_l})^{-1}(w(\alpha))
=H^jR\alpha(X_{\mathbb C_l})^{-1}(\lambda_{\alpha}w(\alpha))
\in H^j_{pet}(X_{\mathbb C_p\otimes_{k'}\mathbb C_l},
\mathbb B_{dr,\mathbb C_p}\otimes_{k'}\mathbb B_{dr,\mathbb C_l}).
\end{equation*} 
Since $\alpha$ and $w(\alpha)$ are $\hat G_{\sigma'_p}$ invariant, $\lambda_{\alpha}\in\mathbb Q_p$.
This gives
\begin{equation*}
\alpha=H^jR\alpha(X_{\mathbb C_l})^{-1}(\lambda_{\alpha}w(\alpha))
\notin H^j_{pet}(X_{\mathbb C_p\otimes_{k'}\mathbb C_l},
\mathbb Z_p\otimes\mathbb B_{dr,\mathbb C_l})^{\hat G_{\sigma'_l}}.
\end{equation*} 
Contradiction. We thus have $w(\alpha)\in H^j_{DR}(X)$. 
Conversely if $\alpha\notin H_{et}^j(X_{\bar k},\mathbb Z_p)(l)^G$, we get similarly $w(\alpha)\notin H^j_{DR}(X)$. 
The result then follows from (i) and the equality
$F^lH^j_{DR}(X)=F^lH^j_{DR}(X_{\mathbb C_p})\cap H^j_{DR}(X)$ 
given by the filtered isomorphism in $C_{fil}(\bar X_{\mathbb C_p})$
\begin{equation*}
w(k/\mathbb C_p):(\Omega^{\bullet}_{\bar X}(\log D),F_b)\otimes_k\mathbb C_p
\xrightarrow{\sim}(\Omega^{\bullet}_{\bar X_{\mathbb C_p}}(\log D_{\mathbb C_p}),F_b)
\end{equation*}
which say that the Hodge filtration is defined over $k$: see section 2.
\end{proof}

\begin{thm}\label{CevGd}
\begin{itemize}
\item[(i)]Let $k$ a field of finite type over $\mathbb Q$. 
Denote $\bar k$ its algebraic closure and $G=Gal(\bar k/k)$ its absolute Galois group.
Let $X\in\SmVar(k)$. Let $\sigma:k\hookrightarrow\mathbb C$ an embedding. 
Let $p\in\mathbb N\backslash\delta(k,X)$ a prime number. Let $\sigma_p:k\hookrightarrow\mathbb C_p$ an embedding. 
For each $j,l\in\mathbb Z$, we get from proposition \ref{keypropCp}(i), proposition \ref{LatticeLog},
and proposition \ref{keypropC}, a canonical injective map
\begin{equation*}
H^j\iota_{p,ev}^{G,l}(X):H_{et}^j(X_{\bar k},\mathbb Z_p)(l)^G
\hookrightarrow F^lH^j(X^{an}_{\mathbb C},2i\pi\mathbb Q)\otimes_{\mathbb Q}\mathbb Q_p, \;
\alpha\mapsto H^j\iota_{p,ev}^{G,d}(X)(\alpha):=ev(X)(w(\alpha)),
\end{equation*}
with $w(\alpha):=H^jR\alpha(X)(\alpha\otimes 1)\in H^j_{DR}(X)
\subset H^j_{DR}(X_{\mathbb C_p})\otimes_{\mathbb C_p}\mathbb B_{st,\mathbb C_p}$ (see definition \ref{walpha}), and
\begin{equation*}
F^lH^j(X_{\mathbb C}^{an},2i\pi\mathbb Q):=
H^j_{\sing}(X_{\mathbb C}^{an},2i\pi\mathbb Q)\cap F^lH^j_{DR}(X_{\mathbb C}^{an})
\subset H^j(X_{\mathbb C}^{an},\mathbb C).
\end{equation*}
By construction, for $f:X'\to X$ a morphism with $X,X'\in\SmVar(k)$ and $p\in\mathbb N\backslash\delta(k,X,X')$ a prime number,
we have the commutative diagram
\begin{equation*}
\xymatrix{H_{et}^j(X_{\bar k},\mathbb Z_p)(l)^G\ar[rr]^{H^j\iota_{p,ev}^{G,l}(X)}\ar[d]^{f^*} & \, &  
F^lH^j(X^{an}_{\mathbb C},2i\pi\mathbb Q)\otimes_{\mathbb Q}\mathbb Q_p\ar[d]^{f*} \\
H_{et}^j(X'_{\bar k},\mathbb Z_p)(l)^G\ar[rr]^{H^j\iota_{p,ev}^{G,l}(X')} & \, &
F^lH^j(X^{'an}_{\mathbb C},2i\pi\mathbb Q)\otimes_{\mathbb Q}\mathbb Q_p}.
\end{equation*}
\item[(ii)]Let $K\subset\mathbb C_p$ a $p$-adic field such that $Frac(W(O_K))=K$. 
Let $X\in\SmVar(K)$ such that its canonical model
$X^{\mathcal O}\in\Sch/O_K$ has good or semi-stable reduction modulo $p$.
Denote $G_p=Gal(\bar{\mathbb Q_p}/K)$ its absolute Galois group.
Let $\sigma:K\hookrightarrow\mathbb C$ an embedding.  
For each $j,l\in\mathbb Z$, we get from proposition \ref{keypropCp}(i), proposition \ref{LatticeLog},
and proposition \ref{keypropC}, a canonical injective map
\begin{equation*}
H^j\iota_{ev}^{G_p,l}(X):H_{et}^j(X_{\mathbb C_p},\mathbb Z_p)(l)^{G_p}
\hookrightarrow F^lH^j(X^{an}_{\mathbb C},2i\pi\mathbb Q)\otimes_{\mathbb Q}\mathbb Q_p, \;
\alpha\mapsto H^j\iota_{ev}^{G_p,l}(X)(\alpha):=ev(X)(w(\alpha)),
\end{equation*}
with $w(\alpha):=H^jR\alpha(X)(\alpha\otimes 1)\in H^j_{DR}(X)
\subset H^j_{DR}(X_{\mathbb C_p})\otimes_{\mathbb C_p}\mathbb B_{st,\mathbb C_p}$ (see definition \ref{walpha}), and 
\begin{equation*}
F^lH^j(X_{\mathbb C}^{an},2i\pi\mathbb Q):=
H^j_{\sing}(X_{\mathbb C}^{an},2i\pi\mathbb Q)\cap F^lH^j_{DR}(X_{\mathbb C}^{an})
\subset H^j(X_{\mathbb C}^{an},\mathbb C).
\end{equation*}
By construction, for $f:X'\to X$ a morphism with $X,X'\in\SmVar(K)$,we have the commutative diagram
\begin{equation*}
\xymatrix{H_{et}^j(X_{\mathbb C_p},\mathbb Z_p)(l)^{G_p}\ar[rr]^{H^j\iota_{ev}^{G_p,l}(X)}\ar[d]^{f^*} & \, &  
F^lH^j(X^{an}_{\mathbb C},2i\pi\mathbb Q)\otimes_{\mathbb Q}\mathbb Q_p\ar[d]^{f*} \\
H_{et}^j(X'_{\mathbb C_p},\mathbb Z_p)(l)^{G_p}\ar[rr]^{H^j\iota_{ev}^{G_p,l}(X')} & \, &
F^lH^j(X^{'an}_{\mathbb C},2i\pi\mathbb Q)\otimes_{\mathbb Q}\mathbb Q_p}.
\end{equation*}
Note that (ii) implies (i).
\end{itemize}
\end{thm}

\begin{proof}
\noindent(i):Follows from (ii).

\noindent(ii):By proposition \ref{keypropCp}(i), 
$w(\alpha)\in F^lH^j_{DR}(X_{\hat k_{\sigma_p}})
\subset H^j_{DR}(X_{\mathbb C_p})\otimes_{\mathbb C_p}\mathbb B_{st,\mathbb C_p}$.
By proposition \ref{LatticeLog},
\begin{eqnarray*}
w(\alpha)=\sum_{1\leq i\leq r}\lambda_iw_{Li}\in H^j_{DR}(X_{\mathbb C_p}), \;
(w_{Li})_{1\leq i\leq r}\in\mathbb H_{pet}^j(X_{\mathbb C_p},
\Omega^{\bullet\geq l}_{X_{\mathbb C_p},\log,\mathcal O}), \; (\lambda_i)_{1\leq i\leq r}\in\mathbb Z_p.
\end{eqnarray*}
Then, for each $1\leq i\leq r$,
\begin{itemize}
\item by proposition \ref{keypropC}, $ev(X)(w_{Li})\in H^j_{\sing}(X_{\mathbb C}^{an},2i\pi\mathbb Q)$,
\item $w_{Li}:=H^jOL_X(w_{Li})\in F^lH^j_{DR}(X_{\mathbb C_p})$,
\end{itemize}
that is, $w_{Li}\in F^lH^j(X_{\mathbb C}^{an},2i\pi\mathbb Q)$. Hence,
\begin{equation*}
ev(X)(w(\alpha))=\sum_{1\leq i\leq r}\lambda_iev(X)(w_{Li})\in 
F^lH^j(X_{\mathbb C}^{an},2i\pi\mathbb Q)\otimes_{\mathbb Q}\mathbb Q_p.
\end{equation*}
\end{proof}

\begin{rem}
Let $k$ a field of finite type over $\mathbb Q$. Let $X\in\SmVar(k)$.
Let $p\in\mathbb N$ be a prime number.
Let $\sigma_p:k\hookrightarrow\mathbb C_p$ and $\sigma:k\hookrightarrow\mathbb C$ be embeddings. 
Note that for $w\in\mathbb H^j(X,\Omega_{X^{et}}^{\bullet})$ such that 
$ev(X)(w)\in H^j_{\sing}(X_{\mathbb C}^{an},2i\pi\mathbb Q)$, 
$w$ is NOT logarithmic in general and $T(X)(ev(X)(w))\in H^j_{et}(X_{\bar k},\mathbb Q_p)$
is NOT $G=Gal(\bar k/k)$ equivariant in general, where 
\begin{equation*}
T(X):H^j_{sing}(X^{an}_{\mathbb C},\mathbb Q)\xrightarrow{/p^*}
H^j_{sing}(X^{an}_{\mathbb C},\mathbb Q_p)\xrightarrow{\sim} 
H^j_{et}(X^{an}_{\mathbb C},\mathbb Q_p)=H^j_{et}(X_{\bar k},\mathbb Q_p).
\end{equation*}
is given in section 2.
\end{rem}

\section{The complex and etale Abel Jacobi maps and normal function}

\subsection{The complex Abel Jacobi map for higher Chow group and complex normal functions}

Let $k$ a field of finite type over $\mathbb Q$. 
Consider an embedding $\sigma:k\hookrightarrow\mathbb C$. 
Then $k\subset\bar k\subset\mathbb C$, where $\bar k$ is the algebraic closure of $k$.
We have then the quasi-isomorphism
$\alpha(X):\mathbb C_{X_{\mathbb C}^{an}}\hookrightarrow\Omega^{\bullet}_{X^{an}_{\mathbb C}}$
in $C(X_{\mathbb C}^{an})$.

\begin{itemize}
\item For $X\in\SmVar(k)$, we consider
\begin{equation*}
((H^j_{DR}(X),F),H^j(X_{\mathbb C}^{an},\mathbb Z),H^j\alpha(X))
\in MHM_{k,gm}(k)\subset\Vect_{fil}(k)\times_I\Ab
\end{equation*}
where $F$ is the Hodge filtration on $H^j_{DR}(X)\otimes_k\mathbb C$ and
\begin{equation*}
H^jR\Gamma(X_{\mathbb C}^{an},\alpha(X)):H^j(X_{\mathbb C}^{an},\mathbb C)\xrightarrow{\sim}H^j_{DR}(X)\otimes_k\mathbb C
\end{equation*} 
Recall the geometric mixed Hodge structures (see \cite{B5}) are mixed Hodge structure by the
Hodge decomposition theorem on smooth proper complex varieties.
\item For $f:X\to S$ a morphism with $S,X\in\SmVar(k)$, we consider
\begin{equation*}
(H^jRf_{*Hdg}(O_X,F_b),H^jRf_*\mathbb Z_{X_{\mathbb C}^{an}},H^jRf_*\alpha(X))
\in MHM_{k,gm}(S)\subset\PSh_{\mathcal D(1,0)fil}(S)\times_IP_{fil,k}(S_{\mathbb C}^{an})
\end{equation*}
where 
\begin{equation*}
H^jRf_*\alpha(X):H^jRf_*\mathbb C_{X_{\mathbb C}^{an}}\xrightarrow{\sim}
DR(S)(o_{fil}H^jRf_{*Hdg}(O_X,F_b)\otimes_k\mathbb C=H^j\int_fO_X\otimes_k\mathbb C)
\end{equation*} 
Recall the geometric mixed Hodge modules (see \cite{B5}) are mixed Hodge modules by 
a theorem of Saito for proper morphisms of smooth complex varieties.
\end{itemize}

Let $X\in\SmVar(k)$. We have for $j,d\in\mathbb N$, the generalized Jacobian
\begin{eqnarray*}
J_{\sigma}^{j,d}(X):=
H^j(X_{\mathbb C}^{an},\mathbb C)/(F^dH^j(X_{\mathbb C}^{an},\mathbb C)\oplus H^j(X_{\mathbb C}^{an},\mathbb Z))
\end{eqnarray*}
where $F$ is given the Hodge filtration on $H^j_{DR}(X)\otimes_k\mathbb C$ and
\begin{equation*}
H^jR\Gamma(X_{\mathbb C}^{an},\alpha(X)):H^j(X_{\mathbb C}^{an},\mathbb C)\xrightarrow{\sim}H^j_{DR}(X)\otimes_k\mathbb C.
\end{equation*}
If $X\in\PSmVar(k)$ and $2d\geq n$, $J^{j,d}(X)$ is a complex torus since 
\begin{equation*}
((H^j_{DR}(X),F),H^j(X_{\mathbb C}^{an},\mathbb Z),H^jR\Gamma(X_{\mathbb C}^{an},\alpha(X)))
\in HM_{k,gm}(k)\subset MHM_{k,gm}(k)
\end{equation*}
is a pure Hodge structure.
For $X\in\PSmVar(k)$, we have a canonical isomorphism of abelian groups 
\begin{eqnarray*}
I_{\sigma}^{j,d}(X):J_{\sigma}^{j,d}(X)\xrightarrow{\sim}
\Ext^1_{MHM_{k,gm}(k)}(\mathbb Z_{\pt}^{Hdg}(d),(H^j_{DR}(X),H^j(X_{\mathbb C}^{an},\mathbb Z),
H^jR\Gamma(X_{\mathbb C}^{an},\alpha(X)))).
\end{eqnarray*}

\begin{defi}\label{AJdef}
Let $X\in\SmVar(k)$ irreducible. 
Let $\bar X\in\PSmVar(k)$ a compactification of $X$ with $D:=\bar X\backslash X\subset\bar X$ a normal crossing divisor.
The map of complexes of abelian groups (see \cite{B2})
\begin{eqnarray*}
\mathcal R_X:\mathcal Z^d(X,\bullet)\to C_{\bullet}^{\mathcal D}(\bar X_{\mathbb C}^{an},D_{\mathbb C}^{an}),
Z\mapsto\mathcal R_X:=(T_Z,\Omega_Z,R_Z)
\end{eqnarray*}
where $C_{\bullet}^{\mathcal D}(\bar X_{\mathbb C}^{an},D_{\mathbb C}^{an})$ is the Deligne homology complex
induces the complex Abel Jacobi map for higher Chow groups
\begin{eqnarray*}
AJ_{\sigma}(X):\mathcal Z^d(X,n)^{\partial=0}_{hom}\to J_{\sigma}^{2d-1-n,d}(X), 
Z\mapsto AJ_{\sigma}(X)(Z):=D^{-1}(R'_Z), \\
R'_Z=R_Z-\Omega'_Z+T'_Z, \,; \mbox{with} \; \partial T'_Z=T_Z, \, \partial\Omega'_Z=\Omega_Z
\end{eqnarray*}
where 
\begin{equation*}
D:C_{\bullet}^{\mathcal D}(X_{\mathbb C}^{an})\to 
C_{\bullet}^{\mathcal D}(\bar X_{\mathbb C}^{an},D_{\mathbb C}^{an}) 
\end{equation*}
is the Poincare dual for Deligne homology.
\end{defi}

\begin{thm}\label{AJC}
Let $k\subset\mathbb C$ a subfield. Let $X\in\PSmVar(k)$. 
\begin{itemize}
\item[(i)]For $Z\in\mathcal Z^d(X,n)^{\partial=0}_{hom}$, we have
\begin{eqnarray*}
AJ_{\sigma}(X)(Z)=
I_{\sigma}^{j,d}(X)^{-1}(0\to (H^{2d-1-n}_{DR}(X),H^{2d-1-n}_{\sing}(X^{an}_{\mathbb C},\mathbb Z),
H^{2d-1-n}R\Gamma(X_{\mathbb C}^{an},\alpha(X))) \\ 
\xrightarrow{(j^*,j^*,0)} \\
(H^{2d-1}_{DR}((X\times\square^n)\backslash|Z|),H^{2d-1}_{\sing}(((X\times\square^n)\backslash|Z|)^{an}_{\mathbb C},\mathbb Z),
H^{2d-1-n}R\Gamma((X\times\square^n)\backslash|Z|)_{\mathbb C}^{an},\alpha(X\times\square^n))^{[Z]} \\
\xrightarrow{(\partial,\partial,0)} \\
(H_{DR,|Z|}^{2d}(X\times\square^n),H_{\sing,|Z|}^{2d}((X\times\square^n)^{an}_{\mathbb C},\mathbb Z),
R\Gamma_{|Z|}((X\times\square^n)_{\mathbb C}^{an},\alpha(X\times\square^n)))^{[Z]}
=\mathbb Z_{\pt}^{Hdg}(n-d)\to 0)
\end{eqnarray*}
where $j:(X\times\square^n)\backslash|Z|\hookrightarrow X\times\square^n$ is the open embedding and
\begin{equation*}
H^{2d}_{Hdg,|Z|}(X\times\square^n)^{[Z]}\subset H^{2d}_{Hdg,|Z|}(X\times\square^n), \;
H^{2d-1}_{Hdg}((X\times\square^n)\backslash|Z|)^{[Z]}\subset H^{2d-1}_{Hdg}((X\times\square^n)\backslash|Z|).
\end{equation*}
are the subobjects given by the pullback of the class of $Z$ (see section 2).
\item[(ii)]Let $Z\in\mathcal Z^d(X,n)^{\partial=0}_{hom}$. Then $AJ_{\sigma}(X)(Z)=0$ if and only if there exist 
$w\in H^{2d-1}_{DR}((X\times\square^n)\backslash|Z|)^{[Z]}$ such that 
\begin{itemize}
\item $w\in F^dH^{2d-1}_{DR}((X\times\square^n)\backslash|Z|)$, 
\item $ev((X\times\square^n)\backslash|Z|)(w)\in 
H^{2d-1}_{\sing}(((X\times\square^n)\backslash|Z|)^{an}_{\mathbb C},\mathbb Z(2i\pi))$
\item $\partial w\neq 0$.
\end{itemize}
\end{itemize}
\end{thm}

\begin{proof}
\noindent(i):See \cite{Carlson}.

\noindent(ii):Follows from (i)
\end{proof}

Let $f:X\to S$ a morphism with $S,X\in\SmVar(k)$.
We have for $j,d\in\mathbb N$ such that $H^jRf_*\mathbb C_{X_{\mathbb C}^{an}}$ is a local system
and $F^dH^jRf_{*Hdg}(O_X,F_b)\subset H^j\int_f(O_X)$ are locally free sub $O_S$ modules, 
the generalized relative intermediate Jacobian
\begin{eqnarray*}
J_{\sigma}^{j,d}(X/S):=
(H^jRf_*\mathbb C_{X_{\mathbb C}^{an}}\otimes O_{S_{\mathbb C}^{an}})/
(F^d(H^jRf_*\mathbb C_{X_{\mathbb C}^{an}}\otimes O_{S_{\mathbb C}^{an}})\oplus H^jRf_*\mathbb Z_{X_{\mathbb C}^{an}})
\end{eqnarray*}
where $F$ is given by the Hodge filtration on $H^j\int_f(O_X)$ and
\begin{equation*}
H^jRf_*\alpha(X):H^jRf_*\mathbb C_{X_{\mathbb C}^{an}}\xrightarrow{\sim}DR(S)(H^j\int_f(O_X)).
\end{equation*}
A generalized normal function is then a section $\nu\in\Gamma(S_{\mathbb C}^{an},J_{\sigma}^{j,d}(X/S))$
which is horizontal (i.e. $\nabla\nu=0$).
For $s\in S$, we get immediately that $i_s^{*mod}J_{\sigma}^{j,d}(X/S)=J_{\sigma}^{j,d}(X_s)$.
In particular we get for $f:X\to S$ a smooth projective morphism with $S,X\in\SmVar(k)$
and $j,d\in\mathbb N$, the relative intermediate Jacobian
\begin{eqnarray*}
J_{\sigma}^{j,d}(X/S):=
(H^jRf_*\mathbb C_{X_{\mathbb C}^{an}}\otimes O_{S_{\mathbb C}^{an}})/
(F^d(H^jRf_*\mathbb C_{X_{\mathbb C}^{an}}\otimes O_{S_{\mathbb C}^{an}})\oplus H^jRf_*\mathbb Z_{X_{\mathbb C}^{an}})
\end{eqnarray*}
where $F$ is given by the Hodge filtration on $H^j\int_f(O_X)=H^jRf_*\Omega^{\bullet}_{X/S}$ and $H^jRf_*\alpha(X)$.
A normal function is then a section $\nu\in\Gamma(S_{\mathbb C}^{an},J_{\sigma}^{j,d}(X/S))$ which is horizontal.
For $f:X\to S$ a smooth projective morphism with $S,X\in\SmVar(k)$, we have a canonical isomorphism of abelian groups 
\begin{eqnarray*}
I_{\sigma}^{j,d}(X/S):J_{\sigma}^{j,d}(X/S)\xrightarrow{\sim}
\Ext^1_{MHM(S)}(\mathbb Z_S^{Hdg}(d),(H^jRf_{*Hdg}(O_X,F_b),H^jRf_*\mathbb Z_{X_{\mathbb C}^{an}},H^jRf_*\alpha(X))).
\end{eqnarray*}

\begin{defiprop}\label{AJdefS}
Let $f:X\to S$ a morphism with $S,X\in\SmVar(k)$. Let $j:S^o\hookrightarrow S$ an open subset such that 
for all $j,d\in\mathbb Z$, $j^*H^jRf_*\mathbb C_{X_{\mathbb C}^{an}}$ is a local system
and $j^*F^dH^jRf_{*Hdg}(O_X,F_b)\subset j^*H^j\int_f(O_X)$ is a locally free sub $O_S$ module.
Let $\sigma:k\hookrightarrow\mathbb C$ an embedding.
Let $d,n\in\mathbb N$. We have then, denoting $X^o:=X\times_SS^o$ and using definition \ref{AJdef}, the map
\begin{eqnarray*}
AJ_{\sigma}(X^o/S^o):\mathcal Z^d(X,n)^{f,\partial=0}_{fhom}\to
\Gamma(S_{\mathbb C}^{o,an},J_{\sigma}^{2d-n-1,d}(X^o/S^o))
\subset\Gamma(S_{\mathbb C}^{o,an},\oplus_{s\in S_{\mathbb C}^{o,an}}i_{s*}J_{\sigma'}^{2d-n-1,d}(X_s)), \\
Z\mapsto AJ_{\sigma}(X^o/S^o)(Z):=\nu_Z:=((s\in S_{\mathbb C}^o)\mapsto 
(AJ_{\sigma'}(X_s)(Z_s)\in J_{\sigma'}^{2d-n-1,d}(X_s)))
\end{eqnarray*}
where $\mathcal Z^d(X,n)^{f,\partial=0}_{fhom}\subset\mathcal Z^d(X,n)^{f,\partial=0}$ denote the sub-abelian group
consisting of algebraic cycles $Z$ with $Z_s:=i_s^*Z\in\mathcal Z^d(X_s,n)^{\partial=0}_{hom}$,
and $\sigma':k(s)\hookrightarrow\mathbb C$ is the embedding given by $s$ extending $\sigma:k\hookrightarrow\mathbb C$, 
denoting again $s:=\pi{k/\mathbb C}(S)(s)\in S$, $\pi_{k/\mathbb C}(S):S_{\mathbb C}\to S$ being the projection.
\end{defiprop}

\begin{proof}
Standard : to show that 
\begin{equation*}
\nu_Z:=(s\in S_{\mathbb C}^o)\mapsto AJ_{\sigma}(X_s)(Z_s)\in
\Gamma(S_{\mathbb C}^{o,an},\oplus_{s\in S_{\mathbb C}^{o,an}}i_{s*}J_{\sigma}^{2d-n-1,d}(X_s)) 
\end{equation*}
is holomorphic and horizontal
we consider a compactification $\bar f:\bar X\to S$ of $f$ with $\bar X\in\SmVar(k)$ and 
use trivializations of 
$f:(\bar X_{\mathbb C}^{o,an},(\bar X\backslash X)_{\mathbb C}^{o,an})\to S_{\mathbb C}^{o,an}$ 
which gives trivialization of the local system $j^*Rf_*\mathbb Z_{S_{\mathbb C}^{an}}$ (see \cite{B2} for example).
\end{proof}

\begin{cor}\label{AJCcor}
Let $f:X\to S$ a smooth projective morphism with $S,X\in\SmVar(k)$. 
Let $d,n\in\mathbb N$. For $Z\in\mathcal Z^d(X,n)^{f,\partial=0}_{fhom}$, we have 
\begin{eqnarray*}
AJ_{\sigma}(X/S)(Z)=
I_{\sigma}^{j,d}(X/S)^{-1}(0\to (H^{2d-1-n}Rf_{*Hdg}(O_X,F_b),
H^{2d-1-n}Rf_*\mathbb Z_{X^{an}_{\mathbb C}},H^{2d-1-n}Rf_*\alpha(X)) \\
\xrightarrow{(j^*,j^*,0)} 
(H^{2d-1}R(f\circ j)_{*Hdg}(O_{(X\times\square^n)\backslash|Z|},F_b),
H^{2d-1}R(f\circ j)_*\mathbb Z_{((X\times\square^n)\backslash|Z|)^{an}_{\mathbb C}}, \\
H^{2d-1}R(f\circ j)_*\alpha((X\times\square^n)\backslash|Z|))^{[Z]} 
\xrightarrow{(\partial,\partial,0)} \\
(H^{2d}Rf_{*Hdg}R\Gamma^{Hdg}_{|Z|}(O_{X\times\square^n},F_b),
(H^{2d}Rf_*R\Gamma_{|Z|}\mathbb Z_{(X\times\square^n)^{an}_{\mathbb C}}), 
Rf_*\Gamma_{|Z|}\alpha(X\times\square^n))^{[Z]}=\mathbb Z_S^{Hdg}(n-d)\to 0)
\end{eqnarray*}
where $j:(X\times\square^n)\backslash|Z|\hookrightarrow X\times\square^n$ is the open embedding and
\begin{eqnarray*}
(H^{2d}Rf_{*Hdg}R\Gamma^{Hdg}_{|Z|}\mathbb Z_{X\times\square^n}^{Hdg})^{[Z]}
\subset H^{2d}Rf_{*Hdg}R\Gamma^{Hdg}_{|Z|}\mathbb Z_{X\times\square^n}^{Hdg}, \\
(H^{2d}R(f\circ j)_{*Hdg}\mathbb Z_{(X\times\square^n)\backslash|Z|}^{Hdg})^{[Z]}
\subset H^{2d}R(f\circ j)_{*Hdg}\mathbb Z_{(X\times\square^n)\backslash|Z|}^{Hdg}.
\end{eqnarray*}
are the subobjects given by the pullback of the class of $Z$.
\end{cor}

\begin{proof}
Follows from theorem \ref{AJC} by definition of the Abel Jacobi map and by the base change for mixed hodge modules.
\end{proof}

We have the following main result of \cite{BPS} :

\begin{thm}\label{BPSthm}
Let $f:X\to S$ a smooth projective morphism with $S,X\in\SmVar(k)$. Let $d,n\in\mathbb N$. 
Let $\sigma:k\hookrightarrow\mathbb C$ an embedding.
For $Z\in\mathcal Z^d(X,n)^{f,\partial=0}_{fhom}$, the zero locus $V(\nu_Z)\subset S_{\mathbb C}$ of 
\begin{equation*}
\mu_Z:=AJ_{\sigma}(X/S)(Z)\in\Gamma(S_{\mathbb C}^{an},J_{\sigma}^{2d-1-n,d}(X/S))
\end{equation*}
is an algebraic subvariety.
\end{thm}

\begin{proof}
See \cite{BPS}: if $\bar S\in\PSmVar(k)$ is a compactification of $S$ 
with $\bar S\backslash S=\cup_iD_i\subset\bar S$ a normal crossing divisor,
there exist an analytic subset $\Sigma(\nu_Z)\subset\bar S_{\mathbb C}$ such that 
$V(\nu_Z)=\Sigma(\nu_Z)\cap S_{\mathbb C}$. 
By GAGA $\Sigma(\nu_Z)\subset\bar S_{\mathbb C}$ is algebraic subvariety. 
Hence $V(\nu_Z)\subset S_{\mathbb C}$ is an algebraic subvariety.
\end{proof}

\subsection{The etale Abel Jacobi map for higher Chow group and etale normal functions}

Let $k$ a field of finite type over $\mathbb Q$. 
Let $\bar k$ the algebraic closure of $k$ and denote by $G=Gal(\bar k/k)$ its galois group.
Let $p\in\mathbb N$ a prime integer.

\begin{defi}\label{AJpdef}
Let $X\in\SmVar(k)$ irreducible. 
Let $\bar X\in\PSmVar(k)$ a compactification of $X$ with $D:=\bar X\backslash X\subset\bar X$ a normal crossing divisor.
Denote $G=Gal(\bar k/k)$ the absolute galois group. The cycle class map
\begin{eqnarray*}
\mathcal R^{et,p}_X:\mathcal Z^d(X,n)^{\partial=0}\to H^{2d-n}_{p_X(|Z|),et}(X,D,\hat{\mathbb Z_p})
\to H_{et}^{2d-n}(X,D,\hat{\mathbb Z_p}), 
\end{eqnarray*}
to continuous etale cohomology induces the etale Abel Jacobi map for higher Chow groups
\begin{eqnarray*}
AJ_{et,p}(X):\mathcal Z^d(X,n)^{\partial=0}_{hom}\to\Ext^1_G(\bar k,H_{et}^{2d-1-n}(X_{\bar k},D_{\bar k},\mathbb Z_p)), \\ 
Z\mapsto AJ_{et,p}(X)(Z):=L^1\mathcal R^{et,p}_X(Z)/L^2\mathcal R^{et,p}_X(Z),
\end{eqnarray*}
where $L$ is the filtration given by the Leray spectral sequence of the map of sites $a_X::X^{et}\to\Spec(k)^{et}$.
\end{defi}

\begin{thm}\label{AJk}
Let $X\in\PSmVar(k)$. Denote $G=Gal(\bar k/k)$ the absolute galois group.
\begin{itemize}
\item[(i)]For $Z\in\mathcal Z^d(X,n)^{\partial=0}_{hom}$, we have
\begin{eqnarray*}
AJ_{et,p}(X)(Z)=(0\to H_{et}^{2d-1-n}(X_{\bar k},\mathbb Z_p) 
\xrightarrow{j^*}H_{et}^{2d-1}(((X\times\square^n)\backslash|Z|)_{\bar k},\mathbb Z_p)^{[Z]} \\
\xrightarrow{\partial}H_{|Z|,et}^{2d}((X\times\square^n)_{\bar k},\mathbb Z_p)^{[Z]}=\bar k(n-d)\to 0)
\end{eqnarray*}
with $j:(X\times\square^n)\backslash|Z|\hookrightarrow X\times\square^n$ the open embedding, and 
\begin{eqnarray*}
H_{|Z|,et}^{2d}((X\times\square^n)_{\bar k},\mathbb Z_p)^{[Z]}\subset 
H_{|Z|,et}^{2d}((X\times\square^n)_{\bar k},\mathbb Z_p), \\ 
H_{et}^{2d-1}(((X\times\square^n)\backslash|Z|)_{\bar k},\mathbb Z_p)^{[Z]}\subset 
H_{et}^{2d-1}(((X\times\square^n)\backslash|Z|)_{\bar k},\mathbb Z_p)
\end{eqnarray*}
are the subobjects given by the pullback by the class of $Z$ (see section 2).
\item[(ii)]Let $Z\in\mathcal Z^d(X,n)^{\partial=0}_{hom}$. Then $AJ_{et,p}(X)(Z)=0$ if and only if there exist 
\begin{equation*}
\alpha\in H_{et}^{2d-1}(((X\times\square^n)\backslash|Z|)_{\bar k},\mathbb Z_p)^{[Z]} 
\end{equation*}
such that $\alpha\in H_{et}^{2d-1}(((X\times\square^n)\backslash|Z|)_{\bar k},\mathbb Z_p)^G$
and $\partial\alpha\neq 0$.
\end{itemize}
\end{thm}

\begin{proof}
\noindent(i):See \cite{Jansen}.

\noindent(ii): Follows from (i).
\end{proof}

\begin{defi}\label{AJpdefS}
Let $f:X\to S$ a morphism with $S,X\in\SmVar(k)$. Let $j:S^o\hookrightarrow S$ an open subset such that 
for all $j\in\mathbb Z$, $j^*H^jRf_*\mathbb Z_{p,X_{\bar k}}$ is a local system.
Let $d,n\in\mathbb N$. We have then, denoting $X^o:=X\times_SS^o$ and using definition \ref{AJpdef}, the map
\begin{eqnarray*}
AJ_{et,p}(X^o/S^o):\mathcal Z^d(X,n)^{f,\partial=0}_{fhom}\to
\Gamma(S^o,\oplus_{s\in S^o_{(0)}}i_{s*}\Ext^1_{Gal(\bar k/k(s))}(\bar k,H_{et}^{2d-n-1}(X_{s,\bar k},\mathbb Z_p))), \\
Z\mapsto AJ_{et,p}(X^o/S^o)(Z):=\nu^{et,p}_Z := \\
((s\in S^o_{(0)})\mapsto (AJ_{et,p}(X_s)(Z_s)\in\Ext^1_{Gal(\bar k/k(s))}(\bar k,H_{et}^{2d-n-1}(X_{s,\bar k},\mathbb Z_p))))
\end{eqnarray*}
where $\mathcal Z^d(X,n)^{f,\partial=0}_{fhom}\subset\mathcal Z^d(X,n)^{f,\partial=0}$ denote the subabelian group
consisting of algebraic cycles $Z$ with $Z_s:=i_s^*Z\in\mathcal Z^d(X_s,n)_{hom}$.
Recall that $i_s:\left\{s\right\}\hookrightarrow S_{(0)}\subset S$ is a closed Zariski point of $S$.
\end{defi}

We now localize, for each prime number $l$ and each embedding $\sigma_l:k\hookrightarrow\mathbb C_l$
the definition given above. 

\begin{defi}\label{AJpdefl}
Let $X\in\SmVar(k)$ irreducible. 
Let $\bar X\in\PSmVar(k)$ a compactification of $X$ with $D:=\bar X\backslash X\subset\bar X$ a normal crossing divisor.
Let $\sigma_l:k\hookrightarrow\mathbb C_l$ an embedding.
Then $k\subset\bar k\subset\mathbb C_l$, where $\bar k$ is the algebraic closure of $k$
and $k\subset\hat k_{\sigma_l}\subset\mathbb C_l$ where $\hat k_{\sigma_l}$ is the completion of $k$ with respect to $\sigma_l$. 
Denote $\hat G_{\sigma_l}:=Gal(\mathbb C_l\hat k_{\sigma_l})$. The cycle class map
\begin{eqnarray*}
\mathcal R_{X,\sigma_l}^{et,p}:\mathcal Z^d(X,n)^{\partial=0}
\to H^{2d-n}_{p_X(|Z|),et}(X_{\hat k_{\sigma_l}},D_{\hat k_{\sigma_l}},\hat{\mathbb Z_p})
\to H_{et}^{2d-n}(X_{\hat k_{\sigma_l}},D_{\hat k_{\sigma_l}},\hat{\mathbb Z_p}), 
\end{eqnarray*}
to continuous etale cohomology induces the etale Abel Jacobi map for higher Chow groups
\begin{eqnarray*}
AJ_{et,p,\sigma_l}(X):\mathcal Z^d(X,n)^{\partial=0}_{hom}\to
\Ext^1_{\hat G_{\sigma_l}}(\mathbb C_l,H_{et}^{2d-1-n}(X_{\mathbb C_l},D_{\mathbb C_l},\mathbb Z_p)), \\ 
Z\mapsto AJ_{et,p,\sigma_l}(X)(Z):=L^1\mathcal R_{X,\sigma_l}^{et,p}(Z)/L^2\mathcal R_{X,\sigma_l}^{et,p}(Z),
\end{eqnarray*}
where $L$ is the filtration given by the Leray spectral sequence of the map of sites 
$a_X::X_{\hat k_{\sigma_l}}^{et}\to\Spec(\hat k_{\sigma_l})^{et}$.
We have then the commutative diagram
\begin{equation*}
\xymatrix{\, & \Ext^1_G(\bar k,H_{et}^{2d-1-n}(X_{\bar k},D_{\bar k},\mathbb Z_p))
\ar[dd]^{\Ext^1(\ad(\pi_{k/\hat k_{\sigma_l}}^*,\pi_{k/\hat k_{\sigma_l}*})(-),
\ad(\pi_{k/\hat k_{\sigma_l}}^*,\pi_{k/\hat k_{\sigma_l}*})(-))} \\
\mathcal Z^d(X,n)^{\partial=0}_{hom}\ar[ru]^{AJ_{et,p}(X)}\ar[rd]^{AJ_{et,p,\sigma_l}(X)} & \, \\
\, & \Ext^1_{\hat G_{\sigma_l}}(\mathbb C_l,H_{et}^{2d-1-n}(X_{\mathbb C_l},D_{\mathbb C_l},\mathbb Z_p))
=\Ext^1_{\hat G_{\sigma_l}}(\bar k,H_{et}^{2d-1-n}(X_{\bar k},D_{\bar k},\mathbb Z_p))}
\end{equation*}
where the right column arrow is given by the restriction 
$\pi_{k/\hat k_{\sigma_l}}:\hat G_{\sigma_l}\hookrightarrow G$.
\end{defi}

\begin{thm}\label{AJkl}
Let $X\in\PSmVar(k)$. Let $\sigma_l:k\hookrightarrow\mathbb C_l$ an embedding and
$k\subset\hat k_{\sigma_l}\subset\mathbb C_l$ the completion of $k$ with respect to $\sigma_l$.
\begin{itemize}
\item[(i)]For $Z\in\mathcal Z^d(X,n)^{\partial=0}_{hom}$, we have
\begin{eqnarray*}
AJ_{et,p,\sigma_l}(X)(Z)=(0\to H_{et}^{2d-1-n}(X_{\mathbb C_l},\mathbb Z_p) 
\xrightarrow{j^*}H_{et}^{2d-1}(((X\times\square^n)\backslash|Z|)_{\mathbb C_l},\mathbb Z_p)^{[Z]} \\
\xrightarrow{\partial}H_{|Z|,et}^{2d}((X\times\square^n)_{\mathbb C_l},\mathbb Z_p)^{[Z]}=\mathbb C_l(n-d)\to 0)
\end{eqnarray*}
with $j:X\times\square^n\backslash|Z|\hookrightarrow X$ the open embedding and 
\begin{eqnarray*}
H_{|Z|,et}^{2d}((X\times\square^n)_{\mathbb C_l},\mathbb Z_p)^{[Z]}\subset 
H_{|Z|,et}^{2d}((X\times\square^n)_{\mathbb C_l},\mathbb Z_p), \\
H_{|Z|,et}^{2d}(((X\times\square^n)\backslash|Z|)_{\mathbb C_l},\mathbb Z_p)^{[Z]}\subset 
H_{|Z|,et}^{2d}(((X\times\square^n)\backslash|Z|)_{\mathbb C_l},\mathbb Z_p),
\end{eqnarray*}
are the subobjects given by the pullback by the class of $Z$ (see section 2).
\item[(ii)]Let $Z\in\mathcal Z^d(X,n)^{\partial=0}_{hom}$. Then $AJ_{et,p,\sigma_l}(X)(Z)=0$ if and only if there exist 
\begin{equation*}
\alpha\in H_{et}^{2d-1}(((X\times\square^n)\backslash|Z|)_{\mathbb C_l},\mathbb Z_p)^{[Z]} 
\end{equation*}
such that $\partial\alpha\neq 0$ and 
$\alpha\in H_{et}^{2d-1}(((X\times\square^n)\backslash|Z|)_{\mathbb C_l},\mathbb Z_p)^{\hat G_{\sigma_l}}$.
\end{itemize}
\end{thm}

\begin{proof}
Similar to the proof of theorem \ref{AJk}.
\end{proof}

\begin{defi}\label{AJpdefSl}
Let $f:X\to S$ a morphism with $S,X\in\SmVar(k)$. Let $j:S^o\hookrightarrow S$ an open subset such that 
for all $j\in\mathbb Z$, $j^*H^jRf_*\mathbb Z_{p,X_{\bar k}}$ is a local system. Let $d,n\in\mathbb N$. 
Let $\sigma_l:k\hookrightarrow\mathbb C_l$ an embedding and 
$k\subset\hat k_{\sigma_l}\subset\mathbb C_l$ the completion of $k$ with respect to $\sigma_l$.
Denoting $X^o:=X\times_SS^o$, we consider
\begin{eqnarray*}
AJ_{et,p,\sigma_l}(X^o/S^o):\mathcal Z^d(X,n)^{f,\partial=0}_{fhom}\to
\Gamma(S^o,\oplus_{s\in S_{(0)}^o}
i_{s*}\Ext^1(Gal(\mathbb C_l/\hat k_{\sigma_l}(s)),H_{et}^{2d-n-1}(X_{s,\bar k},\mathbb Z_p))), \\
Z\mapsto AJ_{et,p,\sigma_l}(X^o/S^o)(Z):=\nu^{et,p}_{Z,\sigma_l}:= \\
((s\in S_{(0)}^o)\mapsto (AJ_{et,p,\sigma_l}(X_s)(Z_s)
\in\Ext^1_{Gal(\mathbb C_l/\hat k_{\sigma_l}(s))}(\bar k,H_{et}^{2d-n-1}(X_{s,\bar k},\mathbb Z_p))))
\end{eqnarray*}
where $\mathcal Z^d(X,n)^{f,\partial=0}_{fhom}\subset\mathcal Z^d(X,n)^{f,\partial=0}$ denote the subabelian group
consisting of algebraic cycles $Z$ with $Z_s:=i_s^*Z\in\mathcal Z^d(X_s,n)_{hom}$.
Recall that $i_s:\left\{s\right\}\hookrightarrow S_{(0)}\subset S$ is a closed Zariski point of $S$.
\end{defi}

\section{The vanishing of the etale Abel Jacobi map implies the vanishing of the complex Abel Jacobi map}

The $p$ adic Hodge theory for open varieties implies the following main theorem : 

\begin{thm}\label{main}
Let $k$ a field of finite type over $\mathbb Q$. 
Denote $\bar k$ the algebraic closure of $k$ and $G=Gal(\bar k/k)$ its absolute Galois group.
Let $X\in\PSmVar(k)$. Consider an embedding $\sigma:k\hookrightarrow\mathbb C$.  
Let $p\in\mathbb N\backslash\delta(k,X)$ be a prime number (any by finitely many). 
Let $\sigma_p:k\hookrightarrow\mathbb C_p$ be an embedding.
Denote $\hat k_{\sigma_p}\subset\mathbb C_p$ the completion of $k$ with respect to $\sigma_p$.
Let $Z\in\mathcal Z^d(X,n)^{\partial=0}_{hom}$. Consider the exact sequences
\begin{itemize}
\item $0\to H^{2d-1-n}_{et}(X_{\bar k},\mathbb Z_p)\xrightarrow{j^*}  
H^{2d-1}_{et}(((X\times\square^n)\backslash|Z|)_{\bar k},\mathbb Z_p)\xrightarrow{\partial} 
H_{et,|Z|}^{2d}((X\times\square^n)_{\bar k},\mathbb Z_p)^0\to 0$ 
\item $0\to H^{2d-1-n}_{DR}(X)\xrightarrow{j^*}H^{2d-1}_{DR}((X\times\square^n)\backslash|Z|)
\xrightarrow{\partial}H_{DR,|Z|}^{2d}(X\times\square^n)^0\to 0$,
\end{itemize}
where $j:(X\times\square^n)\backslash|Z|\hookrightarrow X\times\square^n$ is the open embedding.
Consider the following assertions :
\begin{itemize}
\item[(i)] $AJ_{et,p}(X)(Z)=0\in\Ext_1^G(\bar k,H_{et}^{2d-1-n}(X_{\bar k},\mathbb Z_p)(d-n))$,
\item[(i)']there exist $\alpha\in H_{et}^{2d-1}(((X\times\square^n)\backslash|Z|)_{\bar k},\mathbb Z_p)(d)^{[Z]}$ 
such that $\alpha\in H_{et}^{2d-1}(((X\times\square^n)\backslash|Z|)_{\bar k},\mathbb Z_p)(d)^G$ 
and $\partial\alpha\neq 0$,
\item[(ii)] $AJ_{et,p,\sigma_p}(X)(Z)=0\in
\Ext_1^{\hat G_{\sigma_p}}(\mathbb C_p,H_{et}^{2d-1-n}(X_{\mathbb C_p},\mathbb Z_p)(d-n))$,
obviously (i) implies (ii),
\item[(ii)']there exist $\alpha\in H_{et}^{2d-1}(((X\times\square^n)\backslash|Z|)_{\mathbb C_p},\mathbb Z_p)(d)^{[Z]}$ 
such that $\alpha\in H_{et}^{2d-1}(((X\times\square^n)\backslash|Z|)_{\mathbb C_p},\mathbb Z_p)(d)^{\hat G_{\sigma_p}}$ 
and $\partial\alpha\neq 0$,
obviously (i)' implies (ii)',
\item[(iii)]there exist $w\in H^{2d-1}_{DR}(((X\times\square^n)\backslash|Z|)_{\hat k_{\sigma_p}})^{[Z]}$ such that 
$w\in F^dH^{2d-1}_{DR}(((X\times\square^n)\backslash|Z|)_{\hat k_{\sigma_p}})$, 
\begin{equation*}
w\in H^{2d-1}OL_{X\times\square^n}(\mathbb H_{pet}^{2d-1}(((X\times\square^n)\backslash|Z|)_{\mathbb C_p},
\Omega^{\bullet}_{(X\times\square^n)_{\mathbb C_p}^{pet},\log,\mathcal O}))
\end{equation*}
and $\partial w\neq 0$,
\item[(iv)] there exist $w\in H^{2d-1}_{DR}(((X\times\square^n)\backslash|Z|)_{\mathbb C})^{[Z]}$ 
such that $w\in F^dH^{2d-1}_{DR}(((X\times\square^n)\backslash|Z|)_{\mathbb C})$,  
\begin{equation*}
H^{2d-1}ev((X\times\square^n)\backslash|Z|)(w)\in 
H_{\sing}^{2d-1}(((X\times\square^n)\backslash|Z|)_{\mathbb C}^{an},2i\pi\mathbb Q),
\end{equation*}
and $\partial w\neq 0$,
\item[(iv)'] there exist an integer $m\in\mathbb N$ such that $m\cdot AJ_{\sigma}(X)(Z)=0\in J_{\sigma}^{2d-1-n,d}(X)$,
\end{itemize}
where the inclusion $OL_X:\Omega^{\bullet}_{X_{\mathbb C_p}^{pet},\log,\mathcal O}
\hookrightarrow\Omega^{\bullet}_{X_{\mathbb C_p}^{pet}}$ of $C(X_{\mathbb C_p}^{pet})$
is the subcomplex of logarithmic forms.
Then (i) is equivalent to (i)', (ii) is equivalent to (ii)', (ii)' implies (iii), (iii) implies (iv), 
(iv) is equivalent to (iv)'. Hence (i) implies (iv)'.
\end{thm}

\begin{proof}
\noindent (i) is equivalent to (i)': see theorem \ref{AJk}(ii),

\noindent (ii) is equivalent to (ii)': see theorem \ref{AJkl}(ii),

\noindent (ii)' implies (iii):By proposition \ref{LatticeLog} and proposition \ref{keypropCp}(i), we have :
\begin{eqnarray*}
H^{2d-1-n}(X_{\mathbb C_p},\mathbb Z_p)(d-n)^{\hat G_{\sigma_p}}
=F^dH^{2d-1-n}(X_{\hat k_{\sigma_p}})\cap H^{2d-1-n}_{et}(X_{\mathbb C_p},\mathbb Z_p) \\
=H^{2d-1-n}OL_X(\mathbb H_{pet}^{2d-1-n}(X_{\hat k_{\sigma_p}},
\Omega^{\bullet\geq d}_{X_{\hat k_{\sigma_p}}^{pet},\log,\mathcal O}\otimes\mathbb Z_p))
\end{eqnarray*}
and 
\begin{eqnarray*}
H^{2d-1}((X\times\square^n\backslash|Z|)_{\mathbb C_p},\mathbb Z_p)(d-n)^{\hat G_{\sigma_p}}
=F^dH^{2d-1}((X\times\square^n\backslash|Z|)_{\hat k_{\sigma_p}})\cap 
H^{2d-1}_{et}((X\times\square^n\backslash|Z|)_{\mathbb C_p},\mathbb Z_p) \\
=H^{2d-1}OL_X(\mathbb H_{pet}^{2d-1-n}((X\times\square^n\backslash|Z|)_{\hat k_{\sigma_p}},
\Omega^{\bullet\geq d}_{(X\times\square^n\backslash|Z|)_{\hat k_{\sigma_p}}^{pet},\log,\mathcal O}\otimes\mathbb Z_p)).
\end{eqnarray*}
We consider then a basis 
\begin{eqnarray*}
(\alpha_1,\cdots,\alpha_s)\in H^{2d-1-n}(X_{\mathbb C_p},\mathbb Z_p)(d-n)^{\hat G_{\sigma_p}}, \\ 
\mbox{with} \; w(\alpha_1),\cdots,w(\alpha_s)\in 
H^{2d-1-n}OL_X(\mathbb H_{pet}^{2d-1-n}(X_{\mathbb C_p},\Omega^{\bullet\geq d}_{X_{\mathbb C_p}^{pet},\log,\mathcal O})).
\end{eqnarray*}
By assumption on $\alpha$, it extend to a basis
\begin{eqnarray*}
(\alpha_1,\cdots,\alpha_s,\alpha_{s+1})\in 
H^{2d-1}((X\times\square^n\backslash|Z|)_{\mathbb C_p},\mathbb Z_p)^{[Z]}(d)^{\hat G_{\sigma_p}} \\
\mbox{with} \; w(\alpha_1),\cdots,w(\alpha_{s+1})\in 
H^{2d-1}OL_{X\times\square^n}(\mathbb H_{pet}^{2d-1}(((X\times\square^n)\backslash|Z|)_{\mathbb C_p},
\Omega^{\bullet\geq d}_{(X\times\square^n)_{\mathbb C_p}^{pet},\log,\mathcal O})).
\end{eqnarray*}
We then take $w=w(\alpha_{s+1})$.

\noindent (ii) implies (iii):follows from proposition \ref{keypropC} with 
$w:=\pi_{k/\mathbb C}((X\times\square^n)\backslash|Z|)^*w$,

\noindent (iii) is equivalent to (iii)' : see theorem \ref{AJC}(ii).
\end{proof}

It implies the following :

\begin{cor}\label{maincor} 
\begin{itemize}
\item[(i)]Let $k$ a field of finite type over $\mathbb Q$. Denote $\bar k$ the algebraic closure of $k$.
Let $f:X\to S$ a smooth projective morphism with $S,X\in\SmVar(k)$.
Consider an embedding $\sigma:k\hookrightarrow\mathbb C$. 
Then we have $\bar k\subset\mathbb C$ the canonical algebraic closure of $k$ inside $\mathbb C$.
Let $p\in\mathbb N$ a prime number. Then for $Z\in\mathcal Z^d(X,n)_{fhom}^{f,\partial=0}$, we have
\begin{equation*}
V_{tors}(\nu^{et,p}_Z)_{\mathbb C}\subset V_{tors}(\nu_Z)\subset S_{\mathbb C}
\end{equation*}
where
\begin{itemize}
\item $V(\nu_Z)\subset V_{tors}(\nu_Z)\subset S_{\mathbb C}$ is the zero locus, resp. torsion locus,
of the complex normal function 
\begin{equation*}
\nu_Z=:AJ_{\sigma}(X/S)(Z)\in\Gamma(S^{an}_{\mathbb C},J_{\sigma}^{2d-1-n,d}(X/S)) 
\end{equation*}
associated to $Z$ (see proposition-definition \ref{AJdefS}),
\item $V(\nu^{et,p}_Z)\subset V_{tors}(\nu^{et,p}_Z)\subset S$ is the zero locus, resp. torsion locus 
of the etale normal function 
\begin{equation*}
\nu^{et,p}_Z\in\Gamma(S,\oplus_{s\in S_{(0)}}
i_{s*}\Ext^1_{Gal(\bar k/k(s))}(\bar k,H_{et}^{2d-n-1}(X_{s,\bar k},\mathbb Z_p))(d-n)) 
\end{equation*}
associated to $Z$ (see definition \ref{AJpdefS}) and 
\begin{equation*}
V(\nu^{et,p}_Z)_{\mathbb C}:=\pi_{k/\mathbb C}(S)^{-1}(V(\nu^{et,p}_Z)), \;
V_{tors}(\nu^{et,p}_Z)_{\mathbb C}:=\pi_{k/\mathbb C}(S)^{-1}(V_{tors}(\nu^{et,p}_Z))
\end{equation*}
where we recall $\pi_{k/\mathbb C}(S):S_{\mathbb C}\to S$ is the projection.
\end{itemize}
\item[(ii)] Let $\sigma:k\hookrightarrow\mathbb C$ a subfield which is of finite type over $\mathbb Q$.
Denote $\bar k\subset\mathbb C$ the algebraic closure of $k$.
Let $f:X\to S$ a smooth projective morphism with $S,X\in\SmVar(k)$. 
Then for $Z\in\mathcal Z^d(X,n)_{fhom}^{f,\partial=0}$,
the zero locus $V(\nu_Z)\subset S_{\mathbb C}$ of the complex normal function 
\begin{equation*}
\nu_Z=:AJ_{\sigma}(X/S)(Z)\in\Gamma(S^{an}_{\mathbb C},J_{\sigma}^{2d-1-n,d}(X/S)) 
\end{equation*} 
associated to $Z$ is defined over $\bar k$ if $V(\nu^{et,p}_Z)\neq\emptyset$.
\end{itemize}
\end{cor}

\begin{proof}
\noindent(i):Follows immediately from theorem \ref{main} since 
for $s\in S_{(0)}$, $k(s)$ is of finite type over $\mathbb Q$ and for $s'\in\pi_{k/\mathbb C}(S)^{-1}(s)$
denoting $\sigma':k(s)\hookrightarrow\mathbb C$ the embedding given by $s'$, we have by definition
\begin{itemize}
\item $\nu_Z(s')=:AJ_{\sigma}(X/S)(Z)(s'):=AJ_{\sigma'}(X_s)(Z_s)\in J_{\sigma'}^{2d-1-n,d}(X_s)$.
\item $\nu^{et,p}_Z(s)=:AJ_{et,p}(X/S)(Z)(s):=AJ_{et,p}(X_s)(Z_s)
\in\Ext^1_{Gal(\bar k/k(s))}(\bar k,H_{et}^{2d-n-1}(X_{s,\bar k},\mathbb Z_p)(d-n))$.
\end{itemize}

\noindent(ii): Since $V(\nu^{et,p}_Z)\subset S$ contain a $\bar k$ point, 
$V(\nu_Z)\subset S_{\mathbb C}$ contain a $\bar k$ point by (i). Hence by 
the work of \cite{SaitoAJ} or \cite{Charles2}, $V(\nu_Z)\subset S_{\mathbb C}$ is defined over $\bar k$. 
\end{proof}

\section{Algebraicity of the zero locus of etale normal functions and the locus of Hodge Tate classes}

Let $k$ a field of finite type over $\mathbb Q$.
Let $f:X\to S$ a smooth proper morphism with $S,X\in\SmVar(k)$ connected. Let $p$ a prime number.
Let $Z\in\mathcal Z^d(X,n)^{\partial=0,f}_{hom}$.
By the definition, we have
\begin{equation*}
V(\nu_Z^{et,p})\subset\cap_{l\in\mathbb N, l \mbox{prime}}\cap_{\sigma_l:k\hookrightarrow\mathbb C_l}.
V(\nu_{Z,\sigma_l}^{et,p})\subset S_{(0)}
\end{equation*}
and
\begin{equation*}
V_{tors}(\nu_Z^{et,p})\subset\cap_{l\in\mathbb N, l \mbox{prime}}\cap_{\sigma_l:k\hookrightarrow\mathbb C_l}.
V_{tors}(\nu_{Z,\sigma_l}^{et,p})\subset S_{(0)}.
\end{equation*}
In this section, we investigate the algebraicity of $V_{tors}(\nu_Z^{et,p})\subset S$
and of $V_{tors}(\nu_{Z,\sigma_l}^{et,p})\subset S$, $\sigma_l:k\hookrightarrow\mathbb C_l$.
Since $Z$ is defined over $k$, we expect that the inclusion
$V_{tors}(\nu_Z^{et,p})\subset V(\nu_{Z,\sigma_l}^{et,p})$ is an equality for all primes $l\in\mathbb N$.

\begin{rem}
Let $k$ a field of finite type over $\mathbb Q$.
Let $f:X\to S$ a smooth proper morphism with $S,X\in\SmVar(k)$ connected. Let $p$ a prime number.
Let $Z\in\mathcal Z^d(X,n)^{\partial=0,f}_{hom}$. We can show, using \cite{JansenHasse}, that we have in fact
\begin{equation*}
V_{tors}(\nu_Z^{et,p})=\cap_{l\in\mathbb N, l \mbox{prime}}\cap_{\sigma_l:k\hookrightarrow\mathbb C_l}.
V_{tors}(\nu_{Z,\sigma_l}^{et,p})\subset S_{(0)}.
\end{equation*}
We don't need this result so we don't give the details.
\end{rem}

Let $k$ a field of finite type over $\mathbb Q$. 
Let $p$ be a prime number. Let $\sigma_p:k\hookrightarrow\mathbb C_p$ be an embedding.
We have then $\hat k_{\sigma_p}$ the completion of $k$ with respect to $\sigma_p$
and we denote $O_{\hat k_{\sigma_p}}\subset\hat k_{\sigma_p}$ its ring of integers.
We then consider the canonical functor of Huber (see section 2)
\begin{equation*}
\mathcal R:\Var(\hat k_{\sigma_p})\to\HubSp(\hat k_{\sigma_p},O_{\hat k_{\sigma_p}})\to\Sch/O_{\hat k_{\sigma_p}}, 
X\mapsto\mathcal R(X)=X^{\mathcal O}
\end{equation*}
which associated to a variety over a $p$ adic field its canonical integral model.
Let $f:X\to S$ a smooth projective morphism with $S,X\in\SmVar(k)$. 
Let $Z\subset X$ a closed subset and $j:U=X\backslash Z\hookrightarrow X$ the open complementary subset. 
We have then $f:=f_{\hat k_{\sigma_p}}:(X,Z)_{\hat k_{\sigma_p}}\to S_{\hat k_{\sigma_p}}$ 
the morphism in $\SmVar^2(\hat k_{\sigma_p})$ induced by the scalar extension functor and 
\begin{equation*}
f^{\mathcal O}:=\mathcal R(f_{\hat k_{\sigma_p}}):(X,Z)^{\mathcal O}_{\hat k_{\sigma_p}}
\to S^{\mathcal O}_{\hat k_{\sigma_p}}
\end{equation*}
its canonical integral model in $\Sch^2/O_{\hat k_{\sigma}}$ to which we denote 
\begin{equation*}
f:=f^{\mathcal O}:(X^{\mathcal O}_{\hat k_{\sigma_p}},N_{U,\mathcal O}):=
(X^{\mathcal O}_{\hat k_{\sigma_p}},(M_Z,N_{\mathcal O}))
\to (S^{\mathcal O}_{\hat k_{\sigma_p}},N_{\mathcal O})
\end{equation*}
the corresponding morphism in $\logSch$, 
where for $K$ a $p$ adic field and $Y\in\Sch/O_K$, $(Y,N_{\mathcal O}):=(Y,M_{Y_k})$ with $k=O_K/(\pi)$ the residual field.
We have then the morphisms of sites
\begin{equation*}
v_{X,N}:(X^{\mathcal O}_{\hat k_{\sigma_p}},N_{U,\mathcal O})^{Falt}\to
(X^{\mathcal O}_{\hat k_{\sigma_p}},N_{U,\mathcal O})^{ket}, \;
u_{X,N}:(X_{\mathbb C_p},M_{Z_{\mathbb C_p}})^{ket}\to
(X^{\mathcal O}_{\hat k_{\sigma_p}},N_{U,\mathcal O})^{Falt}
\end{equation*}
where $(X^{\mathcal O}_{\hat k_{\sigma_p}},N_{U,\mathcal O})^{Falt}$ denote the Falting site, 
and for $(Y,N)\in\logSch$, $(Y,N)^{ket}\subset\logSch/(Y,N)$ is the small Kummer etale site.
If $(X^{\mathcal O}_{\hat k_{\sigma_p}},N_{U,\mathcal O})$ is log smooth, we consider an hypercover
\begin{equation*}
a_{\bullet}:(X^{\mathcal O,\bullet}_{\hat k_{\sigma_p}},N_{U,\mathcal O})\to
(X^{\mathcal O}_{\hat k_{\sigma_p}},N_{U,\mathcal O})
\end{equation*}
in $\Fun(\Delta,\logSch)$ by small log schemes in sense of \cite{ChinoisCrys}.
The main result of \cite{ChinoisCrys} say that
\begin{itemize}
\item if $(X^{\mathcal O}_{\hat k_{\sigma_p}},N_{U,\mathcal O})$ is log smooth,
the embedding in $C((X^{\mathcal O}_{\hat k_{\sigma_p}})^{Falt})$
\begin{equation*}
\alpha(U):(\mathbb B_{st,X_{\hat k_{\sigma_p}}},N_{U,\mathcal O})\hookrightarrow 
a_{\bullet*}DR(X^{\mathcal O,\bullet}_{\hat k_{\sigma_p}}/O_{\hat k_{\sigma_p}})
(O\mathbb B_{st,X^{\bullet}_{\hat k_{\sigma_p}}},N_{U,\mathcal O})
\end{equation*}
is a filtered quasi-isomorphism compatible with the action of $Gal(\mathbb C_p,\hat k_{\sigma_p})$,
the Frobenius $\phi_p$ and the monodromy $N$,
note that we have a commutative diagram in $C_{fil}(X^{an,pet}_{\mathbb C_p})$
\begin{equation*}
\xymatrix{\mathbb B_{st,X_{\mathbb C_p},\log}\ar[rr]^{\alpha(U)}\ar[d]^{\subset} & \, &
O\mathbb B_{st,X_{\mathbb C_p},\log}\otimes_{O_X}\Omega_{X_{\mathbb C_p}}^{\bullet}(\log D_{\mathbb C_p})
\ar[d]^{\subset} \\
\mathbb B_{dr,X_{\mathbb C_p},\log}\ar[rr]^{\alpha(U)} & \, &
O\mathbb B_{dr,X_{\mathbb C_p},\log}\otimes_{O_X}\Omega_{X_{\mathbb C_p}}^{\bullet}(\log D_{\mathbb C_p})}
\end{equation*}
see section 3,
\item if $f:(X^{\mathcal O}_{\hat k_{\sigma_p}},N_{U,\mathcal O})\to (S^{\mathcal O}_{\hat k_{\sigma_p}},N_{\mathcal O})$ 
is log smooth, the morphism in $D_{fil}((S^{\mathcal O}_{\hat k_{\sigma_p}})^{Falt})$
\begin{eqnarray*}
T(f,\mathbb B_{st}):Rf_*(\mathbb B_{st,X_{\hat k_{\sigma_p}}},N_{U,\mathcal O})
\xrightarrow{T(f,f,\otimes)(-)\circ\ad(u_{X,N}^*,Ru_{X,N,*})(-)}
Rf_*\mathbb Z_{p,(X_{\mathbb C_p},M_{Z_{\mathbb C_p}})^{ket}}\otimes_{\mathbb Z_p}\mathbb B_{st,S_{\hat k_{\sigma_p}}} \\
\xrightarrow{\ad(j^*,Rj_*)(\mathbb Z_{p,(X_{\mathbb C_p},M_{Z_{\mathbb C_p}})^{et}})}
R(f\circ j)_*\mathbb Z_{p,U_{\mathbb C_p}^{et}}\otimes_{\mathbb Z_p}\mathbb B_{st,S_{\hat k_{\sigma_p}}} 
\end{eqnarray*}
is an isomorphism, where the last map is an isomorphim by \cite{Il} theorem 7.4.
\end{itemize}
This gives if $(X^{\mathcal O}_{\hat k_{\sigma_p}},N_{U,\mathcal O})$ is log smooth, 
for each $j\in\mathbb Z$, a filtered isomorphism of filtered abelian groups
\begin{eqnarray*}
H^jR\alpha(U):H_{et}^j(U_{\mathbb C_p},\mathbb Z_p)\otimes_{\mathbb Z_p}\mathbb B_{st,\hat k_{\sigma_p}}
\xrightarrow{H^jT(a_X,\mathbb B_{st})^{-1}}
H_{et}^j((X,N)^{Falt})(\mathbb B_{st,X_{\hat k_{\sigma_p}}},N_{U,\mathcal O}) \\
\xrightarrow{H^jR\Gamma((X_{\hat k_{\sigma_p}}^{\mathcal O},N_{U,\mathcal O}),\alpha(U))}
H^j_{DR}(U_{\hat k_{\sigma_p}})\otimes_{\hat k_{\sigma_p}}\mathbb B_{st,\hat k_{\sigma_p}}
\end{eqnarray*}
compatible with the action of $Gal(\mathbb C_p/\hat k_{\sigma_p})$, of the Frobenius $\phi_p$ and the monodromy $N$.
More generally, this gives if $\hat k_{\sigma_p}(s)$ is unramified for all $s\in S_{\hat k_{\sigma_p}}$ and if
$f:(X^{\mathcal O}_{\hat k_{\sigma_p}},N_{U,\mathcal O})\to (S^{\mathcal O}_{\hat k_{\sigma_p}},N_{\mathcal O})$ 
is log smooth, an isomorphism in $\Shv_{fil,G,\phi_p,N}(S_{\hat k_{\sigma_p}})$
\begin{eqnarray*}
H^jf_*\alpha(U):R^jf_*\mathbb Z_{p,U_{\mathbb C_p}^{et}}\otimes_{\mathbb Z_p}\mathbb B_{st,S_{\hat k_{\sigma_p}}}
\xrightarrow{H^jT(f,\mathbb B_{st})^{-1}}Rf_*(\mathbb B_{st,X_{\hat k_{\sigma_p}}},N_{U,\mathcal O}) \\
\xrightarrow{H^jRf_*\alpha(U)}
H^j\int_fj_{*Hdg}(O_{U_{\hat k_{\sigma_p}}},F_b)\otimes_{O_{S_{\hat k_{\sigma_p}}}}O\mathbb B_{st,S_{\hat k_{\sigma_p}}}
\end{eqnarray*}
that is a filtered isomorphism compatible with the action of $G=Gal(\mathbb C_p/\hat k_{\sigma_p})$, 
of the Frobenius $\phi_p$ and the monodromy $N$, writing for short again $f=f\circ j$.

\begin{defi}\label{nuZkalgdef}
Let $k$ a field of finite type over $\mathbb Q$. 
Let $f:X\to S$ a smooth proper morphism with $S,X\in\SmVar(k)$. Let $p$ a prime number.
Let $Z\in\mathcal Z(X,n)^{f,\partial=0}_{fhom}$. Denote $U:=(X\times\square^n)\backslash|Z|$.
We have then the following exact sequence in $DRM(S)$ (see \cite{B5} for the definition of De Rham modules)
\begin{eqnarray*}
0\to E^{2d-1-n}_{DR}(X/S):=H^{2d-1-n}\int_f(O_X,F_b)
\xrightarrow{j^*}E^{2d-1}_{DR}(U/S)^{[Z]}:=(H^{2d-1}\int_fj_{*Hdg}(O_U,F_b))^{[Z]} \\
\xrightarrow{\partial}E^{2d}_{DR,Z}(X/S)^{[Z]}:=(H^{2d}\int_f\Gamma^{Hdg}_{|Z|}(O_{X\times\square^n},F_b))^{[Z]}\to 0.
\end{eqnarray*}
Recall (see section 2) that $(X\times\square^n,M_{|Z|})\in\logVar(k)$ 
denote log structure associated to $(X\times\square^n,|Z|)\in\SmVar^2(k)$.
There exist a finite set of prime numbers $\delta(S)$ such that for all prime $p\in\mathbb N\backslash\delta(S)$,
all embedding $\sigma_p:k\hookrightarrow\mathbb C_p$ and all $s\in S_{(0)}$, 
$\hat k_{\sigma_p}(s)$ is unramified over $\mathbb Q_p$,
where $\hat k_{\sigma_p}\subset\mathbb C_p$ the $p$ adic completion of $k$ with respect to $\sigma_p$.
Let $p\in\mathbb N\backslash\delta(S)$ a prime number and $\sigma_p:k\hookrightarrow\mathbb C_p$ be an embedding
and consider $\hat k_{\sigma_p}\subset\mathbb C_p$ the $p$ adic completion of $k$ with respect to $\sigma_p$.
\begin{itemize}
\item Take, using \cite{DeYong} theorem 8.2 
(after considering an integral model $(X,Z)^{\mathcal R}\in\Sch^2/\mathcal R$ 
over $\mathcal R\subset k$ of finite type over $\mathbb Z$ with function field $k$), 
for $p\in\mathbb N\backslash\delta^0(X/S)$, where $\delta^0(X/S)$ is a finite set,
an alteration $\pi^0:(X\times\square^n)^0\to X\times\square^n$, that is a generically finite morphism 
such that $((X\times\square^n)^0,\pi^{0,-1}(|Z|))_{\hat k_{\sigma_p}}^{\mathcal O}$ is semi-stable pair,  
that is $\pi^{0,-1}(|Z|)\subset((X\times\square^n)^0$ is a normal crossing divisor and
$(((X\times\square^n)^0,\pi^{0,-1}(|Z|))_{\hat k_{\sigma_p}}^{\mathcal O})_t$ has semi-stable reduction 
where $t\in\Spec(O_{\hat k_{\sigma_p}})$ is the closed point.
Then there exists a closed subset $\Delta\subset S$ such that for all $s\in S^o:=S\backslash\Delta$,
$((X\times\square^n)^0,\pi^{0,-1}(|Z|))^{\mathcal O}_{\hat k_{\sigma_p},s}$ is a semi-stable pair.
\item Take using \cite{DeYong} theorem 8.2,
(after considering an integral model $(X_{\Delta},Z_{\Delta})^{\mathcal R}\in\Sch^2/\mathcal R$ 
over $\mathcal R\subset k$ of finite type over $\mathbb Z$ with function field $k$), 
for $p\in\mathbb N\backslash\delta^1(X/S)$, where $\delta^1(X/S)$ is a finite set, an alteration
$\pi^1:(X_{\Delta}\times\square^n)^1\to X_{\Delta}\times\square^n$ 
such that $((X_{\Delta}\times\square^n)^1,|Z_{\Delta}|)_{\hat k_{\sigma_p}}^{\mathcal O}$ is a semi-stable pair. 
Then there exists a closed subset $\Delta^2\subset\Delta$ such that for all $s\in S^1:=\Delta\backslash\Delta^2$,
$((X_{\Delta}\times\square^n)^1,|Z|_{\Delta})_{\hat k_{\sigma_p},s}^{\mathcal O}$ is a semi-stable pair.
\item Go on by induction.
\end{itemize}
We obtain by the above finite induction, for $p\in\mathbb N\backslash\delta(S,X/S)$,
with $\delta(S,X/S):=\delta(S)\cup(\cup_{\alpha\in\Lambda}\delta^{\alpha}(X/S))$,
a stratification $S=\sqcup_{\alpha\in\Lambda}S^{\alpha}$,
$\Lambda$ being a finite set, by locally closed subset $S^{\alpha}\subset S$, and alterations
(i.e. generically finite morphisms) 
$\pi^{\alpha}:(X_{S^{\alpha}}\times\square^n)^{\alpha}\to X_{S^{\alpha}}\times\square^n$ 
such that 
\begin{equation*}
f\circ\pi^{\alpha}:(X_{S^{\alpha}}\times\square^n)_{\hat k_{\sigma_p}}^{\alpha,\mathcal O},N_{U^{\alpha},\mathcal O})
\to (S_{\hat k_{\sigma_p}}^{\alpha},N_{\mathcal O})
\end{equation*}
is log smooth, that is for all $s\in S^{\alpha}$, 
$((X_{S^{\alpha}}\times\square^n)^{\alpha},\pi^{\alpha,-1}(|Z_{S^{\alpha}}|))^{\mathcal O}_{\hat k_{\sigma_p},s}$ 
is a semi-stable pair. 
We then set
\begin{eqnarray*}
T:=p_S(((\sqcup_{\alpha\in\Lambda}
(E^{2d-1}_{DR}(U^{\alpha}_{\hat k_{\sigma_p}}/S_{\hat k_{\sigma_p}}^{\alpha})
\otimes_{O_{S^{\alpha}_{\hat k_{\sigma_p}}}}O\mathbb B_{st,S^{\alpha}_{\hat k_{\sigma_p}}})^{\phi_p,N}) \\
\cap F^dE^{2d-1}_{DR}(U/S)\cap(E^{2d-1}_{DR}(U/S)^{[Z]}))\backslash E^{2d-1}_{DR}((X\times\square^n)/S))
\subset S
\end{eqnarray*}
and 
\begin{eqnarray*}
\hat T_{\sigma_p}:=
p_{S_{\hat k_{\sigma}}}(((\sqcup_{\alpha\in\Lambda}(E^{2d-1}_{DR}(U^{\alpha}_{\hat k_{\sigma_p}}/S_{\hat k_{\sigma_p}}^{\alpha})
\otimes_{O_{S_{\hat k_{\sigma_p}}}}O\mathbb B_{st,S_{\hat k_{\sigma_p}}})^{\phi_p,N})\cap \\ 
F^dE^{2d-1}_{DR}(U_{\hat k_{\sigma_p}}/S_{\hat k_{\sigma_p}})\cap E^{2d-1}_{DR}(U_{\hat k_{\sigma_p}}/S_{\hat k_{\sigma_p}})^{[Z]})
\backslash E_{DR}((X\times\square^n)_{\hat k_{\sigma_p}}/S_{\hat k_{\sigma_p}}))
\subset S_{\hat k_{\sigma_p}}
\end{eqnarray*}
where 
\begin{itemize}
\item $U^{\alpha}:=\pi^{\alpha,-1}(U_{S^{\alpha}})=
(X_{S^{\alpha}}\times\square^n)^{\alpha}\backslash\pi^{\alpha,-1}(|Z_{S^{\alpha}}|)$,
\item $E^{2d-1}_{DR}(U^{\alpha}/S^{\alpha}):=H^{2d-1}Rf_{*Hdg}(O_{U^{\alpha}},F_b)\in\Vect_{fil}(S^{\alpha})$,
\item $\phi_p$ is the Frobenius operator, $N$ is the monodromy operator,
\item $E^{2d-1}_{DR}(U_{S_{\alpha},\hat k_{\sigma_p}}/S^{\alpha}_{\hat k_{\sigma_p}})\subset 
E^{2d-1}_{DR}(U^{\alpha}_{\hat k_{\sigma_p}}/S^{\alpha}_{\hat k_{\sigma_p}})$
\item $\pi_{k/\hat k_{\sigma_p}}(S^{\alpha})^*E^{2d-1}_{DR}(U^{\alpha}/S^{\alpha})
\subset E^{2d-1}_{DR}(U^{\alpha}_{\hat k_{\sigma_p}}/S^{\alpha}_{\hat k_{\sigma_p}})$
is the canonical subset of closed points,
\item $p_S:E^{2d-1}_{DR}(U/S)\to S$ and
$p_{S_{\hat k_{\sigma}}}:E_{DR}(U_{\hat k_{\sigma_p}}/S_{\hat k_{\sigma_p}})\to S_{\hat k_{\sigma_p}}$ 
are the projections.
\end{itemize}
\end{defi}

\begin{lem}\label{nuZkalglem}
Let $G$ be a group. Consider a commutative diagram of $G$ modules
\begin{equation*}
\xymatrix{0\ar[r] & W\ar[r]\ar[d]^{\pi^*} & V\ar[r]^{\partial}\ar[d]^{\pi^*} & K\ar[r]\ar[d]^{\pi^*} & 0 \\
0\ar[r] & W'\ar[r] & V'\ar[r]^{\partial'} & K'\ar[r] & 0}
\end{equation*}
whose rows are exact sequence and $\pi^*:V\to V'$ is injective. 
Let $\alpha\in V$. 
Then $\alpha\in V^G$ and $\partial\alpha\neq 0$ if and only if $\pi^*\alpha\in V^{'G}$ and $\partial'\pi^*\alpha\neq 0$. 
\end{lem}

\begin{proof}
Follows from the fact that $<\alpha>$ define a splitting $W\oplus<\alpha>\subset V$ of $G$ modules.
\end{proof}

\begin{thm}\label{nuZkalg}
Let $k$ a field of finite type over $\mathbb Q$. 
Let $f:X\to S$ a smooth proper morphism with $S,X\in\SmVar(k)$ connected. 
Let $p\in\mathbb N\backslash\delta(S,X/S)$ a be prime number,
where $\delta(S,X/S)$ is the finite set given in definition \ref{nuZkalgdef}.
Let $\sigma_p:k\hookrightarrow\mathbb C_p$ be an embedding.
Consider $\hat k_{\sigma_p}\subset\mathbb C_p$ the $p$ adic completion of $k$ with respect to $\sigma_p$.
Let $Z\in\mathcal Z^d(X,n)^{\partial=0,f}_{hom}$. 
\begin{itemize}
\item[(i)] We have
\begin{equation*}
V_{tors}(\nu_Z^{et,p})=T\cap S_{(0)}\subset S,
\end{equation*}
where $T\subset S$ is given in definition \ref{nuZkalgdef}.
\item[(ii)] For each embedding $\sigma_p:k\hookrightarrow\mathbb C_p$, we have  
\begin{equation*}
V_{tors}(\nu_{Z,\sigma_p}^{et,p})_{\hat k_{\sigma_p}}
=\hat T_{\sigma_p}\cap S_{(0),\hat k_{\sigma_p}}\subset S_{\hat k_{\sigma_p}},
\end{equation*}
where $\hat T_{\sigma_p}\subset S_{\hat k_{\sigma_p}}$ is given in definition \ref{nuZkalgdef},
and for $V\subset S$ a subset, $V_{\hat k_{\sigma_p}}:=\pi_{k/\hat k_{\sigma_p}}(S)^{-1}(V)$,
where $\pi_{k/\hat k_{\sigma_p}}(S):S_{\hat k_{\sigma_p}}\to S$ being the projection.
\end{itemize}
\end{thm}

\begin{proof}
Let $\sigma_p:k\hookrightarrow\mathbb C_p$ be an embedding.
For each $\alpha\in\Lambda$, by the semi-stable comparaison theorem for 
\begin{equation*}
f^{\alpha}:=f\circ\pi^{\alpha}:
((X_{S^{\alpha}}\times\square^n)_{\hat k_{\sigma_p}}^{\alpha,\mathcal O},N^{\mathcal O}_{U^{\alpha}})
\to (S^{\alpha,\mathcal O}_{\hat k_{\sigma_p}},N_{\mathcal O}) 
\end{equation*}
(\cite{ChinoisCrys}) which is log smooth, 
we have the isomorphism in $\Shv_{fil,G,\phi_p,N}(S^{\alpha}_{\hat k_{\sigma_p}})$
\begin{eqnarray}\label{ChinoisCryseq}
H^jf^{\alpha}_*\alpha(U^{\alpha}):R^jf^{\alpha}_*\mathbb Z_{p,U^{\alpha,et}_{\mathbb C_p}}\otimes_{\mathbb Z_p}
\mathbb B_{st,S^{\alpha}_{\hat k_{\sigma_p}}}
\xrightarrow{\sim}
H^j\int_{f^{\alpha}}j_{*Hdg}(O_{U^{\alpha}_{\hat k_{\sigma_p}}},F_b)\otimes_{O_{S_{\hat k_{\sigma_p}}}}
O\mathbb B_{st,S^{\alpha}_{\hat k_{\sigma_p}}},
\end{eqnarray}
recall that since $p\in\mathbb N\backslash\delta(S)$, 
$\hat k_{\sigma_p}(s)$ is unramified for all $s\in S_{\hat k_{\sigma_p}}$.

\noindent(i):
Let $s\in S^{\alpha}_{(0)}$. Denote $G:=Gal(\bar k/k(s))$.
\begin{itemize}
\item The map $\pi_s^{\alpha}:(X_s\times\square^n)^{\alpha}\to X_s\times\square^n$ is generically finite
since $\pi^{\alpha}:(X\times\square^n)^{\alpha}\to X\times\square^n$ is generically finite and $f$ is flat.
Thus by lemma \ref{nuZkalglem} and theorem \ref{AJk}, $\nu^{et,p}_Z(s):=AJ^{et,p}(X_s)(Z_s)=0$ 
if and only if there exists 
\begin{equation*}
\alpha\in H^{2d-1}_{et}(U_{S^{\alpha},s},\mathbb Z_p)^{[Z_s]}
\end{equation*}
such that $\pi^{\alpha,-1}(\alpha)\in H^{2d-1}_{et}(U^{\alpha}_s,\mathbb Z_p)(d)^G$ and 
$\partial\pi^{\alpha,-1}(\alpha)\neq 0$.
\item On the other hand by (\ref{ChinoisCryseq}) and proposition \ref{keypropCp}(ii),
$\alpha'\in H^{2d-1}_{et}(U^{\alpha}_s,\mathbb Z_p)(d)^G$ if and only if (see definition \ref{walpha})
$w(\alpha')_k\in F^dH_{DR}^{2d-1}(U^{\alpha}_s)$ and
\begin{equation*}
w(\alpha')\in (H_{DR}^{2d-1}(U^{\alpha}_{s,\hat k_{\sigma_p}})\otimes_{\hat k_{\sigma_p}(s)}\mathbb B_{st})^{\phi_p,N}.
\end{equation*}
Moreover $w(\pi^{\alpha,-1}(\alpha))=\pi^{\alpha,-1}(w(\alpha))$.
\end{itemize}

\noindent(ii): Let $s\in S_{(0)}^{\alpha}$ and $s'\in\pi_{k/\hat k_{\sigma_p}}(S)^{-1}(s)$. 
\begin{itemize}
\item Since $\pi_s^{\alpha}:(X_s\times\square^n)^{\alpha}\to X_s\times\square^n$ is generically finite (see the proof of (i)) 
we have by lemma \ref{nuZkalglem} and theorem \ref{AJkl}, 
$\nu^{et,p}_{Z,\sigma_p}(s):=AJ^{et,p}_{\sigma_p}(X_s)(Z_s)=0$ if and only if there exists 
\begin{equation*}
\alpha\in H^{2d-1}_{et}(U_{S^{\alpha},s,\hat k_{\sigma_p}},\mathbb Z_p)^{[Z_s]}
\end{equation*}
such that 
$\pi^{\alpha,-1}(\alpha)\in H^{2d-1}_{et}(U^{\alpha}_{s,\hat k_{\sigma_p}},\mathbb Z_p)(d)^{\hat G_{\sigma_p}}$ 
and $\partial\alpha\neq 0$.
\item On the other hand by (\ref{ChinoisCryseq}), 
$\alpha'\in H^{2d-1}_{et}(U^{\alpha}_{s,\hat k_{\sigma_p}},\mathbb Z_p)(d)^{\hat G_{\sigma_p}}$ if and only if 
(see definition \ref{walpha}) $w(\alpha')\in F^dH_{DR}^{2d-1}(U^{\alpha}_{\hat k_{\sigma_p},s})$ and 
\begin{equation*}
w(\alpha')\in (H_{DR}^{2d-1}(U^{\alpha}_{s,\hat k_{\sigma_p}})\otimes_{\hat k_{\sigma_p}(s)}\mathbb B_{st})^{\phi_p,N}
=(H_{DR}^{2d-1}(U^{\alpha}_{s',\hat k_{\sigma_p}})\otimes_{\hat k_{\sigma_p}(s)}\mathbb B_{st})^{\phi_p,N}.
\end{equation*}
Moreover $w(\pi^{\alpha,-1}(\alpha))=\pi^{\alpha,-1}(w(\alpha))$.
\end{itemize}
\end{proof}

Proposition \ref{keypropC} for (i), the main result of \cite{ChinoisCrys}
together with proposition \ref{keypropCp} for (ii), gives the following :

\begin{thm}\label{HLk}
Let $f:X\to S$ be a smooth morphism, with $S,X\in\SmVar(k)$ 
over a field $k\subset\mathbb C$ of finite type over $\mathbb Q$. 
Take a compactification $f:X\xrightarrow{j}\bar X\xrightarrow{\bar f}S$ of $f$ where $\bar X\in\SmVar(k)$,
$j$ is an open embedding and $\bar f$ is projective. Denote $Z:=\bar X\backslash X$. Consider
\begin{equation*}
E^j_{DR}(X/S):=H^j\int_{\bar f}j_{*Hdg}(O_X,F_b)\in DRM(S),
\end{equation*}
see \cite{B5} for the definition of the category of De Rham modules.
Let $p\in\mathbb N\backslash\delta(S,\bar X/S)$ a prime number and $\sigma_p:k\hookrightarrow\mathbb C_p$ an embedding.
Denote $\hat k_{\sigma_p}\subset\mathbb C_p$ the $p$-adic completion of $k$ with respect to $\sigma_p$.
Consider, using definition \ref{nuZkalgdef}, a stratification $S=\sqcup_{\alpha\in\Lambda}S^{\alpha}$,
$\Lambda$ being a finite set, by locally closed subset $S^{\alpha}\subset S$, and alterations
(i.e. generically finite morphisms) $\pi^{\alpha}:X^{\alpha}\to X_{S^{\alpha}}$ such that 
\begin{equation*}
f_{\alpha}:=f\circ\pi^{\alpha}:(\bar X_{\hat k_{\sigma_p}}^{\alpha,\mathcal O},N_{Z,\mathcal O})
\to (S_{\hat k_{\sigma_p}}^{\alpha},N_{\mathcal O})
\end{equation*}
is log smooth. Consider for $j,d\in\mathbb Z$,
\begin{itemize} 
\item the locus of Hodge Tate classes
\begin{equation*}
HT^p_{j,d}(X/S):=(\sqcup_{\alpha\in\Lambda}(E^j_{DR}(X_{\hat k_{\sigma_p}}^{\alpha}/S_{\hat k_{\sigma_p}}^{\alpha})
\otimes_{O_{X_{\hat k_{\sigma_p}}^{\alpha}}}O\mathbb B_{st,X})^{\phi_p,N})\cap
F^dE^j_{DR}(X_{\hat k_{\sigma_p}}/S_{\hat k_{\sigma_p}})\subset E^j_{DR}(X_{\hat k_{\sigma_p}}/S_{\hat k_{\sigma_p}})
\end{equation*}
where
\begin{equation*}
E^j_{DR}(X_{\hat k_{\sigma_p}}/S_{\hat k_{\sigma_p}}):=H^j\int_{\bar f}j_{*Hdg}(O_{X_{\hat k_{\sigma_p}}},F_b)
=E^j_{DR}(X/S)\otimes_k\hat k_{\sigma_p}\in DRM(S_{\hat k_{\sigma_p}}),
\end{equation*}
\item the locus of Hodge classes
\begin{equation*}
HL_{j,d}(X/S):=HL_{j,d}(X_{\mathbb C}/S_{\mathbb C}):=
F^dE^j_{DR}(X_{\mathbb C}/S_{\mathbb C})\cap R^jf_*\mathbb Q_{X_{\mathbb C}^{an}}
\subset E^j_{DR}(X_{\mathbb C}/S_{\mathbb C})
\end{equation*}
where
\begin{equation*}
E^j_{DR}(X_{\mathbb C}/S_{\mathbb C}):=H^j\int_f(O_{X_{\mathbb C}},F_b)
=E^j_{DR}(X/S)\otimes_k\mathbb C\in DRM(S_{\mathbb C}).
\end{equation*}
\end{itemize}
Then,
\begin{itemize}
\item[(i)] for each subfield $k'\subset\mathbb C$, we have a canonical embedding
\begin{eqnarray*}
ev(X/S):HT^p_{j,d}(X/S)(k')\hookrightarrow HL_{j,d}(X_{\mathbb C}/S_{\mathbb C})\otimes_{\mathbb Q}\mathbb Q_p, \\
w_s\mapsto (1/2i\pi)ev(X/S)(w_s):=ev(X_s)(w_s), \, s\in S_{k'}, 
\end{eqnarray*}
note that the image consists of logarithmic classes, hence
\begin{equation*}
p_{S_{\hat k_{\sigma_p}}}(HT^p_{j,d}(X/S))(k')\subset p_{S_{\mathbb C}}(HL_{j,d}(X_{\mathbb C}/S_{\mathbb C})),
\end{equation*}
where $p_{S_{\mathbb C}}:E^j_{DR}(X_{\mathbb C}/S_{\mathbb C})\to S_{\mathbb C}$ and
$p_{S_{\hat k_{\sigma_p}}}:E^j_{DR}(X_{\hat k_{\sigma_p}}/S_{\hat k_{\sigma_p}})\to S_{\hat k_{\sigma_p}}$ 
are the projections,
\item[(ii)] we have 
\begin{equation*}
(R^jf_*\mathbb Q_{p,X_{\bar k}^{et}}(d))^G=<HT^p_{j,d}(X/S)\cap F^dE^j_{DR}(X/S)>_{\mathbb Q_p} 
\subset E^j_{DR}(X_{\hat k_{\sigma_p}}/S_{\hat k_{\sigma_p}}),
\end{equation*}
where $<->_{\mathbb Q_p}$ denote the $\mathbb Q_p$ subvector bundle generated by $(-)$,
\item[(iii)] we have a canonical embedding in $\Shv(S_{\mathbb C})$
\begin{eqnarray*}
ev(X/S):\pi_{k/\mathbb C}(S)^*(R^jf_*\mathbb Q_{p,X_{\bar k}^{et}}(d))^G
\hookrightarrow HL_{j,d}(X_{\mathbb C}/S_{\mathbb C})\otimes_{\mathbb Q}\mathbb Q_p, \\
\alpha_s\mapsto (1/2i\pi)ev(X/S)(w(\alpha_s)):=ev(X_{k(s)})((1/2i\pi)w(\alpha_s)), \,
s'=\pi_{k/\mathbb C}^{-1}(s):k(s)\hookrightarrow\mathbb C, \, s\in S,
\end{eqnarray*}
where $\pi_{k/\mathbb C}(S):S_{\mathbb C}\to S$ is the projection.
\end{itemize}
\end{thm}

\begin{proof}
\noindent(i): Follows from proposition \ref{keypropC} and the equality, by (the proof of) proposition \ref{LatticeLog},
\begin{eqnarray*}
\mathbb H^j(X_{s,\mathbb C_p},
\Omega^{\bullet\geq d}_{\bar X^{\mathcal O}_{s,\hat k'_{\sigma_p}}}(\log D^{\mathcal O}_{s,\hat k'_{\sigma_p}})
\otimes_{O_{X_s^{\mathcal O}}}O\mathbb B_{st,X_{s,\hat k'_{\sigma_p}},\log D_{\hat k'_{\sigma_p}}})^{\phi_p,N,\hat G} 
=\mathbb H_{pet}^j(X_{s,\mathbb C_p},\Omega^{\bullet\geq l}_{X_{s,\mathbb C_p},\log,\mathcal O}\otimes\mathbb Z_p).
\end{eqnarray*}
where $(\bar X'_s,D_s)\to (\bar X_s,Z_s)$ is a desingularization, 
i.e. $\bar X'_s$ is smooth and $D_s\subset\bar X'_s$ is a normal crossing divisor,
after taking an embedding $\sigma'_p:k'\hookrightarrow\mathbb C_p$.

\noindent(ii):Follows from the isomorphism $(H^jf^{\alpha}_*\alpha(X^{\alpha}))$ (c.f. theorem \cite{ChinoisCrys}) 
and proposition \ref{keypropCp}(ii).

\noindent(iii):Follows from (i) and (ii).
\end{proof}


LAGA UMR CNRS 7539 \\
Universit\'e Paris 13, Sorbonne Paris Cit\'e, 99 av Jean-Baptiste Clement, \\
93430 Villetaneuse, France, \\
bouali@math.univ-paris13.fr

\end{document}